%% file: main.tex
\documentclass[a4paper,UKenglish,cleveref, autoref,thm-restate]{lipics-v2021} 

\nolinenumbers 
\hideLIPIcs

\input{preamble}

\input{macros}

\input{macros_davide}

\title{When Lawvere meets Peirce: an equational presentation of boolean hyperdoctrines}

\titlerunning{An equational presentation of boolean hyperdoctrines} %

\author{Filippo Bonchi}{University of Pisa, Italy}{}{}{}
\author{Alessandro Di Giorgio}{University College London, United Kingdom}{a.giorgio@ucl.ac.uk}{0000-0002-6428-6461}{}
\author{Davide Trotta}{University of Pisa, Italy}{}{}{}

\authorrunning{F. Bonchi, A. Di Giorgio and D. Trotta} %

\Copyright{Filippo Bonchi, Alessandro Di Giorgio and Davide Trotta} %



\category{} %

\relatedversion{} %

\EventEditors{John Q. Open and Joan R. Access}
\EventNoEds{2}
\EventLongTitle{42nd Conference on Very Important Topics (CVIT 2016)}
\EventShortTitle{CVIT 2016}
\EventAcronym{CVIT}
\EventYear{2016}
\EventDate{December 24--27, 2016}
\EventLocation{Little Whinging, United Kingdom}
\EventLogo{}
\SeriesVolume{42}
\ArticleNo{23}

\begin{document}

\maketitle

\begin{abstract}
Fo-bicategories are a categorification of  Peirce's calculus of relations. Notably, their laws provide a proof system for first-order logic that is both purely equational and complete. This paper illustrates a correspondence between fo-bicategories and Lawvere's hyperdoctrines. To streamline our proof, we introduce peircean bicategories, which offer a more succinct characterization of fo-bicategories.
\end{abstract}

\input{sections/intro}

\input{sections/fobicatNEW.tex}

\input{sections/hyperdoctrine2}

\input{sections/peircianbicategory}
\input{sections/alltogether}

\input{sections/equivalence}

\input{sections/conclusion}

\bibliography{references}

\appendix

\input{sections/appcb}

\input{sections/apphyp}

\input{sections/appadjunction}

\input{sections/apppb}

\input{sections/apprestrictingadjunction}
\input{sections/apptabulation}

\end{document}

%% file: preamble.tex
\usepackage{amsmath, amsthm,amssymb}
\usepackage{tikz}
\usepackage{xcolor}
\usepackage{graphicx}
\usepackage{booktabs}
\usepackage{pxfonts}
\usepackage{bm}
\usepackage{scalerel}
\usepackage{trimclip}
\usepackage{ifthen}
\usepackage{xparse}
\usepackage{float}
\usepackage[all]{xy}
\usepackage{mathtools}
\usepackage{stackengine}
\usepackage{mathpartir}
\usepackage{bbm}
\usepackage{multicol}
\usepackage{subcaption}
\usepackage{prooftree}
\usepackage{realboxes}
\usepackage{mathtools}
\usepackage[inline]{enumitem}
\usepackage{makecell}

\usepackage{quiver}

\usepackage{todonotes}

\theoremstyle{definition} %

%% file: macros_davide.tex
\newcommand{\angbr}[2]{\langle #1,#2 \rangle}

\newcommand{\freccia}[3]{#2\colon#1 \to #3}

\newcommand{\frecciainj}[3]{\xymatrix{#2 \colon #1  \ar@{^{(}->}[r] &  #3}}

\newcommand{\duefreccia}[3]{\xymatrix@C=0.5cm{#2 \colon #1  \ar@{=>}[r] &  #3}}

\newcommand{\duemorfismo}[6]{\xymatrix@+1pc{
#1^{\op} \ar[rrd]^#2_{}="a" \ar[dd]_{#3^{\op}}\\
&& \infsl\\
#5^{\op}  \ar[rru]_#6^{}="b"
\ar_{#4}  "a";"b"}}
\newcommand{\comsquare}[8]{ \xymatrix@+1pc{ 
#1 \ar[r]^{#5} \ar[d]_{#6} & #2 \ar[d]^{#7} \\
#3 \ar[r]_{#8} & #4 
}}
\newcommand{\pullback}[8]{ \xymatrix@+1pc{ 
#1 \pullbackcorner \ar[r]^{#5} \ar[d]_{#6} & #2 \ar[d]^{#7} \\
#3 \ar[r]_{#8} & #4 
}}
\newcommand{\quadratocomm}[8]{ \xymatrix@+1pc{ 
#1 \ar[r]^{#5} \ar[d]_{#6} & #2 \ar[d]^{#7} \\
#3 \ar[r]_{#8} & #4 
}}
\newcommand{\comsquarelargo}[8]{ \xymatrix@+1pc{ 
#1 \ar[rr]^{#5} \ar[d]_{#6} && #2 \ar[d]^{#7} \\
#3 \ar[rr]_{#8} && #4 
}}
\newcommand{\parallelmorphisms}[4]{\xymatrix@+1pc{
#1 \ar @<+4pt>[r]^{#2} \ar @<-4pt>[r]_{#3} & #4
}}
\newcommand{\relation}[4]{\xymatrix@+1pc{
\angbr{#2}{#3}\colon #1 \ar @<+4pt>[r] \ar @<-4pt>[r] & #4
}}
\newcommand{\frecceparalleleopposte}[4]{\xymatrix@+1pc{
#1 \ar@<+4pt>[r]^{#2} \ar@<-4pt>@{<-}[r]_{#3} & #4
}}
\newcommand{\equalizer}[6]{\xymatrix@+1pc{
#1 \ar[r]^{#2} & #3 \ar @<+4pt>[r]^{#4} \ar @<-4pt>[r]_{#5} & #6
}}
\newcommand{\coequalizer}[6]{\xymatrix@+1pc{
 #1 \ar @<+4pt>[r]^{#2} \ar @<-4pt>[r]_{#3} & #4 \ar[r]^{#5} & #6
}}

\newcommand{\sottoggetto}[2]{\xymatrix{
#1 \ar@{>->}[r] & #2
}}

\newcommand{\pullbackcorner}[1][ul]{\save*!/#1+1.2pc/#1:(1,-1)@^{|-}\restore}

\def\pr{\pi}

\def\mV{\mathcal{V}}

\def\Sub{\mathsf{Sub}}

\def\infsl{\mathsf{InfSl}}

\def\set{\mathsf{Set}}

\newcommand{\doctrine}[2]{#2\colon #1^{\op}\longrightarrow\infsl}
\newcommand{\hyperdoctrine}[2]{#2\colon #1^{\op}\longrightarrow\bool}
\def\bool{\mathsf{Bool}}

\newcommand{\comp}[1]{\{ #1 \}}

\newcommand{\EED}{\mathbb{EED}}

\newcommand{\CartBic}{\mathbb{CB}}
\newcommand{\FOBic}{\mathbb{FOB}}
\newcommand{\PeirceBic}{\mathbb{PB}}

%% file: sections/intro.tex
\section{Introduction}
The first appearences of the characteristic features of first-order logic  can be traced back to the works of Peirce~\cite{peirce1897_the-logic-of-relatives} and Frege~\cite{frege1977begriffsschrift}. Frege was mainly motivated by the pursuit of a rigorous foundation for mathematics: his work was inspired by real analysis, bringing the concept of functions and variables into the logical realm~\cite{ewald2018emergence}. On the other hand Peirce, inspired by the work of De Morgan~\cite{de1860syllogism} on relational reasoning, 
introduced a calculus in which operations allow the combination of relations and adhere to a set of algebraic laws. Like Boole's algebra of classes~\cite{boole1847mathematical}, Peirce's calculus of relations does not feature variables nor quantifiers and its sole deduction rule is substituting equals by equals.

Despite several negative results regarding axiomatizations for the entire calculus~\cite{monk} and various fragments thereof~\cite{hodkinson2000axiomatizability,redko1964defining,freyd1990categories,doumane2020non,DBLP:conf/stacs/Pous18}, its lack of binder-related complexities, coupled with purely equational proofs, has rendered the calculus of relations highly influential in computer science, e.g., in the context of database theory~\cite{codd1983relational}, programming languages~\cite{pratt1976semantical,hoare1986weakest,lassen1998relational,bird1996algebra,DBLP:journals/pacmpl/LagoG22a} and proof assistants~\cite{pous2013kleene, pous2016automata, krauss2012regular}.

In logic, the calculus played a secondary role for many years, likely because it is strictly less expressive than first-order logic~\cite{lowenheim1915moglichkeiten}. This was until Tarski in~\cite{tarski1941calculus} recognized its algebraic flavour and initiated a program of algebraizing first-order logic, including works such as~\cite{everett1946projective,henkin1971cylindric,QUINE1971309}. Quoting Quine \cite{QUINE1971309}:
\begin{center}``Logic in his adolescent phase was algebraic. There was Boole's algebra of classes and Peirce's algebra of relations. But in 1879 logic come of age, with Frege's quantification theory. Here the bound variables, so characteristic of analysis rather than of algebra, became central to logic.''\end{center} 
Such a perspective, which regarded algebraic aspects and those concerning quantifiers as separate entities, changed with the work of Lawvere.

Thanks to the recent development of a new branch of mathematics, namely category theory,  Lawvere introduced  in~\cite{AF,DACCC,EHCSAF} \emph{hyperdoctrines} which enabled the study of
logic from a pure algebraic perspective. The crucial insights of Lawvere was to show that quantifiers, as well as many logical constructs, can be algebraically captured through the crucial notion of adjointness. Hyperdoctrines, along with many categorical structures related to logics, such as regular, Heyting, and boolean categories~\cite{johnstone2014topos,SAE}, align with Frege's functional perspective: arrows represent functions (terms), and relations are derived through specific constructions.

\smallskip

In the last decade, the paradigm shift towards treating data as a physical resource %
has motivated many computer scientists to move from traditional term-based (cartesian) syntax toward a string diagrammatic (monoidal) syntax~\cite{joyal1991geometry,Selinger2009} (see e.g.,~\cite{stein2023probabilistic,DBLP:journals/jacm/BonchiGKSZ22,DBLP:journals/pacmpl/BonchiHPSZ19,Bonchi2015,CoeckeDuncanZX2011,Fong2015,DBLP:journals/corr/abs-2009-06836,Ghica2016,DBLP:conf/lics/MuroyaCG18,Piedeleu2021}). This shift in syntax enables an extension of Peirce's calculus of relations that is as expressive as first-order logic, accompanied by an axiomatization that is purely equational and complete. The axioms are those of \emph{first-order bicategories} \cite{bonchi2024diagrammatic}: see Figures \ref{fig:cb axioms}, \ref{fig:closed lin axioms} and \ref{fig:fo bicat axioms}. In essence, a first-order bicategory, or fo-bicategory, encompasses a cartesian and a cocartesian bicategory \cite{carboni1987cartesian}, interacting as a linear bicategory \cite{cockett2000introduction}, while additionally satisfying linear versions of Frobenius equations and adjointness conditions.

\medskip
In this paper, we reconcile Lawvere's understanding of logic with Peirce's calculus of relations by illustrating a formal correspondence between  boolean hyperdoctrines and first-order bicategories.

\medskip

To reach such a correspondence, we found convenient to introduce the novel notion of \emph{peircean bicategories}: these are cartesian bicategories with each homset carrying a boolean algebra where the negation behaves appropriately with \emph{maps} -- special arrows that intuitively generalize functions. Our first result (\cref{thm_equiv_FOBic_PeirceBic}) establishes that peircean bicategories are equivalent to first-order bicategories.

While the definition of peircean bicategories is not purely equational, as in the case of fo-bicategories, it is notably more concise. Moreover, it allows us  to reuse from \cite{bonchi2021doctrines} an adjunction between cartesian bicategories and \emph{elementary and existential doctrines} \cite{QCFF,EQC,UEC}, which are a generalisation of hyperdoctrines, corresponding to the regular (i.e., $\exists,=,\top, \wedge$) fragment of first-order logic. %

Our main result (\cref{thm:main}) reveals an adjunction between the category of first-order bicategories and the category of boolean hyperdoctrines. One can perceive this as a logical analogue of Fox's theorem \cite{fox1976coalgebras}, which establishes an equivalence between categories with finite products and monoidal categories equipped with natural comonoids. The latter notion, like the one of fo-bicategories, is purely equational.

It is essential to note that our theorem establishes an adjunction rather than an equivalence. The discrepancy  can be intuitively explained by observing that, akin to first-order logic, terms and formulas are distinct entities in hyperdoctrines. This differentiation does not exist in the calculus of relations or first-order bicategories. For instance, given two terms $t_1$ and $t_2$, the hyperdoctrine where the formula $t_1=t_2$ is true differs from the hyperdoctrine where $t_1$ and $t_2$ are equated as terms, a distinction not present in fo-bicategories. These issues, related to the extensionality of equality, are thoroughly analyzed in the context of doctrines in \cite{EQC} and, more generally, in that of fibrations in \cite{CLTT}.

Leveraging another result from~\cite{bonchi2021doctrines}, we demonstrate (\cref{thm:theequivalence}) that the adjunction in~\cref{thm:main} becomes an equivalence when restricted to well-behaved hyperdoctrines (i.e., those whose equality is extentional and satisfying the rule of unique choice \cite{TECH}). Finally, combining this finding with a result in~\cite{TECH}, we illustrate (Corollary \ref{cor:final}) that \emph{functionally complete}~\cite{carboni1987cartesian} first-order bicategories are equivalent to boolean categories~\cite{SAE}.

\emph{Synopsis:}  In \S~\ref{sec:background}, we provide a review of (co)cartesian bicategories, linear bicategories, and fo-bicategories. \S~\ref{sec:hyp} covers a recap of elementary and existential doctrines and boolean hyperdoctrines. The adjunction between cartesian bicategories and doctrines, as detailed in~\cite{bonchi2021doctrines}, is presented in \S \ref{sec:adjunction}.
Our original contributions commence in \S~\ref{sec:pb}, where we introduce peircean bicategories and establish their equivalence with fo-bicategories. This result is further used in \S~\ref{sec:restrictingadjunction} to demonstrate the adjunction and in \S~\ref{sec:equivalence} to establish the equivalence. \S~\ref{sec:tabulation} elucidates the correspondence with boolean categories. All proofs are in appendix.

\emph{Terminology and Notation:} %
All bicategories considered in this paper are just poset-enriched symmetric monoidal categories. For a bicategory $\Cat{C}$, we will write $\opposite{\Cat{C}}$ for the bicategory having the same objects as $\Cat{C}$ but homsets $\opposite{\Cat{C}}[X,Y] \defeq \Cat{C}[Y,X]$. Similarly, we will write  $\co{\Cat{C}}$ to denote the bicategory having the same objects and arrows of $\Cat{C}$ but equipped with the reversed ordering $\geq$. The cartesian bicategories in this paper are called in~\cite{carboni1987cartesian} cartesian bicategories of relations. We refer the reader to~\cite[Rem. 2]{bonchi2024diagrammatic} for a comparison with the presentation of linear bicategories in~\cite{cockett2000introduction}.

%% file: sections/fobicatNEW.tex
\section{From (Co)Cartesian to  First-Order Bicategories}\label{sec:background}

In this section we recall the notion of \emph{first-order bicategory} from \cite{bonchi2024diagrammatic}. To provide a preliminary intuition, it is convenient to consider $\Rel$, the first-order bicategory of sets and relations. 

It is well known that sets and relations form a symmetric monoidal category, hereafter denoted as $\Relp$, with composition, identities, monoidal product and symmetries defined as
\begin{equation}\label{eq:whiteRel}
\begin{array}{c}
\begin{array}{cc}
 \seq[+][a][b] \;\defeq \; \{(x,z) \mid \exists y\!\in\! Y \,.\, (x,y)\in a  \wedge (y,z)\in b\}\subseteq X\times Z  & \id[+][X] \;\defeq \; \{(x,y) \!\mid\! x=y\}\!\subseteq\! X \times X 
\end{array}
\\
\begin{array}{rcl}
 a \tensor[+]c & \defeq & \{ ( \,(x,z), (y,v) \,) \mid (x,y)\in a \wedge (z,v) \in c \} \subseteq (X \times Z) \times (Y \times V) \\
 \symm[+][X][Y] &\defeq & \{(\;(x,y), (y',x')\;) \mid x=x' \wedge \,y=y' \}\subseteq (X\times Y) \times (Y\times X)
\end{array}
\end{array}
\end{equation}
for all sets $X,Y,Z,V$ and relations $a\subseteq X\times Y$, $b\subseteq Y\times Z$ and $c\subseteq Z\times V$. As originally observed by Peirce in~\cite{peirce1883_studies-in-logic.-by-members-of-the-johns-hopkins-university}, beyond $\seq[+]$ there exists another form of relational composition that enjoys noteworthy algebraic properties. This different composition gives rise to another symmetric monoidal category of sets and relations, hereafter denoted by $\Relm$ and defined as follows.
\begin{equation}\label{eq:blackRel}
\begin{array}{c}
\begin{array}{cc}
 \seq[-][a][b] \;\defeq \; \{(x,z) \mid \forall y\!\in\! Y \,.\, (x,y)\in a  \vee (y,z)\in b\}\subseteq X\times Z  & \id[-][X] \;\defeq \; \{(x,y) \!\mid\! x\neq y\}\!\subseteq\! X \times X 
\end{array}
\\
\begin{array}{rcl}
 a \tensor[-]c & \defeq & \{ ( \,(x,z), (y,v) \,) \mid (x,y)\in a \vee (z,v) \in c \} \subseteq (X \times Z) \times (Y \times V) \\
 \symm[-][X][Y] &\defeq & \{(\;(x,y), (y',x')\;) \mid x\neq x' \vee \,y\neq y' \}\subseteq (X\times Y) \times (Y\times X)
\end{array}
\end{array}
\end{equation}
Note that $\tensor[+]$ and $\tensor[-]$ are both defined on objects as the cartesian product of sets and have as unit the singleton set $\unittensor \defeq \{ \star \}$.
Both $\Relp$ and $\Relm$ are poset-enriched symmetric monoidal categories when taking as ordering the inclusion $\subseteq$ and the complement $\nega{} \colon \co{(\Relp)} \to \Relm$ is an isomorphism. As we will explain in \S~\ref{sec:cartesianbi}, the relations defined for all sets $X$ as %
\begin{equation}\label{eq:comonoidsREL}
	\hspace*{-1em}
		\begin{tabular}{rcl rcl}
			$\copier[+][X]$   & $\!\!\!\!\!\defeq\!\!\!\!\!$ & $\{(x, \; (y,z)) \mid x=y \wedge x=z\} \subseteq X \times (X \times X)$  	    & $\copier[-][X] $   & $\!\!\!\!\!\defeq\!\!\!\!\!$ & $\{(x, \; (y,z)) \mid x\neq y \vee x \neq z\} \subseteq X \times (X \times X)$     \\
			$\cocopier[+][X]$ & $\!\!\!\!\!\defeq\!\!\!\!\!$ & $\{((y,z),\; x) \mid x=y \wedge x=z\} \subseteq  (X \times X) \times X$         &  $\cocopier[-][X]$ & $\!\!\!\!\!\defeq\!\!\!\!\!$ & $\{((y,z), \; x  ) \mid x\neq y \vee x \neq z\} \subseteq  (X \times X) \times X$\\
			$\discard[+][X]$   & $\!\!\!\!\!\defeq\!\!\!\!\!$ & $\{(x, \star) \mid x\in X\} \subseteq X \times I$&  $\discard[-][X]$   & $\!\!\!\!\!\defeq\!\!\!\!\!$ & $\varnothing \subseteq X \times I$ \\
			$\codiscard[+][X]$ & $\!\!\!\!\!\defeq\!\!\!\!\!$ &  $\{(\star,x) \mid x\in X\}\subseteq \unittensor \times X$  &  $\codiscard[-][X]$ & $\!\!\!\!\!\defeq\!\!\!\!\!$ & $\varnothing \subseteq \unittensor \times X$
		\end{tabular}
\end{equation}
make $\Relp$ a cartesian bicategory, while $\Relm$ a cocartesian one. 

Intuitively, a first-order bicategory $\Cat{C}$ consists of a cartesian bicategory $\Cat{C}^\circ$, called the ``white structure'', and a cocartesian bicategory $\Cat{C}^\bullet$, called the ``black structure'', that interact by obeying the same laws of $\Relp$ and $\Relm$. The name ``first-order'' is due to the fact that such laws provide a complete system of axioms for first-order logic.

\smallskip

The axioms can be conveniently given by means of a graphical representation inspired by string diagrams \cite{joyal1991geometry,Selinger2009}: composition is depicted as horizontal composition while the monoidal product by vertically ``stacking'' diagrams. However, since there are two compositions  $\seq[+]$ and $\seq[-]$ and two monoidal products $\tensor[+]$ and $\tensor[-]$,  to distinguish them we use different colors. %
All white constants have white background, mutatis mutandis for the black ones: for instance $\copier[+][X]$ and $\cocopier[-][X]$ are drawn
$\copierCirc[+][X]$ and $\cocopierCirc[-][X]$, while for some arrows $a,b,c,d$ of the appropriate type, $(a\tensor[+]c) \seq[-] (b \tensor[-]d)$ is drawn as on the right of \eqref{ax:linStrn1} in \cref{fig:closed lin axioms}. %

\subsection{(Co)Cartesian Bicategories}\label{sec:cartesianbi}
We commence with the notion of cartesian bicategories by Carboni and Walters \cite{carboni1987cartesian}.

\begin{definition}\label{def:cartesian bicategory}
    A \emph{cartesian bicategory} $(\Cat{C}, \tensor[+], \unittensor, \copier[+], \discard[+], \cocopier[+], \codiscard[+])$, shorthand $(\Cat{C}, \copier[+], \cocopier[+])$,
    is a poset-enriched symmetric monoidal category $(\Cat{C}, \tensor[+], \unittensor)$ and, for every object $X$ in $\Cat{C}$,
    arrows $\copier[+][X]\colon X \to X\tensor[+]X$, $\discard[+][X]\colon X \to \unittensor$, $\cocopier[+][X] \colon X \tensor[+]X \to X$,  $\codiscard[+][X]\colon \unittensor \to X$ such that
     \begin{enumerate}
     \item $(\copier[+][X], \discard[+][X])$ is a comonoid and $(\cocopier[+][X], \codiscard[+][X])$ a monoid, i.e., the equalities \eqref{ax:comPlusAssoc}, \eqref{ax:comPlusUnit},  \eqref{ax:comPlusComm} and \eqref{ax:monPlusAssoc}, \eqref{ax:monPlusUnit}, \eqref{ax:monPlusComm} in Figure~\ref{fig:cb axioms} hold;
     \item every arrow $c \colon X \to Y$ is a lax comonoid homomorphism, i.e., \eqref{ax:comPlusLaxNat} and \eqref{ax:discPlusLaxNat} hold;
    
    \item comonoids are left adjoints to the monoids, i.e., %
    \eqref{ax:plusCopyCocopy}, \eqref{ax:plusCocopyCopy}, \eqref{ax:plusDiscCodisc} and \eqref{ax:plusCodiscDisc} hold;
    \item monoids and comonoids form special Frobenius bimonoids, i.e., \eqref{ax:plusFrob} and \eqref{ax:plusSpecFrob} hold;
    \item monoids and comonoids satisfy the expected coherence conditions (see e.g. \cite{bonchi2021doctrines}).
    \end{enumerate}
    $\Cat{C}$ is a \emph{cocartesian bicategory} if  $\co{\Cat{C}}$ is a cartesian bicategory. %
    A \emph{morphism of (co)cartesian bicategories} is a poset-enriched strong symmetric monoidal functor preserving monoids and comonoids.  We denote by $\CartBic$ the category of cartesian bicategories and their morphisms.
\end{definition}
As already mentioned, $\Relp$ with $\copier[+][X]$, $\discard[+][X]$, $\cocopier[+][X]$ and $\codiscard[+][X]$ defined in \eqref{eq:comonoidsREL} form a cartesian bicategory: the reader can easily check, using the definitions in \eqref{eq:whiteRel} and \eqref{eq:comonoidsREL}, that all the laws in \cref{fig:cb axioms} are satisfied. Similarly, one can observe that the opposite inequality of \eqref{ax:comPlusLaxNat} holds iff the relation $c\subseteq X\times Y$ is single-valued (i.e., deterministic), while the opposite of \eqref{ax:discPlusLaxNat} iff $c$ is total. In other words, $c$ is a function iff both \eqref{ax:comPlusLaxNat} and \eqref{ax:discPlusLaxNat} holds as equalities.

\begin{figure}[t]
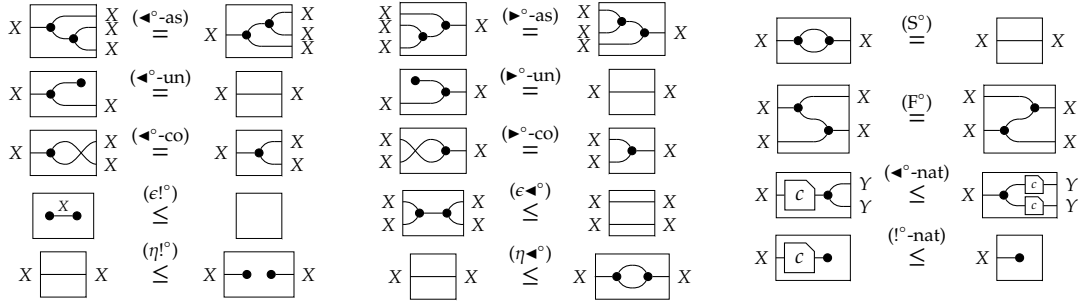

    \mylabel{ax:comPlusAssoc}{\ensuremath{\copier[+]}\text{-as}}
    \mylabel{ax:comPlusUnit}{\ensuremath{\copier[+]}\text{-un}}
    \mylabel{ax:comPlusComm}{\ensuremath{\copier[+]}\text{-co}}
    \mylabel{ax:monPlusAssoc}{\ensuremath{\cocopier[+]}\text{-as}}
    \mylabel{ax:monPlusUnit}{\ensuremath{\cocopier[+]}\text{-un}}
    \mylabel{ax:monPlusComm}{\ensuremath{\cocopier[+]}\text{-co}}
    \mylabel{ax:plusSpecFrob}{\text{S}\ensuremath{^\circ}}
    \mylabel{ax:plusFrob}{\text{F}\ensuremath{^\circ}}
    \mylabel{ax:comPlusLaxNat}{\ensuremath{\copier[+]}\text{-nat}}
    \mylabel{ax:discPlusLaxNat}{\ensuremath{\discard[+]}\text{-nat}}
    \mylabel{ax:plusCodiscDisc}{\ensuremath{\epsilon\discard[+]}}
    \mylabel{ax:plusDiscCodisc}{\ensuremath{\eta\discard[+]}}
    \mylabel{ax:plusCocopyCopy}{\ensuremath{\epsilon\copier[+]}}
    \mylabel{ax:plusCopyCocopy}{\ensuremath{\eta\copier[+]}}
    \[
            \begin{array}{@{}c c c c c@{}}
            \begin{array}{@{}c@{}c@{}c@{}}
                    
    \InputIfFileExists{axiomsNEW/cb/plus/comAssoc1.tikz}{}{\input{tikz/axiomsNEW/cb/plus/comAssoc1.tikz}}
  & \Leq{\ref*{ax:comPlusAssoc}}    & 
    \InputIfFileExists{axiomsNEW/cb/plus/comAssoc2.tikz}{}{\input{tikz/axiomsNEW/cb/plus/comAssoc2.tikz}}
 \\
                    
    \InputIfFileExists{axiomsNEW/cb/plus/comUnit.tikz}{}{\input{tikz/axiomsNEW/cb/plus/comUnit.tikz}}
    & \Leq{\ref*{ax:comPlusUnit}}     & \idCirc[+][X] \\
                    
    \InputIfFileExists{axiomsNEW/cb/plus/comComm.tikz}{}{\input{tikz/axiomsNEW/cb/plus/comComm.tikz}}
    & \Leq{\ref*{ax:comPlusComm}}     & \copierCirc[+][X] \\
                    
    \InputIfFileExists{axiomsNEW/cb/plus/codiscDisc.tikz}{}{\input{tikz/axiomsNEW/cb/plus/codiscDisc.tikz}}
 & \Lleq{\ref*{ax:plusCodiscDisc}} & \emptyCirc[+] \\
                    \idCirc[+][X]                          & \Lleq{\ref*{ax:plusDiscCodisc}} & 
    \InputIfFileExists{axiomsNEW/top.tikz}{}{\input{tikz/axiomsNEW/top.tikz}}
 \\
            \end{array} & \!\!\!\! &
            \begin{array}{@{}c@{}c@{}c@{}}
                    
    \InputIfFileExists{axiomsNEW/cb/plus/monAssoc1.tikz}{}{\input{tikz/axiomsNEW/cb/plus/monAssoc1.tikz}}
      & \Leq{\ref*{ax:monPlusAssoc}}    & 
    \InputIfFileExists{axiomsNEW/cb/plus/monAssoc2.tikz}{}{\input{tikz/axiomsNEW/cb/plus/monAssoc2.tikz}}
 \\
                    
    \InputIfFileExists{axiomsNEW/cb/plus/monUnit.tikz}{}{\input{tikz/axiomsNEW/cb/plus/monUnit.tikz}}
        & \Leq{\ref*{ax:monPlusUnit}}     & \idCirc[+][X] \\
                    
    \InputIfFileExists{axiomsNEW/cb/plus/monComm.tikz}{}{\input{tikz/axiomsNEW/cb/plus/monComm.tikz}}
        & \Leq{\ref*{ax:monPlusComm}}     & \cocopierCirc[+][X] \\
                    
    \InputIfFileExists{axiomsNEW/cb/plus/cocopierCopier.tikz}{}{\input{tikz/axiomsNEW/cb/plus/cocopierCopier.tikz}}
 & \Lleq{\ref*{ax:plusCocopyCopy}} & 
    \InputIfFileExists{axiomsNEW/id2P.tikz}{}{\input{tikz/axiomsNEW/id2P.tikz}}
 \\
                    \idCirc[+][X]                              & \Lleq{\ref*{ax:plusCopyCocopy}} & 
    \InputIfFileExists{axiomsNEW/cb/plus/specFrob.tikz}{}{\input{tikz/axiomsNEW/cb/plus/specFrob.tikz}}

            \end{array} & \!\!\!\! &
            \begin{array}{@{}c@{}c@{}c@{}}
                    
    \InputIfFileExists{axiomsNEW/cb/plus/specFrob.tikz}{}{\input{tikz/axiomsNEW/cb/plus/specFrob.tikz}}
      & \Leq{\ref*{ax:plusSpecFrob}}    & \idCirc[+][X] \\[8pt]
                    
    \InputIfFileExists{axiomsNEW/cb/plus/frob1.tikz}{}{\input{tikz/axiomsNEW/cb/plus/frob1.tikz}}
         & \Leq{\ref*{ax:plusFrob}}        & 
    \InputIfFileExists{axiomsNEW/cb/plus/frob2.tikz}{}{\input{tikz/axiomsNEW/cb/plus/frob2.tikz}}
 \\[10pt]
                    
    \InputIfFileExists{axiomsNEW/cb/plus/copierLaxNat1.tikz}{}{\input{tikz/axiomsNEW/cb/plus/copierLaxNat1.tikz}}
 & \Lleq{\ref*{ax:comPlusLaxNat}}  & 
    \InputIfFileExists{axiomsNEW/cb/plus/copierLaxNat2.tikz}{}{\input{tikz/axiomsNEW/cb/plus/copierLaxNat2.tikz}}
 \\
                    
    \InputIfFileExists{axiomsNEW/cb/plus/discardLaxNat.tikz}{}{\input{tikz/axiomsNEW/cb/plus/discardLaxNat.tikz}}
 & \Lleq{\ref*{ax:discPlusLaxNat}} & \discardCirc[+][X][X]
            \end{array}
            \end{array}
    \]
    \caption{Axioms of Cartesian bicategories}\label{fig:cb axioms}
\end{figure}

\begin{definition}\label{def:maps} Let $c\colon X \to Y$ be an arrow of a cartesian bicategory $\Cat{C}$. It is a \emph{map} if 
    \begin{equation}\label{eq:def:map} 
    \InputIfFileExists{axiomsNEW/cb/plus/copierLaxNat1.tikz}{}{\input{tikz/axiomsNEW/cb/plus/copierLaxNat1.tikz}}
 \geq 
    \InputIfFileExists{axiomsNEW/cb/plus/copierLaxNat2.tikz}{}{\input{tikz/axiomsNEW/cb/plus/copierLaxNat2.tikz}}
 \qquad \text{ and }\qquad 
    \InputIfFileExists{axiomsNEW/cb/plus/discardLaxNat.tikz}{}{\input{tikz/axiomsNEW/cb/plus/discardLaxNat.tikz}}
 \!\!\!\geq \discardCirc[+][X] \text{.} \end{equation}
\end{definition}
It is easy to see that maps form a monoidal subcategory of $\Cat{C}$~\cite{carboni1987cartesian}, hereafter denoted by $\map{(\Cat{C})}$. Since, by \eqref{ax:comPlusLaxNat}, \eqref{ax:discPlusLaxNat} and \eqref{eq:def:map}, comonoids are natural w.r.t. maps, Fox theorem \cite{fox1976coalgebras} guarantees that $\map{(\Cat{C})}$ is a category with finite products.

In a cartesian bicategory $\Cat{C}$, each homset $\Cat{C}[X,Y]$ carries the structure of inf-semilattice, defined for all $c,d\colon X \to Y$ as in \eqref{eq:def:cap} below. %
Furthermore, the equation \eqref{eqdagger} defines an identity-on-objects isomorphism of cartesian bicategories $\op{(\cdot)}\colon \Cat{C} \to \opposite{\Cat{C}}$.

\noindent\begin{minipage}{0.6\textwidth}
 \begin{equation}\label{eq:def:cap}c \wedge d \defeq \intersectionCirc{c}{d}[X][Y] \qquad  \top \defeq \topCirc[X][Y]
    \end{equation}
\end{minipage}
\begin{minipage}{0.4\textwidth}
\begin{equation}\label{eqdagger}
\op{c} \defeq \daggerCirc[+]{c}[Y][X]
\end{equation}
\end{minipage}
The reader can check, using \eqref{eq:whiteRel} and \eqref{eq:comonoidsREL} that in $\Relp$, $\op{c}\colon Y \to X$ is the opposite of the relation $c$, namely $\{(y,x)\mid (x,y)\in c\}$.
It is well known that a relation $c$ is a function iff it is left adjoint to $\op{c}$. More generally in a cartesian bicategory $c$ is a map iff it is left adjoint to $\op{c}$. Summarising:
\begin{proposition}\label{lemma:meet semilattice}\label{prop:map adj}
Let $\Cat{C}$ be a cartesian bicategory and $c\colon X\to Y$ an arrow of $\Cat{C}$. The following hold:
\begin{enumerate}
\item every homset carries the inf-semilattice structure, defined as in \eqref{eq:def:cap};
\item there is an isomorphism of cartesian bicategories $\op{(\cdot)}\colon \Cat{C} \to \opposite{\Cat{C}}$, defined as in \eqref{eqdagger};
\item $c$ is a map iff $c$ is left adjoint to $\op{c}$;
\item $\map{(\Cat{C})}$ is a category with finite products; moreover, a morphism of cartesian bicategories $F\colon \Cat{C} \to \Cat{D}$ restricts to a functor $\tilde{F}\colon \map{(\Cat{C})} \to \map{(\Cat{D})}$ preserving finite products.
\end{enumerate}
\end{proposition}
Hereafter, we draw $\boxOpCirc[+]{c}[X][Y]$ for $\op{(\boxCirc[+]{c}[X][Y])}$ and $\funcCirc[+]{c}[X][Y]$ for a map $c \colon X \to Y$.

We mentioned that $\Relm$ with $\copier[-][X]$, $\discard[-][X]$, $\cocopier[-][X]$ and $\codiscard[-][X]$ defined in \eqref{eq:comonoidsREL} forms a cocartesian bicategory. To prove this, it is enough to observe that the complement $\nega{}$ is a poset-enriched symmetric monoidal isomorphism $\nega{} \colon \co{(\Relp)} \to \Relm$ preserving (co)monoids.

\subsection{Linear Bicategories}\label{sec:linbic}
\begin{figure}[t]
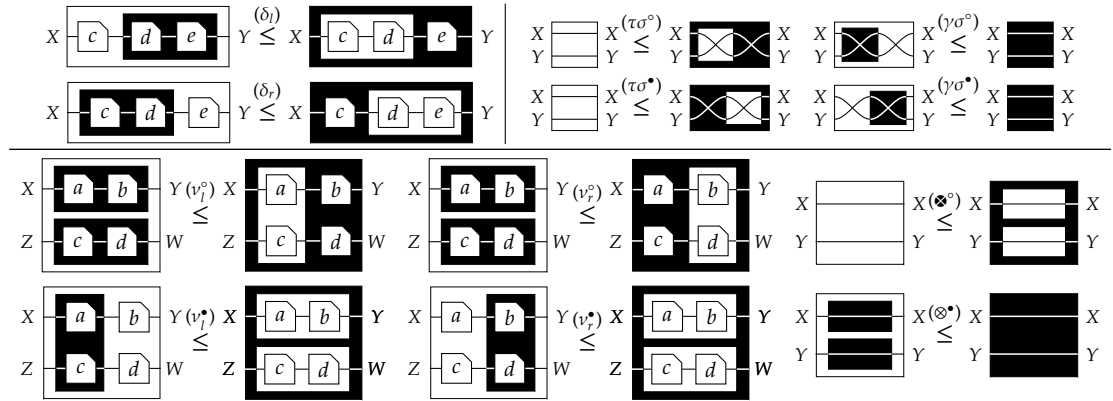

    \mylabel{ax:leftLinDistr}{\ensuremath{\delta_l}}
    \mylabel{ax:rightLinDistr}{\ensuremath{\delta_r}}
    \mylabel{ax:linStrn1}{\ensuremath{\nu^\circ_l}}
    \mylabel{ax:linStrn2}{\ensuremath{\nu^\circ_r}}
    \mylabel{ax:linStrn3}{\ensuremath{\nu^\bullet_l}}
    \mylabel{ax:linStrn4}{\ensuremath{\nu^\bullet_r}}
    \mylabel{ax:tensorPlusIdMinus}{\ensuremath{\tensor[+]^\bullet}}
    \mylabel{ax:tensorMinusIdPlus}{\ensuremath{\tensor[-]^\circ}}
    \mylabel{ax:tauSymmPlus}{\ensuremath{\tau\symm[+]}}
    \mylabel{ax:tauRPlus}{\ensuremath{\tau R^\circ}}
    \mylabel{ax:gammaSymmPlus}{\ensuremath{\gamma\symm[+]}}
    \mylabel{ax:gammaRPlus}{\ensuremath{\gamma R^\circ}}
    \mylabel{ax:tauSymmMinus}{\ensuremath{\tau\symm[-]}}
    \mylabel{ax:tauRMinus}{\ensuremath{\tau R^\bullet}}
    \mylabel{ax:gammaSymmMinus}{\ensuremath{\gamma\symm[-]}}
    \mylabel{ax:gammaRMinus}{\ensuremath{\gamma R^\bullet}}
    \[
        \begin{array}{@{}c@{}c@{}c@{}}
            \multicolumn{3}{c}{
                \begin{array}{@{}c@{}|c@{}}
                    \begin{array}{@{}c@{}c@{}c@{}}
                        
    \InputIfFileExists{axiomsNEW/leftLinDistr1.tikz}{}{\input{tikz/axiomsNEW/leftLinDistr1.tikz}}
  & \!\!\Lleq{\ref*{ax:leftLinDistr}}\!\!  & 
    \InputIfFileExists{axiomsNEW/leftLinDistr2.tikz}{}{\input{tikz/axiomsNEW/leftLinDistr2.tikz}}
 \\[10pt]
                        
    \InputIfFileExists{axiomsNEW/rightLinDistr1.tikz}{}{\input{tikz/axiomsNEW/rightLinDistr1.tikz}}
 & \!\!\Lleq{\ref*{ax:rightLinDistr}}\!\! & 
    \InputIfFileExists{axiomsNEW/rightLinDistr2.tikz}{}{\input{tikz/axiomsNEW/rightLinDistr2.tikz}}

                    \end{array}
                    &
                    \begin{array}{@{}c@{}c@{}c@{}c@{}c@{}c@{}}
                        
    \InputIfFileExists{axiomsNEW/idXYP.tikz}{}{\input{tikz/axiomsNEW/idXYP.tikz}}
 & \!\!\Lleq{\ref*{ax:tauSymmPlus}}\!\!   & 
    \InputIfFileExists{axiomsNEW/linadj/symMsym.tikz}{}{\input{tikz/axiomsNEW/linadj/symMsym.tikz}}
 & 
    \InputIfFileExists{axiomsNEW/linadj2/symPsym.tikz}{}{\input{tikz/axiomsNEW/linadj2/symPsym.tikz}}
     & \!\!\Lleq{\ref*{ax:gammaSymmPlus}}\!\!   & 
    \InputIfFileExists{axiomsNEW/idXYM.tikz}{}{\input{tikz/axiomsNEW/idXYM.tikz}}
 \\
                        
    \InputIfFileExists{axiomsNEW/idXYP.tikz}{}{\input{tikz/axiomsNEW/idXYP.tikz}}
 & \!\!\Lleq{\ref*{ax:tauSymmMinus}}\!\!   & 
    \InputIfFileExists{axiomsNEW/linadj2/symMsym.tikz}{}{\input{tikz/axiomsNEW/linadj2/symMsym.tikz}}
 & 
    \InputIfFileExists{axiomsNEW/linadj/symPsym.tikz}{}{\input{tikz/axiomsNEW/linadj/symPsym.tikz}}
     & \!\!\Lleq{\ref*{ax:gammaSymmMinus}}\!\!   & 
    \InputIfFileExists{axiomsNEW/idXYM.tikz}{}{\input{tikz/axiomsNEW/idXYM.tikz}}
                    \end{array}
                \end{array}
            } \\
            \midrule
                \begin{array}{@{} c @{} c @{} c @{}}
                    
    \InputIfFileExists{axiomsNEW/linStr1_1.tikz}{}{\input{tikz/axiomsNEW/linStr1_1.tikz}}
 & \!\!\Lleq{\ref*{ax:linStrn1}}\!\! & 
    \InputIfFileExists{axiomsNEW/linStr1_2.tikz}{}{\input{tikz/axiomsNEW/linStr1_2.tikz}}
 \\[20pt]
                    
    \InputIfFileExists{axiomsNEW/linStr3_1.tikz}{}{\input{tikz/axiomsNEW/linStr3_1.tikz}}
 & \!\!\Lleq{\ref*{ax:linStrn3}}\!\! & 
    \InputIfFileExists{axiomsNEW/linStr3_2.tikz}{}{\input{tikz/axiomsNEW/linStr3_2.tikz}}

                \end{array}
            &
                \begin{array}{@{} c @{} c @{} c @{}}
                    
    \InputIfFileExists{axiomsNEW/linStr1_1.tikz}{}{\input{tikz/axiomsNEW/linStr1_1.tikz}}
 & \!\!\Lleq{\ref*{ax:linStrn2}}\!\! & 
    \InputIfFileExists{axiomsNEW/linStr2_2.tikz}{}{\input{tikz/axiomsNEW/linStr2_2.tikz}}
 \\[20pt]
                    
    \InputIfFileExists{axiomsNEW/linStr4_1.tikz}{}{\input{tikz/axiomsNEW/linStr4_1.tikz}}
 & \!\!\Lleq{\ref*{ax:linStrn4}}\!\! & 
    \InputIfFileExists{axiomsNEW/linStr3_2.tikz}{}{\input{tikz/axiomsNEW/linStr3_2.tikz}}

                \end{array}
            &
            \begin{array}{@{} c @{} c @{} c @{}}
                
    \InputIfFileExists{axiomsNEW/id2Pbig.tikz}{}{\input{tikz/axiomsNEW/id2Pbig.tikz}}
   & \!\!\Lleq{\ref*{ax:tensorMinusIdPlus}}\!\!  & 
    \InputIfFileExists{axiomsNEW/idPMinusidP.tikz}{}{\input{tikz/axiomsNEW/idPMinusidP.tikz}}
 \\[20pt]
                
    \InputIfFileExists{axiomsNEW/idMPlusidM.tikz}{}{\input{tikz/axiomsNEW/idMPlusidM.tikz}}
 & \!\!\Lleq{\ref*{ax:tensorPlusIdMinus}}\!\! & 
    \InputIfFileExists{axiomsNEW/id2Mbig.tikz}{}{\input{tikz/axiomsNEW/id2Mbig.tikz}}

            \end{array}
        \end{array}
    \]
    \caption{Axioms of closed symmetric monoidal linear bicategories}\label{fig:closed lin axioms}
\end{figure}

We have seen that $\Relp$ forms a cartesian bicategory, and $\Relm$ a cocartesian bicategory. %
The next step consists in merging them into one entity and study their algebraic interactions. However, the coexistence of two different compositions $\seq[+]$ and $\seq[-]$ on the same class of objects and arrows brings us out of the realm of ordinary categories. The appropriate setting is provided by \emph{linear bicategories} \cite{cockett2000introduction}
by Cockett, Koslowski and Seely.

\begin{definition}\label{def:linear bicategory}
A \emph{linear bicategory} $(\Cat{C}, \seq[+], \id[+], \seq[-], \id[-])$ consists of two poset-enriched categories $(\Cat{C}, \seq[+], \id[+])$ and $(\Cat{C}, \seq[-], \id[-])$ with the same class of objects, arrows and orderings (but possibly different identities and compositions) such that
$\seq[+]$ linearly distributes over $\seq[-]$, i.e., \eqref{ax:leftLinDistr} and \eqref{ax:rightLinDistr} in Figure~\ref{fig:closed lin axioms} hold.

A \emph{symmetric monoidal linear bicategory} $(\Cat{C}, \seq[+], \id[+], \seq[-], \id[-], \tensor[+],\symm[+], \tensor[-], \symm[-], \unittensor)$, shortly  $(\Cat{C},\tensor[+], \tensor[-], \unittensor)$, consists of
a linear bicategory $(\Cat{C}, \seq[+], \id[-], \seq[-], \id[-])$ and  two poset-enriched symmetric monoidal categories $(\Cat{C}, \tensor[+],  \unittensor)$ and $(\Cat{C}, \tensor[-], \unittensor)$ s.t.\  $\tensor[+]$ and $\tensor[-]$ agree on objects, i.e., $X \tensor[+]Y= X\tensor[-]Y$, share the same unit $\unittensor$ and
\begin{enumerate}\setcounter{enumi}{1}
\item \label{eq:linearstr1}\label{eq:linearstr2} there are linear strengths for $(\tensor[+],\tensor[-])$, i.e., the inequalities \eqref{ax:linStrn1}, \eqref{ax:linStrn2}, \eqref{ax:linStrn3} and \eqref{ax:linStrn4}  hold;
\item  $\tensor[-]$  preserves $\id[+]$ colaxly and $\tensor[+]$  preserves $\id[-]$ laxly, i.e., %
\eqref{ax:tensorPlusIdMinus} and \eqref{ax:tensorMinusIdPlus} hold.
\end{enumerate}
A \emph{morphism of symmetric monoidal linear bicategories} $F\colon (\Cat{C_1}, \tensor[+], \tensor[-],  \unittensor) \to (\Cat{C_2},\tensor[+], \tensor[-], \unittensor)$ consists of two poset-enriched symmetric monoidal  functors $F^\circ \colon(\Cat{C_1},  \tensor[+], \unittensor) \to (\Cat{C_2}, \tensor[+],  \unittensor)$ and $F^\bullet \colon(\Cat{C_1}, \ \tensor[-], \unittensor) \to (\Cat{C_2},  \tensor[-],  \unittensor)$ that agree on objects and arrows: $F^{\circ} (X) = F^{\bullet}(X)$ and $F^{\circ} (c) = F^{\bullet}(c)$.%
\end{definition}

All linear bicategories in this paper are symmetric monoidal. Hence, we usually omit the adjective \emph{symmetric monoidal} and refer to them simply as linear bicategories.
In linear bicategories one can define \emph{linear} adjoints: for $a \colon X \to Y$ and $b \colon Y \to X$, $a$ is  \emph{left linear adjoint} to $b$, or $b$ is \emph{right linear adjoint} to $a$, written $b \Vdash a$, if $\id[+][X] \leq a \seq[-] b$  and $b \seq[+] a \leq \id[-][Y]$.

\begin{definition}
A linear bicategory $(\Cat{C},\tensor[+],\tensor[-],\unittensor)$ is said to be \emph{closed} if every $a\colon X \to Y$ has both a left and a right linear adjoint and, in particular, the white symmetry $\symm[+]$ is both left and right linear adjoint to the black symmetry $\symm[-]$ ($\symm[-] \Vdash \symm[+] \Vdash \symm[-]$), i.e.\ \eqref{ax:tauSymmPlus}, \eqref{ax:gammaSymmPlus}, \eqref{ax:tauSymmMinus} and \eqref{ax:gammaSymmMinus} in Figure~\ref{fig:closed lin axioms} hold.
\end{definition}

Our main example is the closed linear bicategory $\Rel$ of sets and relations. The white structure is the symmetric monoidal category $\Relp$ and the black structure is $\Relm$. Observe that the two have the same objects, arrows and ordering. The white and black monoidal products $\tensor[+]$ and $\tensor[-]$ agree on objects (they are the cartesian product of sets) and have common unit object (the singleton set $\unittensor$). By \eqref{eq:whiteRel} and \eqref{eq:blackRel}, one can easily check all the inequalities in \cref{fig:closed lin axioms}.
Both left and right linear adjoints of a relation $c \subseteq X \times Y $ are given by $\op{\nega{c}}$.

\subsection{First-Order Bicategories}\label{sec:fobic}

\begin{figure}[t]
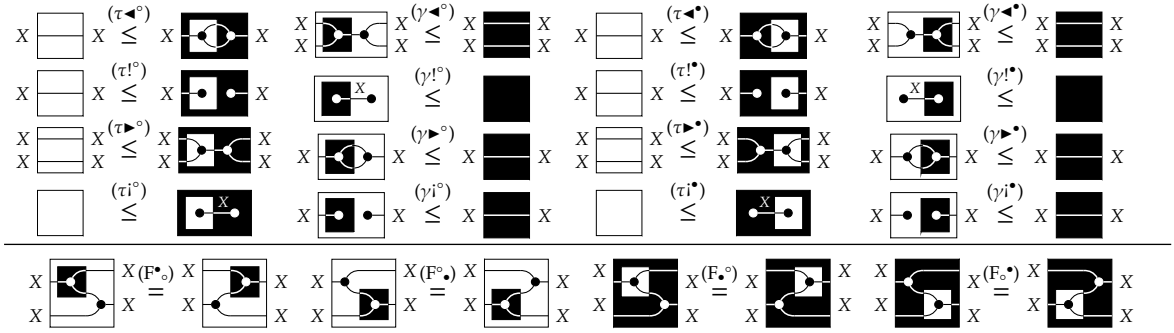

    \mylabel{ax:tauCopierPlus}{\ensuremath{\tau\copier[+]}}
    \mylabel{ax:tauDiscardPlus}{\ensuremath{\tau\discard[+]}}
    \mylabel{ax:tauCocopierPlus}{\ensuremath{\tau\cocopier[+]}}
    \mylabel{ax:tauCodiscardPlus}{\ensuremath{\tau\codiscard[+]}}

    \mylabel{ax:gammaCopierPlus}{\ensuremath{\gamma\copier[+]}}
    \mylabel{ax:gammaDiscardPlus}{\ensuremath{\gamma\discard[+]}}
    \mylabel{ax:gammaCocopierPlus}{\ensuremath{\gamma\cocopier[+]}}
    \mylabel{ax:gammaCodiscardPlus}{\ensuremath{\gamma\codiscard[+]}}

    \mylabel{ax:tauCopierMinus}{\ensuremath{\tau\copier[-]}}
    \mylabel{ax:tauDiscardMinus}{\ensuremath{\tau\discard[-]}}
    \mylabel{ax:tauCocopierMinus}{\ensuremath{\tau\cocopier[-]}}
    \mylabel{ax:tauCodiscardMinus}{\ensuremath{\tau\codiscard[-]}}

    \mylabel{ax:gammaCopierMinus}{\ensuremath{\gamma\copier[-]}}
    \mylabel{ax:gammaDiscardMinus}{\ensuremath{\gamma\discard[-]}}
    \mylabel{ax:gammaCocopierMinus}{\ensuremath{\gamma\cocopier[-]}}
    \mylabel{ax:gammaCodiscardMinus}{\ensuremath{\gamma\codiscard[-]}}

    \mylabel{ax:bwFrob}{\text{F}\ensuremath{\tiny{\begin{array}{@{}c@{}c@{}} \bullet & \\[-3pt] & \circ \end{array}}}}
    \mylabel{ax:bwFrob2}{\text{F}\ensuremath{\tiny{\begin{array}{@{}c@{}c@{}} \circ & \\[-3pt] & \bullet \end{array}}}}
    \mylabel{ax:wbFrob}{\text{F}\ensuremath{\tiny{\begin{array}{@{}c@{}c@{}}  & \circ  \\[-3pt] \bullet & \end{array}}}}
    \mylabel{ax:wbFrob2}{\text{F}\ensuremath{\tiny{\begin{array}{@{}c@{}c@{}}  & \bullet  \\[-3pt] \circ & \end{array}}}}
    \[
        \hspace*{-1em}
        \begin{array}{@{}c@{}c@{}c@{}c@{}}
            \begin{array}{@{}c@{}c@{}c@{}}
                \idCirc[+][X] & \!\!\Lleq{\ref*{ax:tauCopierPlus}}\!\! & 
    \InputIfFileExists{axiomsNEW/linadj/comMmon.tikz}{}{\input{tikz/axiomsNEW/linadj/comMmon.tikz}}
 \\
                \idCirc[+][X] & \!\!\Lleq{\ref*{ax:tauDiscardPlus}}\!\! & 
    \InputIfFileExists{axiomsNEW/linadj/bangMcobang.tikz}{}{\input{tikz/axiomsNEW/linadj/bangMcobang.tikz}}
 \\
                
    \InputIfFileExists{axiomsNEW/id2P.tikz}{}{\input{tikz/axiomsNEW/id2P.tikz}}
 & \!\!\Lleq{\ref*{ax:tauCocopierPlus}}\!\! &  
    \InputIfFileExists{axiomsNEW/linadj/monMcom.tikz}{}{\input{tikz/axiomsNEW/linadj/monMcom.tikz}}
 \\
                \emptyCirc[+]         & \!\!\Lleq{\ref*{ax:tauCodiscardPlus}}\!\! & 
    \InputIfFileExists{axiomsNEW/linadj/cobangMbang.tikz}{}{\input{tikz/axiomsNEW/linadj/cobangMbang.tikz}}

            \end{array} &
            \begin{array}{@{}c@{}c@{}c@{}}
                
    \InputIfFileExists{axiomsNEW/linadj2/monPcom.tikz}{}{\input{tikz/axiomsNEW/linadj2/monPcom.tikz}}
     & \!\!\Lleq{\ref*{ax:gammaCopierPlus}}\!\! & 
    \InputIfFileExists{axiomsNEW/id2M.tikz}{}{\input{tikz/axiomsNEW/id2M.tikz}}
 \\
                
    \InputIfFileExists{axiomsNEW/linadj2/cobangPbang.tikz}{}{\input{tikz/axiomsNEW/linadj2/cobangPbang.tikz}}
 & \!\!\Lleq{\ref*{ax:gammaDiscardPlus}}\!\! & \emptyCirc[-]  \\
                
    \InputIfFileExists{axiomsNEW/linadj/comPmon.tikz}{}{\input{tikz/axiomsNEW/linadj/comPmon.tikz}}
     & \!\!\Lleq{\ref*{ax:gammaCocopierPlus}}\!\! & \idCirc[-][X] \\
                
    \InputIfFileExists{axiomsNEW/linadj/bangPcobang.tikz}{}{\input{tikz/axiomsNEW/linadj/bangPcobang.tikz}}
 & \!\!\Lleq{\ref*{ax:gammaCodiscardPlus}}\!\! & \idCirc[-][X]
            \end{array} &
            \begin{array}{@{}c@{}c@{}c@{}}
                \idCirc[+][X] & \!\!\Lleq{\ref*{ax:tauCopierMinus}}\!\! & 
    \InputIfFileExists{axiomsNEW/linadj2/comMmon.tikz}{}{\input{tikz/axiomsNEW/linadj2/comMmon.tikz}}
 \\
                \idCirc[+][X] & \!\!\Lleq{\ref*{ax:tauDiscardMinus}}\!\! & 
    \InputIfFileExists{axiomsNEW/linadj2/bangMcobang.tikz}{}{\input{tikz/axiomsNEW/linadj2/bangMcobang.tikz}}
 \\
                
    \InputIfFileExists{axiomsNEW/id2P.tikz}{}{\input{tikz/axiomsNEW/id2P.tikz}}
 & \!\!\Lleq{\ref*{ax:tauCocopierMinus}}\!\! &  
    \InputIfFileExists{axiomsNEW/linadj2/monMcom.tikz}{}{\input{tikz/axiomsNEW/linadj2/monMcom.tikz}}
 \\
                \emptyCirc[+]         & \!\!\Lleq{\ref*{ax:tauCodiscardMinus}}\!\! & 
    \InputIfFileExists{axiomsNEW/linadj2/cobangMbang.tikz}{}{\input{tikz/axiomsNEW/linadj2/cobangMbang.tikz}}

            \end{array} &\!\!\!\!\!
            \begin{array}{@{}c@{}c@{}c@{}}
                
    \InputIfFileExists{axiomsNEW/linadj/monPcom.tikz}{}{\input{tikz/axiomsNEW/linadj/monPcom.tikz}}
     & \!\!\Lleq{\ref*{ax:gammaCopierMinus}}\!\! & 
    \InputIfFileExists{axiomsNEW/id2M.tikz}{}{\input{tikz/axiomsNEW/id2M.tikz}}
 \\
                
    \InputIfFileExists{axiomsNEW/linadj/cobangPbang.tikz}{}{\input{tikz/axiomsNEW/linadj/cobangPbang.tikz}}
 & \!\!\Lleq{\ref*{ax:gammaDiscardMinus}}\!\! & \emptyCirc[-]  \\
                
    \InputIfFileExists{axiomsNEW/linadj2/comPmon.tikz}{}{\input{tikz/axiomsNEW/linadj2/comPmon.tikz}}
     & \!\!\Lleq{\ref*{ax:gammaCocopierMinus}}\!\! & \idCirc[-][X] \\
                
    \InputIfFileExists{axiomsNEW/linadj2/bangPcobang.tikz}{}{\input{tikz/axiomsNEW/linadj2/bangPcobang.tikz}}
 & \!\!\Lleq{\ref*{ax:gammaCodiscardMinus}}\!\! & \idCirc[-][X]
            \end{array} \\
            \midrule
            \multicolumn{4}{c}{
            \begin{array}{@{}c@{}c@{}c@{}c@{}c@{}c@{}c@{}c@{}c@{}c@{}c@{}c}
                
    \InputIfFileExists{axiomsNEW/bwS2.tikz}{}{\input{tikz/axiomsNEW/bwS2.tikz}}
 & \!\!\Leq{\ref*{ax:bwFrob}}\!\! & 
    \InputIfFileExists{axiomsNEW/bwZ2.tikz}{}{\input{tikz/axiomsNEW/bwZ2.tikz}}
  & 
    \InputIfFileExists{axiomsNEW/bwS.tikz}{}{\input{tikz/axiomsNEW/bwS.tikz}}
 & \!\!\Leq{\ref*{ax:bwFrob2}}\!\! & 
    \InputIfFileExists{axiomsNEW/bwZ.tikz}{}{\input{tikz/axiomsNEW/bwZ.tikz}}
 & 
    \InputIfFileExists{axiomsNEW/wbS2.tikz}{}{\input{tikz/axiomsNEW/wbS2.tikz}}
 & \!\!\Leq{\ref*{ax:wbFrob}}\!\! & 
    \InputIfFileExists{axiomsNEW/wbZ2.tikz}{}{\input{tikz/axiomsNEW/wbZ2.tikz}}
  & 
    \InputIfFileExists{axiomsNEW/wbS.tikz}{}{\input{tikz/axiomsNEW/wbS.tikz}}
 & \!\!\Leq{\ref*{ax:wbFrob2}}\!\! & 
    \InputIfFileExists{axiomsNEW/wbZ.tikz}{}{\input{tikz/axiomsNEW/wbZ.tikz}}

            \end{array}
            }
        \end{array}
    \]
    \caption{Additional axioms for fo-bicategories}\label{fig:fo bicat axioms}

\end{figure}

After (co)cartesian and linear bicategories, we can recall first-order bicategories from \cite{bonchi2024diagrammatic}.
\begin{definition}\label{def:fobicategory}
    A \emph{first-order bicategory} $\Cat{C}$ %
    consists of a closed linear bicategory $(\Cat{C}, \tensor[+], \tensor[-],\unittensor)$, a cartesian bicategory %
 $(\Cat{C},\copier[+], \cocopier[+])$ and a cocartesian bicategory $(\Cat{C}, \copier[-], \cocopier[-])$, such that
\begin{enumerate}%
\item  the white comonoid  $(\copier[+], \discard[+])$ is left and right linear adjoint to black monoid $(\cocopier[-], \codiscard[-])$ and $(\cocopier[+], \codiscard[+])$ is left and right linear adjoint to $(\copier[-], \discard[-])$ %
i.e., the 16 inequalities in the top of \cref{fig:fo bicat axioms} hold;
\item white and black (co)monoids satisfy the linear Frobenius laws, i.e.\ \eqref{ax:bwFrob}, \eqref{ax:bwFrob2}, \eqref{ax:wbFrob}, \eqref{ax:wbFrob2} hold.
\end{enumerate}
A \emph{morphism of fo-bicategories} is a morphism of linear bicategories \emph{and} of (co)cartesian bicategories. %
We denote by $\FOBic$ the category of fo-bicategories and their morphisms.

\end{definition}
We have seen that $\Rel$ is a closed linear bicategory, $\Relp$ a cartesian bicategory and $\Relm$ a cocartesian bicategory.
Given~\eqref{eq:comonoidsREL}, it is easy to check the inequalities in \cref{fig:fo bicat axioms}. %

If $\Cat{C}$ is a fo-bicategory, then $\co{\Cat{C}}$ is a fo-bicategory when swapping white and black structures. Similarly, $\opposite{\Cat{C}}$ is a fo-bicategory when swapping monoids and comonoids.

In a fo-bicategory $\Cat{C}$, left and right linear adjoints of an arrow $c$ coincide and are denoted by $\rla{c}$. The assignment $c \mapsto \rla{c}$ gives rise to an identity-on-objects isomorphism of fo-bicategories $\rla{(\cdot)} \colon \Cat{C} \to \opposite{(\co{\Cat{C}})}$. Similarly, $\op{(\cdot)} \colon \Cat{C} \to \opposite{\Cat{C}}$ in \eqref{eqdagger} is also an isomorphism of fo-bicategories.

\noindent\begin{minipage}{0.75\linewidth}
\hspace*{0.3cm} Since the diagram on the right commutes, %
one can define the complement as the diagonal of the square, namely
$\nega{(\cdot)} \defeq \op{(\rla{(\cdot)})}$. Clearly $\nega{} \colon \Cat{C} \to \co{\Cat{C}}$ is an isomorphism of fo-bicategories.
Moreover, it induces a boolean algebra on each homset of $\Cat{C}$.
\end{minipage}
\begin{minipage}{0.25\linewidth}\vspace{-0.3cm}
\[{%
\xymatrix@R=5mm{\Cat{C} \ar[r]|{\op{(\cdot)}} \ar[d]_{\rla{(\cdot)}} & \opposite{\Cat{C}} \ar[d]^{\rla{(\cdot)}} \\
\opposite{(\co{\Cat{C}})} \ar[r]|{\op{(\cdot)}} & \co{\Cat{C}}}}\]
\end{minipage}

\begin{proposition}\label{prop:enrichment}
Let $\Cat{C}$ be a fo-bicategory. Then, every homset of $\Cat{C}$ is a boolean algebra.
\end{proposition}

\begin{proposition}\label{prop:map}
Let $F\colon \Cat{C} \to \Cat{D}$ be a morphism of fo-bicategories. Then, $\nega{F(c)} =F(\nega{c})$ for all arrows $c$, and hence $F$ preserves the boolean structure on the homsets.
\end{proposition}

The next properties of maps (\cref{def:maps}) plays a key role in our work.
\begin{proposition}\label{prop:maps}%
For all maps $f\colon X \to Y$ and arrows $c\colon Y\to Z$, it holds that $f \seq[+] \nega{c} = \nega{(f \seq[+] c)}$ %
\end{proposition}

\subsection{Freely Generated First-Order Bicategories}
We conclude this section by giving to the reader a taste of how fo-bicategories relate to first-order theories. 
First, we recall from~\cite{bonchi2024diagrammatic} the freely generated fo-bicategory $\FOB$.

    Given a monoidal signature $\sign$, namely a set of symbols $R \colon n \to m$ with arity $n$ and coarity $m$, $\FOB$ is the fo-bicategory whose objects are natural numbers and arrows $c \colon n \to m$ are string diagrams generated by the following rules:
    \begin{center}
        $
        \def\arraystretch{3}
        \begin{array}{ccccc}
            \inferrule{ }{\emptyCirc[+] \colon 0 \to 0}
            &
            \inferrule{ }{\idCirc[+] \colon 1 \to 1}
            &
            \inferrule{ }{\symmCirc[+] \colon 2 \to 2}
            &
            \inferrule{R \colon n \to m \in \sign}{\boxCirc[+]{R} \colon n \to m}
            &
            \inferrule{\Circ{c} \colon n \to m, \Circ{d} \colon m \to o}{\seqCirc[+]{c}{d}[n][o] \colon n \to o}
            \\
            \inferrule{ }{\copierCirc[+] \colon 1 \to 2}
            &
            \inferrule{ }{\discardCirc[+] \colon 1 \to 0}
            &
            \inferrule{ }{\cocopierCirc[+] \colon 2 \to 1}
            &
            \inferrule{ }{\codiscardCirc[+] \colon 0 \to 1}
            &
            \inferrule{\Circ{c} \colon n \to m, \Circ{d} \colon o \to p}{\tensorCirc[+]{c}{d}[n][m][o][p] \colon n+o \to m +p}
            \\
            \inferrule{ }{\emptyCirc[-] \colon 0 \to 0}
            &
            \inferrule{ }{\idCirc[-] \colon 1 \to 1}
            &
            \inferrule{ }{\symmCirc[-] \colon 2 \to 2}
            &
            \inferrule{R \colon n \to m \in \sign}{\boxOpCirc[-]{R} \colon m \to n}
            &
            \inferrule{\Circ{c} \colon n \to m, \Circ{d} \colon m \to o}{\seqCirc[-]{c}{d}[n][o] \colon n \to o}
            \\
            \inferrule{ }{\copierCirc[-] \colon 1 \to 2}
            &
            \inferrule{ }{\discardCirc[-] \colon 1 \to 0}
            &
            \inferrule{ }{\cocopierCirc[-] \colon 2 \to 1}
            &
            \inferrule{ }{\codiscardCirc[-] \colon 0 \to 1}
            &
            \inferrule{\Circ{c} \colon n \to m, \Circ{d} \colon o \to p}{\tensorCirc[-]{c}{d}[n][m][o][p] \colon n+o \to m +p}
        \end{array}
        $
    \end{center}
    More precisely, arrows are equivalence classes of string diagrams w.r.t $\syninclusion \cap \syninclusionop$, where $\syninclusion$ is the precongruence (w.r.t. $\seq[+], \tensor[+], \seq[-]$ and $\tensor[-]$) generated by the axioms in Figures~\ref{fig:cb axioms},\ref{fig:closed lin axioms},\ref{fig:fo bicat axioms},\ref{fig:cocb axioms} (with $X,Y,Z,W$ replaced by natural numbers, and $a,b,c,d$ by diagrams of the appropriate type) and the axioms forcing $\boxCirc[+]{R}$ and $\boxOpCirc[-]{R}$ to be linear adjoints:
    \begin{center}
        $
        \begin{array}{@{}c@{}c@{}c@{}c}
            \!\!\idCirc[+][n] \!\!\leq\!\! 
    \InputIfFileExists{axiomsNEW/linadj/rMrop.tikz}{}{\input{tikz/axiomsNEW/linadj/rMrop.tikz}}
 & 
    \InputIfFileExists{axiomsNEW/linadj/ropPr.tikz}{}{\input{tikz/axiomsNEW/linadj/ropPr.tikz}}
  \!\!\leq\!\! \idCirc[-][m] &
            \idCirc[+][m] \!\!\leq\!\!  
    \InputIfFileExists{axiomsNEW/linadj2/rMrop.tikz}{}{\input{tikz/axiomsNEW/linadj2/rMrop.tikz}}
 & 
    \InputIfFileExists{axiomsNEW/linadj2/ropPr.tikz}{}{\input{tikz/axiomsNEW/linadj2/ropPr.tikz}}
  \!\!\leq\!\! \idCirc[-][n]
        \end{array}
        $
    \end{center}

To give semantics to these diagrams we need \emph{interpretations}, i.e. pairs $\interpretation = (X, \rho)$, where $X$ is a set and $\rho$ is a function assigning to each $R  \colon n \to m \in \sign$ a relation $\rho(R) \colon X^n \to X^m$. Since $\FOB$ is the free fo-bicategory, for any interpretation $\interpretation$ there exists a unique morphism of fo-bicategories $\interpretationFunctor \colon \FOB \to \Rel$ such that  $\interpretationFunctor(1) = X$ and $\interpretationFunctor(\boxCirc[+]{R}[n][m]) = \rho(R) \subseteq X^n \times X^m$.
Intuitively, $\interpretationFunctor$ is defined inductively by \eqref{eq:whiteRel}, \eqref{eq:blackRel} and \eqref{eq:comonoidsREL} with the free cases provided by $\interpretation$.

\medskip

A \emph{diagrammatic first-order theory} is a pair $\T{T} = (\sign, \T{I})$ where $\sign$ is a monoidal signature and $\T{I}$ is a set of \emph{axioms}: pairs $(c,d)$ for $c,d \colon n \to m$ in $\FOB$, standing for $c \leq d$. An interpretation $\interpretation$ is a \emph{model} of $\T{T}$ if and only if, for all $(c,d) \in \T{I}$, $\interpretationFunctor(c) \subseteq \interpretationFunctor(d)$.
As illustrated in~\cite{bonchi2024diagrammatic}, one can generate the fo-bicategory $\FOB[\T{T}]$ and, in the spirit of Lawvere's functorial semantics \cite{LawvereOriginalPaper}, models of $\T{T}$ are in one-to-one correspondence with morphisms $F \colon \FOB[\T{T}] \to \Rel$.

\begin{example}
    Consider the theory $\T{T} = (\sign, \T{I})$, where $\sign = \{ R \colon 1 \to 1 \}$ and $\mathbb{I}$ be as follows:
    \begin{center}
        $\{\, (\; \idCirc[+] \;,\; \boxCirc[+]{R} \;) ,\quad
        (\; \seqCirc[+]{R}{R} \;,\; \boxCirc[+]{R} \;) ,\quad
        (\; \intersectionCirc{R}{R}[][][w][op] \;,\; \idCirc[+] \;) ,\quad
        (\; \topCirc \;,\; \unionCirc{R}{R}[][][w][op] \;) \,\}$.
    \end{center}
An interpretation is a set $X$ and a relation $R \subseteq X \times X$. It is a model iff $R$ is an order, i.e., reflexive ($\id[+][X]\subseteq R$), transitive ($R \seq[+]R \subseteq R$), antisymmetric ($R\cap \op{R} \subseteq \id[+]$) and total ($\top \subseteq R\cup \op{R}$).
\end{example}

\begin{remark}
    A direct encoding of traditional first-order theories into diagrammatic ones is illustrated in~\cite{bonchi2024diagrammatic}. Shortly, a predicate symbol $P$ of arity $n$ becomes a symbol $P \colon n \to 0 \in \sign$, drawn as $\predicateCirc[+]{P}[n]$, and a $n$-ary function symbol $f$ becomes $f \colon n \to 1 \in \sign$, drawn as $\funcCirc[+]{f}[n]$. 
    
    \noindent\begin{minipage}{0.85\textwidth}
        For instance, the formula $\exists x . P(x) \wedge Q(x, f(y))$ is rendered as on the right, where $\codiscardCirc[+]$ plays the role of $\exists$ and $\copierCirc[+]$ that of $\wedge$.
        Note that both predicate and function symbols of traditional first-order theories are regarded as symbols of 
    \end{minipage}
    \begin{minipage}{0.15\textwidth}
        \[\begin{tikzpicture}
            \begin{pgfonlayer}{nodelayer}
                \node [style=none] (81) at (0, -0.2) {};
                \node [style=none] (112) at (-2, 1.5) {};
                \node [style=none] (113) at (-2, -1.5) {};
                \node [style=none] (114) at (1.5, -1.5) {};
                \node [style=none] (115) at (1.5, 1.5) {};
                \node [{boxStyle/+}] (123) at (0.5, -0.5) {$Q$};
                \node [style={funcStyle/+}, scale=0.8] (126) at (-1, -0.75) {$f$};
                \node [style=none] (128) at (0.5, -0.2) {};
                \node [style=none] (129) at (0.5, -0.75) {};
                \node [style=none] (130) at (-2, -0.75) {};
                \node [{boxStyle/+}] (131) at (0.5, 0.65) {$P$};
                \node [style={dotStyle/+}] (132) at (-1.25, 0.225) {};
                \node [style=none] (133) at (0, 0.65) {};
                \node [style={dotStyle/+}] (134) at (-0.75, 0.225) {};
            \end{pgfonlayer}
            \begin{pgfonlayer}{edgelayer}
                \draw [{bgStyle/+}] (114.center)
                     to (113.center)
                     to (112.center)
                     to (115.center)
                     to cycle;
                \draw (128.center) to (81.center);
                \draw (126) to (129.center);
                \draw [in=180, out=0] (130.center) to (126);
                \draw (133.center) to (131);
                \draw (132) to (134);
                \draw [bend left] (134) to (133.center);
                \draw [bend right] (134) to (81.center);
            \end{pgfonlayer}
        \end{tikzpicture}
        \]
    \end{minipage}

    \noindent the monoidal signature $\sign$. Function symbols are constrained to represent functions by adding to $\T{I}$ the axioms of maps, i.e., the inequalities in \eqref{eq:def:map}.

\end{remark}

%% file: sections/hyperdoctrine2.tex
\section{From Elementary-Existential Doctrines to Boolean Hyperdoctrines}\label{sec:hyp}
The notion of hyperdoctrine was introduced by Lawvere in a series
of seminal papers \cite{AF,EHCSAF}, in order to provide an algebraic framework for first-order (intuitionistic) logic. 
Over the years, various generalizations and specializations of this concept have been formulated and applied across multiple domains in the fields of logic and computer science.

In this work, we employ a generalization of the notion of hyperdoctrine  introduced by Maietti and Rosolini in \cite{QCFF,EQC,UEC}, namely that of \emph{elementary and existential doctrine}.

\subsection{Elementary and Existential Doctrines}
Elementary and existential doctrines can be thought of as a categorification of the so-called ``regular fragment'' of first-order intuitionistic logic, i.e. the $(\exists,=,\top, \wedge)$-fragment.

Hereafter $\langle f,g\rangle$ denotes the pairing of $f$ and $g$ and $\Delta_X$ denotes $\langle \id[+][X], \id[+][X] \rangle$.

\begin{definition}\label{def primary doctrine}
    An \emph{elementary and existential doctrine}~is a functor
    $\doctrine{\Cat{C}}{P}$ from the opposite of a category $\Cat{C}$ with finite products to the category of inf-semilattices such that:
    \begin{itemize}
        \item for every $Y$ in $\Cat{C}$ there exists an element $\delta_Y$ in $P(Y\times Y)$, called \emph{equality predicate}, such that for a morphism $\freccia{X\times Y}{\id[+][X]\times \Delta_Y}{X\times Y\times Y}$ in $\Cat{C}$  and every element $\alpha$ in $P(X\times Y)$, the assignment
        \[ \exists_{\id[+][X]\times \Delta_Y}(\alpha) \defeq P_{\angbr{\pr_1}{\pr_2}}(\alpha)\wedge P_{\angbr{\pr_2}{\pr_3}}(\delta_Y)\]
         determines a left adjoint to the functor $\freccia{P(X\times Y \times Y)}{P_{\id[+][X]\times \Delta_Y}}{P(X\times Y)}$;
 \item for any projection $\freccia{X\times  Y }{\pr_X}{X}$, the functor $\freccia{P(X)}{{P_{\pr_X}}}{P(X\times Y)}$ has a left adjoint $\exists_{\pr_X}$, and these satisfy the \emph{Beck-Chevalley condition} and \emph{Frobenius reciprocity}, see~\cite[Sec. 2]{QCFF}.
    \end{itemize}
   
    \end{definition}

\begin{remark}\label{rem:generalised adjoint hyperdoctrine}
        In an elementary and existential doctrine, for every $\freccia{X}{f}{Y}$ of $\Cat{C}$ the functor $P_f$ has a left adjoint $\exists_f$ that can be computed as
        $\exists_{\pr_Y}(P_{f\times {\id[+][X]}_Y}(\delta_Y)\wedge P_{\pr_X}(\alpha))$
        for $\alpha$ in $P(X)$, where $\pr_X$ and $\pr_Y$ are the projections from $X\times Y$. These left ajoints satisfy the Frobenius reciprocity but not necessarily the Beck-Chevalley condition. See  \cite[Rem. 6.4]{MaiettiTrotta21}.
 \end{remark}

\begin{definition}%
\label{def:morphism of primary doctrine}
Let $\doctrine{\Cat{C}}{P}$ and $\doctrine{\Cat{D}}{R}$ be two elementary and existential doctrines. A \emph{morphism of elementary and existential doctrines} is given by a pair $(F,\mathfrak{b})$ where    

          \noindent \begin{minipage}{0.65\linewidth}
            \begin{itemize}
            \item $\freccia{\Cat{C}}{F}{\Cat{D}}$ is a finite product preserving functor;
            \item $\freccia{ P}{\mathfrak{b}}{ F^{\op} \seq[+] R}$ is a natural transformation;
            \end{itemize}
        satisfying the following conditions:
        \end{minipage}     
 \begin{minipage}{0.35\linewidth}           \[\begin{tikzcd}[row sep=1ex]
               \Cat{C}^{\op} \\
               && \infsl \\
               \Cat{D}^{\op}
               \arrow[""{name=0, anchor=center, inner sep=0}, "R"', from=3-1, to=2-3]
               \arrow[""{name=1, anchor=center, inner sep=0}, "P", from=1-1, to=2-3]
               \arrow["F^{\op}"', from=1-1, to=3-1]
               \arrow["\mathfrak{b}"', shorten <=4pt, shorten >=4pt, from=1, to=0]
            \end{tikzcd}\]
          \end{minipage}
          \vspace{-0.7cm}
        \begin{enumerate}
            \item  for every object $X$ of $\Cat{C}$, $\mathfrak{b}_{X\times X}(\delta_X)= \delta_{FX\times FX}$;
        \item  for every  $\freccia{X\times Y}{\pr_X}{X}$ of $\Cat{C}$ and  for every  $\alpha$ in $P(X\times Y)$, 
            $\exists_{F(\pr_X)}\mathfrak{b}_{X\times Y}(\alpha)=\mathfrak{b}_X(\exists_{\pr_X}(\alpha))$.
        \end{enumerate}
\end{definition}

We write $\EED$ for the category of elementary and existential doctrines and morphisms. %

            \begin{example}\label{ex_sub_on_regular}
            The powerset functor $\doctrine{\set}{\Pow}$ is the archetypal example of an elementary and existential doctrine. More generally, for any regular category $\Cat{C}$, the subobjects functor $\doctrine{\Cat{C}}{\Sub_{\Cat{C}}}$ is an elementary and existential doctrine, see~\cite{EQC,QCFF}. This assignment extends to an inclusion of the category $\mathbb{REG}$ of regular categories into $\EED$.
                \end{example}
           
            \begin{example}\label{example_cartbicat_provide_eed}
                For a cartesian bicategory $\Cat{C}$, the functor $\doctrine{\mathsf{Map}(\Cat{C})}{\Cat{C}[-,I]}$ is an elementary and existential doctrine, where the actions of left adjoints is given $\exists_g(f):=f \seq[+]g^\dagger$~\cite[Thm. 20]{bonchi2021doctrines}. As we will see in \S \ref{sec:adjunction}, this assignment extends to an inclusion of $\CartBic$ into $\EED$.
            \end{example}

           Similarly to cartesian bicategories, elementary and existential doctrines have enough structure to deal with the notion of \emph{functional} (or single-valued) and  \emph{entire} (total) predicates. %
            \begin{definition}[From \cite{TECH}]\label{def_func_and_entire_for_P}
                Let $\doctrine{\Cat{C}}{P}$ be an elementary and existential doctrine. An element $\alpha\in P(X\times Y)$ is said to be \emph{functional from} $X$ to $Y$ if 
                $P_{\angbr{\pi_1}{\pi_2}}(\alpha)\wedge P_{\angbr{\pi_1}{\pi_3}}(\alpha)\leq P_{\angbr{\pi_2}{\pi_3}}(\delta_Y)$ in $P(X\times Y\times Y)$.
                Also, $\alpha$ is said to be \emph{entire from} $X$ \emph{to} $Y$ if $\top_X\leq \exists_{\pr_X}(\alpha)$ in $P(X)$.
            \end{definition}

            \begin{remark}\label{rem_EED_morph_sens_Fe_and_Te_into_Fe_and_Te}
 By definition, a morphism of elementary and existential doctrines preserves both $\exists_{\pr_X}$ and $\delta_Y$. Therefore it preserves functional and entire elements. %
            \end{remark}

\begin{example}
In the doctrine $\doctrine{\set}{\Pow}$ from \cref{ex_sub_on_regular}, an element $\alpha\in \Pow(X\times Y)$ is functional if and only if it defines a partial function from $X$ to $Y$, while it is entire if it provides a total relation from $X$ to $Y$.  
\end{example}
\begin{example}
In the doctrine $\doctrine{\map(\Cat{C})}{\Cat{C}[-,I]}$ from \cref{example_cartbicat_provide_eed}, functional and entire elements are precisely maps of $\Cat{C}$. A detailed proof is in \cref{ex_func_ent_rel_in_CB} in \cref{app:resadj}. 
\end{example}

\subsection{Boolean Hyperdoctrines}
In this section we recall the notion of \emph{boolean hyperdoctrine}, and some useful properties. 

\begin{definition}[boolean hyperdoctrine]\label{def:hyperdoctrine}
    Let $\Cat{C}$ be a category with finite products. A functor $\hyperdoctrine{\Cat{C}}{P}$  is a \emph{boolean hyperdoctrine} if it is an elementary and existential doctrine.
\end{definition}
    A morphism $(F,\mathfrak{b}):P\to R$ of boolean hyperdoctrines is a morphism of elementary and existential doctrines such that $\mathfrak{b}_X$ is a morphism of boolean algebras for all objects $X$ of $\Cat{C}$. We denote by $\BHD$ the category of boolean hyperdoctrines and their morphisms.

It is well-known that in first-order logic the universal quantifier can be derived by the existential quantifier and the negation. The same happens in boolean hyperdoctrines: for all arrows $\freccia{X}{f}{Y}$, the functor $\forall_f(-)\defeq \neg \exists_f \neg (-)$ is a right adjoint to $P_f$ -- see~\cref{app:boolhyp}.

\begin{example}\label{ex:booleancat}
The powerset functor $\hyperdoctrine{\set}{\Pow}$ provides an example of boolean hyperdoctrine. This can be generalized to an arbitrary \emph{boolean category} $\Cat{B}$, namely a coherent category such that every subobject has a complement, see~\cite[Sec. A1.4, p. 38]{SAE}. The subobjects functor on $\Cat{B}$ is a boolean hyperdoctrine $\hyperdoctrine{\Cat{B}}{\Sub_{\Cat{B}}}$.
 \end{example}

\begin{example}\label{ex:syntacticdoc}
Given a standard first-order theory $\textsc{Th}$ in a first-order language $\mathcal{L}$ (for simplicity single sorted), one can consider the functor
$\hyperdoctrine{\mV}{\mathcal{L}^\textsc{Th}}$. The base category $\mathcal{V}$ is the \emph{syntactic} category of $\mathcal{L}$, i.e. the category where objects are  natural numbers and  morphisms are lists of terms, while the predicates of  $\mathcal{L}^\textsc{Th}(n)$ are given by equivalence classes (with respect to provable reciprocal
consequence $\dashv\vdash $) of well-formed formulae with free variables in $\{x_1, \dots, x_n\}$, and the partial order is given by the provable consequences, according to the fixed theory $\textsc{Th}$. In this case, the left adjoint to the weakening functor $\mathcal{L}^\textsc{Th}_{\pr}$ is computed by existentially quantifying the variables that are not involved in the substitution induced by the projection $\pi$. Dually, the right adjoint is computed by quantifying universally.
\end{example}

We conclude this section with a result that, intuitively, is the analogous of \cref{prop:maps}.  %
\begin{lemma}\label{lemma:boolhypcomp}
Let $\hyperdoctrine{\Cat{C}}{P}$ be a boolean hyperdoctrine and $\phi\in P(X\times Y)$ a functional and entire element from $X$ to$Y$. For all $\psi \in P(Y\times Z)$, it holds that 
\[\exists_{\pr_{X\times Z}}(P_{\pr_{X\times Y}}(\phi)\wedge P_{\pr_{Y\times Z}}(\neg \psi)) =  \neg(\, \exists_{\pr_{X\times Z}}(P_{\pr_{X\times Y}}(\phi)\wedge P_{\pr_{Y\times Z}}(\psi))\,) \text{.}\]
\end{lemma}

\section{An Adjunction and an Equivalence}\label{sec:adjunction} 

In \cite{bonchi2021doctrines}, cartesian bicategories and elementary existential doctrines are compared. The main results of \cite[Thm. 28]{bonchi2021doctrines} states that there exists the following adjunction. %
\begin{equation}\label{adj_CB_EED}
\begin{tikzcd}
	\CartBic && \EED
	\arrow["\HM"',{name=0, anchor=center, inner sep=0}, curve={height=20pt}, hook, from=1-1, to=1-3] 
	\arrow["\REL"',{name=1, anchor=center, inner sep=0}, curve={height=20pt}, from=1-3, to=1-1]
	\arrow["\dashv"{anchor=center, rotate=-90}, draw=none, from=1, to=0]
\end{tikzcd}
\end{equation}
The embedding $\HM\colon  \CartBic \to \EED$ maps a cartesian bicategory $\Cat{C}$ into the hom-functor $\doctrine{\map(\Cat{C})}{\Cat{C}[-,I]}$ that, as explained in \cref{example_cartbicat_provide_eed}, is an elementary existential doctrine. A morphism of cartesian bicategories $F\colon \Cat{C} \to \Cat{D}$ is mapped to the morphism of doctrines $(\tilde{F},\mathfrak{b}^F)$ where $\tilde{F}\colon \map(\Cat{C}) \to \map(\Cat{D})$ is the functor $F$ restricted to $\map(\Cat{C})$ and $\mathfrak{b}^F_X \colon \Cat{C}[X,I] \to \Cat{D}[F(X),I]$ is defined  as $\mathfrak{b}^F_X(c) \defeq F(c)$ for all objects $X$ of $\Cat{C}$ and arrows $c\in  \Cat{C}[X,I]$.

The functor $\REL\colon  \EED  \to \CartBic$ is a generalisation  to elementary and existential doctrines of the construction of bicategory relations associated with a regular category (see \cite[Ex. 1.4]{carboni1987cartesian}).
For $\doctrine{\Cat{C}}{P}$,  the cartesian bicategory $\REL(P)$ is defined as follows:
\begin{itemize}
    \item objects are those of $\Cat{C}$; for objects $X,Y$, the homsets $\REL(P)[X,Y]$ are the posets $P(X\times Y)$;
    \item the identity for an object $X$ is the equality predicate $\delta_X$ in $P(X \times X)$;
    \item composition of $\phi\colon X \to Y$ and $\psi \colon Y\to Z$ is given by  $\exists_{\pr_{X\times Z}}(P_{\pr_{X\times Y}}(\phi)\wedge P_{\pr_{Y\times Z}}(\psi))$.
\end{itemize}
For a morphism of doctrines $(F,\mathfrak{b})\colon P \to Q$, the morphism of cartesian bicategories $\REL(F,\mathfrak{b})\colon \REL(P) \to \REL(Q)$ is defined for all objects $X$ as $ \REL(F,\mathfrak{b}) (X) \defeq F(X)$ and for all arrows $\phi \colon X\to Y$ in $\REL(P)$, i.e., elements $\phi\in P(X\times Y)$, as  $\REL(F,\mathfrak{b})(\phi) \defeq \mathfrak{b}_{X\times Y}(\phi)$. The reader is referred to \cite{bonchi2021doctrines} or to \Cref{app_section_adj} for further details on the adjunction in \eqref{adj_CB_EED}.

\medskip

Another result in \cite[Thm. 35]{bonchi2021doctrines} shows that the adjunction in \eqref{adj_CB_EED} restricts to an equivalence 
\begin{equation}\label{eq:theequivalence}
\CartBic\equiv \overline{\EED}
\end{equation}
where $\overline{\EED}$ is a full subcategory of $\EED$ whose objects are particularly well-behaved doctrines. %
For the sake of readability, we will make clear in \S \ref{sec:equivalence} what these doctrines are.

%% file: sections/peircianbicategory.tex
\section{Peircean Bicategories}\label{sec:pb}
In this section we introduce the notion of \emph{peircean bicategory}, and we prove that such a new notion provides an alternative presentation of fo-bicategories. The name peircean is due to the fact that, like in Peirce's algebra of relations~\cite{peirce1883_studies-in-logic.-by-members-of-the-johns-hopkins-university}, and differently from fo-bicategories, the structure of boolean algebra is taken as a primitive.

\begin{definition}\label{def:peircean-bicategory}
A \emph{peircean bicategory} consists of a cartesian bicategory $(\Cat{C}, \copier[+], \cocopier[+])$ such that
\begin{enumerate}
\item every homset $\Cat{C}[X,Y]$ carries a Boolean algebra $(\Cat{C}[X,Y], \vee, \bot, \wedge, \top, \nega{})$;
\item for all maps $f\colon X \to Y$ and arrows $c\colon Y \to Z$,
\begin{equation}\label{eq:prop mappe}\tag{$\nega{}\mathcal{M}$}
    f \seq[+] \nega{c} = \nega{(f \seq[+] c)}\text{.}
\end{equation}
\end{enumerate}
A \emph{morphism of peircean bicategories} is a morphism of cartesian bicategories $F \colon \Cat{C}\to \Cat{D}$ such that $F(\nega{c})=\nega{F(c)}$.
We write $\PeirceBic$ for the category of peircean bicategories and their morphisms.
\end{definition}
By Propositions \ref{prop:enrichment} and \ref{prop:maps} every fo-bicategory is a peircean bicategory. By Proposition \ref{prop:map} every morphism of fo-bicategories is a morphism of peircean bicategories.

Vice versa, every peircean bicategory  $(\Cat{C}, \copier[+], \cocopier[+])$
gives rise to a fo-bicategory. The black structure $(\Cat{C}, \copier[-], \cocopier[-])$ is defined as expected from the white one and $\nega{}$. Namely:
\begin{equation}\label{eq:defnegative}
\begin{array}{cccc}
c \seq[-] d \defeq \nega{(\nega{c} \seq[+] \nega{d})} 
&
 \id[-][X] \defeq \nega{\id[+][X]}
&
 c \tensor[-]d \defeq \nega{(\nega{c}\tensor[+] \nega{d})}
 &
 \symm[-][X][Y] \defeq \nega{\symm[+][X][Y]}
 \\
 \copier[-][X] \defeq \nega{\copier[+][X]}
&
 \discard[-][X] \defeq \nega{\discard[+][X]}
&
 \cocopier[-][X] \defeq \nega{\cocopier[+][X]}
&
\codiscard[-][X] \defeq \nega{\codiscard[+][X]}
\end{array}
\end{equation}
With this definition, it is immediate to see that $\nega{} \colon (\co{\Cat{C}}, \copier[+], \cocopier[+]) \to (\Cat{C}, \copier[-], \cocopier[-])$ is an isomorphism and thus to conclude that $(\Cat{C}, \copier[-], \cocopier[-])$ is a cocartesian bicategory. Proving that $(\Cat{C}, \copier[+], \cocopier[+])$ and $(\Cat{C}, \copier[-], \cocopier[-])$ give rise to a fo-bicategory
is the main technical effort of this paper: the diagrammatic proof in Appendix~\ref{app:pb} crucially exploits the boolean properties and \eqref{eq:prop mappe}.

\begin{theorem}\label{thm_equiv_FOBic_PeirceBic}
There is an isomorphism of categories $\FOBic \cong \PeirceBic$.
\end{theorem}
Note that, differently from \cref{def:fobicategory}, \cref{def:peircean-bicategory} is not purely axiomatic, since the property 2 requires $f$ to be a map. However, the notion of a peircean bicategory is notably more succinct than that of a fo-bicategory, making it more convenient for our purposes. %

%% file: sections/alltogether.tex
\section{An Equational Presentation of Boolean Hyperdoctrines}\label{sec:restrictingadjunction}

The main purpose of this section is to establish a formal link between fo-bicategories and boolean hyperdoctrines. In particular, we are going to show that the adjunction presented in \eqref{adj_CB_EED} restricts to an adjunction 
between $\FOBic$ and $\mathbb{BHD}$. \Cref{thm_equiv_FOBic_PeirceBic}  allows us to conveniently work with peircean bicategories. We commence with the following result. %

\begin{proposition}\label{prop_inclusion_PD_BHD}
Let $\Cat{C}$ be a peircean bicategory. Then $\HM(\Cat{C})$ %
is a boolean hyperdoctrine. 
\end{proposition}
  \begin{proof}%
  By \eqref{adj_CB_EED}, $\doctrine{\map(\Cat{C})}{\Cat{C}[-,I]}$ is an elementary and existential doctrine and, by definition of peircean bicategories, $\Cat{C}[X,I]$ is a boolean algebra for all objects $X$. To conclude that $\hyperdoctrine{\map(\Cat{C})}{\Cat{C}[-,I]}$, one has only to show that, for all maps $f\colon X \to Y$, $\Cat{C}[f,I]\colon \Cat{C}[Y,I] \to \Cat{C}[X,I]$ is a morphism of boolean algebras. Since, by \eqref{adj_CB_EED}, $\Cat{C}[f,I]$ is a morphism of inf-semilattices, it is enough to show that it preserves negation: for all $c\in \Cat{C}[Y,I] $
  \begin{align}
  \Cat{C}[f,I](\nega{c}) & = f\seq[+]\nega{c} \tag{Definition of $\Cat{C}[-,I]$} \\
  &= \nega{(f \seq[+]c)} \tag{\ref{eq:prop mappe}}\\
  &=\nega{\Cat{C}[f,I](c)} \tag{Definition of $\Cat{C}[-,I]$}
  \end{align}
    \end{proof}

The above proposition allows us to characterize peircean bicategories as follows:
\begin{corollary}\label{cor_cart_bic_is_peircea_iff_hyperdoc}
	Let $\Cat{C}$ be a cartesian bicategory. Then it is a peircean bicategory if and only if $\HM(\Cat{C})$ is a boolean hyperdoctrine.
\end{corollary}

To prove that, for any boolean hyperdoctrine $P$, $\REL(P)$ is a peircean bicategory, we need to establish a formal correspondence between \cref{def:maps} and \cref{def_func_and_entire_for_P}. %
\begin{proposition}\label{prop_maps_rel_P}
    Let $\doctrine{\Cat{C}}{P}$ be an elementary and existential doctrine. Then the maps of   $\mathsf{Rel}(P)$ are precisely the functional and entire elements of $P$.
\end{proposition}

\begin{proposition}\label{thm_adj_BHD_PB}
    Let $P$ be a boolean hyperdoctrine. Then $\REL(P)$ is a peircean bicategory. %
\end{proposition}
\begin{proof} %
By \eqref{adj_CB_EED}, $\REL(P)$ is a cartesian bicategory. Since  $P(X)$ is a boolean algebra for all objects $X$, then each hom-set $\REL(P)[X,Y]$ -- by definition  $P(X\times Y)$ -- is a boolean algebra. To conclude that $\REL(P)$ is a peircean bicategory, it is enough to show that $\eqref{eq:prop mappe}$ holds, that is %
\[\phi\seq[+] \neg \psi=\neg(\phi\seq[+] \psi)\] for all maps $\phi\in \REL(P)[X,Y]$ and arrows $\psi \in \REL(P)[Y,Z]$.
By Proposition~\ref{prop_maps_rel_P}, %
$\phi$ is a functional and entire element of $P$. %
Thus, one can rely on \cref{lemma:boolhypcomp} to conclude that  
\begin{align}
\phi\seq[+] \neg \psi & = \exists_{\pr_{X\times Z}}(P_{\pr_{X\times Y}}(\phi)\wedge P_{\pr_{Y\times Z}}(\neg \psi)) \tag{Defintion of $\REL(P)$} \\
&= \neg(\, \exists_{\pr_{X\times Z}}(P_{\pr_{X\times Y}}(\phi)\wedge P_{\pr_{Y\times Z}}(\psi))\,) \tag{\cref{lemma:boolhypcomp}} \\
&= \neg(\phi\seq[+]  \psi) \tag{Defintion of $\REL(P)$} 
\end{align}
\end{proof}

By \cref{prop_inclusion_PD_BHD,thm_adj_BHD_PB} proving the following result amounts to a few routine checks.
\begin{theorem}\label{thm:main}
The adjunction in \eqref{adj_CB_EED}, restricts to the adjunction below on the left.
\[\begin{tikzcd}
	\PeirceBic && \BHD
	\arrow["\HM"',{name=0, anchor=center, inner sep=0}, curve={height=20pt}, hook, from=1-1, to=1-3] 
	\arrow["\REL"',{name=1, anchor=center, inner sep=0}, curve={height=20pt}, from=1-3, to=1-1]
	\arrow["\dashv"{anchor=center, rotate=-90}, draw=none, from=1, to=0]
\end{tikzcd}
\text{Thus, by \Cref{thm_equiv_FOBic_PeirceBic}, there is an adjunction }
\begin{tikzcd}
	\FOBic && \BHD
	\arrow[""',{name=0, anchor=center, inner sep=0}, curve={height=20pt}, hook, from=1-1, to=1-3] 
	\arrow[""',{name=1, anchor=center, inner sep=0}, curve={height=20pt}, from=1-3, to=1-1]
	\arrow["\dashv"{anchor=center, rotate=-90}, draw=none, from=1, to=0]
\end{tikzcd}\text{.}
\]
\end{theorem}

%% file: sections/equivalence.tex
\section{Boolean Hyperdoctrines Representing First-Order Bicategories}\label{sec:equivalence}\label{sec:the two equivalences}
As anticipated in \S \ref{sec:adjunction}, the adjunction in \eqref{adj_CB_EED} becomes an equivalence for certain well-behaved doctrines. Definitions \ref{def_comp_diagonals} and \ref{def:RUC} state the conditions that such doctrines must satisfy.

\begin{definition}\label{def_comp_diagonals}
An elementary and existential doctrine $\doctrine{\Cat{C}}{P}$ has \emph{comprehensive diagonals} if for the equality predicate $\delta_X\in P(X)$ it holds that $P_{\Delta_X}(\delta_X)=\top_X$ and every arrow $\freccia{Y}{f}{X\times X}$ such that $P_f(\delta_X)=\top_Y$ factors (uniquely) through $\Delta_X$ .%
\end{definition}
Intuitively, a doctrine has comprehensive diagonals if its equality is \emph{extensional}, namely if a formula $t_1=t_2$ is true, then the terms $t_1$ and $t_2$ are syntactically equal. In the language of cartesian bicategories, for two maps $t_1,t_2$, this can be stated by means of diagrams as
\begin{equation}\label{eq:extensionalitydiag}
    \text{if }\seqCirc[+]{t_1}{t_2}[X][X][f][fop] = \topCirc[X][X] \text{ then }\funcCirc[+]{t_1}[X][Y] = \funcCirc[+]{t_2}[X][Y]\text{.}
\end{equation}
While it is sometimes meaningful to consider syntactic doctrines (e.g. \cref{ex:syntacticdoc}) in which the equality is not extensional, in several semantical doctrines this condition is satisfied. %
\begin{definition}\label{def:RUC}
Let $\doctrine{\Cat{C}}{P}$ be an elementary existential doctrine. We say that $P$ satisfies the \emph{Rule of Unique Choice} (RUC) if for every entire functional element $\phi$ in $P(X\times Y)$ there exists an arrow $\freccia{X}{f}{Y}$ such that $\top_X\leq P_{\angbr{\id[+][X]}{f}}{(\phi)}$.
\end{definition}
The reader can think that a doctrine has (RUC) if for every element (intuitively formula) that is entire and functional, there exists an arrow in $\Cat{C}$ (intuitively a term) that represents it.

\begin{example}
   The doctrine $\doctrine{\set}{\Pow}$ has comprehensive diagonals, and  it satisfies the (RUC) (since every functional and total relation can be represented by a function). More generally, every 
 subobject doctrine $\doctrine{\Cat{C}}{{\Sub_{\Cat{C}}}}$ on a regular category, as presented in  \Cref{ex_sub_on_regular} satisfies the (RUC) and it has comprehensive diagonals, as observed in \cite{TECH}.
\end{example}

\begin{example}
The doctrine  $\doctrine{\mathsf{Map}(\Cat{C})}{\Cat{C}[-,I]}$ presented in \Cref{example_cartbicat_provide_eed} satisfies the (RUC) and it has comprehensive diagonals, as proved in \cite{bonchi2021doctrines}. The reader can find a diagrammatic proof of \eqref{eq:extensionalitydiag} in \Cref{prop:func ext} in \cref{app:background}.
\end{example}

Hereafter --and in the equivalence in \eqref{eq:theequivalence}-- $\overline{\EED}$ is the full subcategory of $\EED$ whose objects are doctrines satisfying (RUC) and with comprehensive diagonals. Similarly $\overline{\BHD}$ is the full subcategory of $\BHD$ whose objects are boolean hyperdoctrines satisfying (RUC) and with comprehensive diagonals.

By means of \cref{thm:main}, it is easy to prove that the equivalence in \eqref{eq:theequivalence} restricts as follows.
\begin{theorem}\label{thm:theequivalence}
$\PeirceBic \equiv \overline{\BHD}$ and thus, by \Cref{thm_equiv_FOBic_PeirceBic}, $\FOBic\equiv  \overline{\BHD}$.
\end{theorem}

\section{Comparing Boolean Categories with First-Order Bicategories}\label{sec:tabulation}
In this section we show that boolean categories, described in \cref{ex:booleancat}, correspond to those fo-bicategories that are \emph{functionally complete}, a property introduced in~\cite{carboni1987cartesian}. %
\begin{definition}
A cartesian bicategory $(\Cat{C}, \copier[+], \cocopier[+])$ is \emph{functionally complete} if for every arrow $\freccia{X}{r}{I}$ there exists a map $\freccia{X_r}{i}{X}$, called \emph{tabulation of} $r$, such that %
\[    \seqCirc[+]{i}{i}[X_r][X_r][f][fop]  = \idCirc[+][X_r] \quad\text{ and }\quad \tabulation[+]{i}[X][fop] \;\;= \predicateCirc[+]{r}[X] \; .\]
\end{definition}
\begin{example}
    Let us consider the category $\Relp$. The tabulation of a relation $\freccia{X}{r}{I}$ is given by the subset $X_r\subseteq X$ of those elements of $X$ on which $r$ is defined, together with the trivial inclusion $i\colon X_r\hookrightarrow X$. 
\end{example}
The previous example emphasizes the essential intuition behind the concept of tabulation, namely, that a tabulation represents the ``domain of definition of a relation''. %
A notion aiming to abstract the same concept has been introduced in the context of fibrations in~\cite{CLTT} and, as particular instance, in the context of doctrines in~\cite{QCFF}, under the name of \emph{comprehensions}.

\begin{definition}\label{def_doct_with_comp}
    Let $\doctrine{\Cat{C}}{P}$ be an elementary and existential doctrine and $\alpha$ be an element of $P(X)$.  A \emph{comprehension} of $\alpha$ is an arrow $\freccia{X_{\alpha}}{{\comp{\alpha}}}{X}$ such that $P_{\comp{\alpha}}(\alpha)=\top_{X_{\alpha}}$ and, for every $\freccia{Y}{f}{X}$ such that $P_f(\alpha)=\top_Y$, there exists a unique arrow $\freccia{Y}{g}{X_{\alpha}}$ such that $f=g \seq[+]\comp{\alpha} $.  We say that $P$ \emph{has comprehensions} if every $\alpha$ has a comprehension. We say that $P$ has \emph{full comprehensions} if it has comprehensions and, $\alpha\leq \beta$ whenever $\comp{\alpha}$ factors through $\comp{\beta}$.
    \end{definition}
    In the light of the previous definition, we can rephrase \Cref{def_comp_diagonals}, saying that an elementary doctrine has comprehensive diagonals if $\Delta_X$ is the comprehension of $\delta_X$.
\begin{example}
    Every subobject doctrine $\doctrine{\Cat{C}}{{\Sub_{\Cat{C}}}}$ on a regular category, as presented in  \Cref{ex_sub_on_regular} has full comprehensions, as observed in \cite{TECH}. In this case, the comprehension of a subobject is given by the subobject itself.
\end{example}
   
We prove that the notion of tabulation and that of full comprehension happen to be equivalent when we consider the doctrines associated with a cartesian bicategory.

\begin{theorem}\label{thm_func_comp_cb_iff_comprehensions}
A cartesian bicategory $(\Cat{C}, \copier[+], \cocopier[+])$ is functionally complete if and only if the elementary and existential doctrine $\doctrine{\map(\Cat{C})}{\Cat{C}[-,I]}$ has full comprehensions. 
\end{theorem}

By combining \cref{thm_func_comp_cb_iff_comprehensions} with \cref{thm:theequivalence}  and Proposition 5.3 from \cite{TECH}, that characterizes doctrines satisfying (RUC) with comprehensive diagonals and full comprehensions (see also \Cref{app_section_tablutation_and equivalences}), we obtain equivalences between  $\FOBic_f$, the full subcategory of $\FOBic$ of the functionally complete fo-bicategories, $\overline{\BHD}_c$, the full subcategory of $\overline{\BHD}$ of boolean doctrines with full comprehensions and $\mathbb{BC}$, the category of boolean categories. %

\begin{corollary}\label{cor:final}
    $\FOBic_f \equiv \overline{\BHD}_c \equiv \mathbb{BC}$.
\end{corollary}

%% file: sections/conclusion.tex
\section{Conclusions, Related and Future work}\label{sec:conclusion}

Theorems \ref{thm:main}, \ref{thm:theequivalence} and Corollary \ref{cor:final} provide a solid bridge between functional and relational approaches to classical logic. The former rely on categorical structures that are usually defined by means of exactness properties; the latter on fo-bicategories which enjoy a purely equational presentation, much in the spirit of Boole's  algebra and Peirce's calulus. 

To achieve our result, we found it extremely convenient to introduce the notion of peircean bicategories that, by  \cref{thm_equiv_FOBic_PeirceBic}, provide a far handier characterisation of fo-bicategories.

The isomorphism between fo-bicategories and peircean bicategories might also be useful to establish a correspondence with \emph{allegories} \cite{freyd1990categories}, likely the most influential approach to categorical relational algebra. Since cartesian bicategories are equivalent to unitary pretabular allegories \cite{knijnenburg1994two}, we expect that such allegories where, additionally, homsets carry boolean algebras and the negations satisfy \eqref{eq:prop mappe} are equivalent to fo-bicategories.  Despite searching the literature on allegories, we did not find analogous structures. Interestingly, the property \eqref{eq:prop mappe} can be proven in any Peirce allegories, as shown in Proposition 4.6.1 in~\cite{olivier1997peirce}.

Boolean hyperdoctrines are used in~\cite{brady2000_a-categorical-interpretation-of-c.s.-peirces-propositional-logic-alpha} as a categorical treatment of another work of Peirce: \emph{existential graphs}~\cite{roberts1973_the-existential-graphs-of-charles-s.-peirce}. While the latter share some similarities with the graphical language of fo-bicategories there is one notable difference: negation is a primitive operator rather than a derived one, as it happens for instance also in~\cite{Haydon2020} and \Cref{def:peircean-bicategory}. In~\cite{bonchi2024diagrammatic} and in \S \ref{sec:pb}, it is emphasised how this choice makes the resulting calculus less algebraic in flavour, having to deal with convoluted rules such as the one for (de)iteration or properties which are not purely equational, such as~\eqref{eq:prop mappe}.
Inspired by~\cite{brady2000_a-categorical-interpretation-of-c.s.-peirces-propositional-logic-alpha}, another graphical language~\cite{mellies2016bifibrational} akin to Peirce's graphs is based on a decomposition of a hyperdoctrine into a bifibration. In this work, the categorical treatment revolves around the notion of monoidal \emph{chiralities}~\cite{mellies2016dialogue}, which are much more closer in spirit to fo-bicategories. We believe that our results might set an initial step towards a connection between fo-bicategories and chiralities.

A recent work~\cite{dagnino_et_al:LIPIcs.FSCD.2023.25} proposes a relational understanding of doctrines. However, these corresponds to the regular fragment of first-order logic, and thus it might by intriguing to understand the role of the additional black structure of first-order bicategories in this setting.
Finally, it is also worth mentioning a closely related research line, exemplified by works such as \cite{hyland04,PYM_2007}, along with the references therein. These work primarily focus on the categorical approach to classical proof theory, involving suitable variants of poset-enriched categories.

As future work we also aim to investigate how our characterizations can be extended to higher-order classical logic, which is categorically represented through the notion of \emph{tripos}~\cite{TT,TTT}. Indeed, we believe that the constructions and results presented in this work, together with notion of tripos, can serve as a guide for defining a variant of fo-bicategories --hopefully, purely equational-- capable of representing higher-order classical logic.

An important result in the theory of databases~\cite{chandra1977optimal} shows that the problem of query inclusion (entailment of regular-logic formulas) is equivalent to the existence of  morphisms of hypergraphs. This combinatorial characterisation found a neat algebraic understanding in~\cite{GCQ} by means of cartesian bicategories. We hope that~\cref{thm_equiv_FOBic_PeirceBic} may lead to an analogous combinatorial understanding of first-order logic.

%% file: sections/appcb.tex
\section{Appendix to Section~\ref{sec:background}}\label{app:background}

In this appendix we collect some useful results about (co)cartesian bicategories. Moreover, we report in diagrammatic form the axioms of cocartesian bicategories (Figure~\ref{fig:cocb axioms}) and we summarise  in Table~\ref{table:daggerproperties} the properites of the isomorphism $\op{(\cdot)}\colon \Cat{C} \to \opposite{\Cat{C}}$.

\begin{figure}[t]
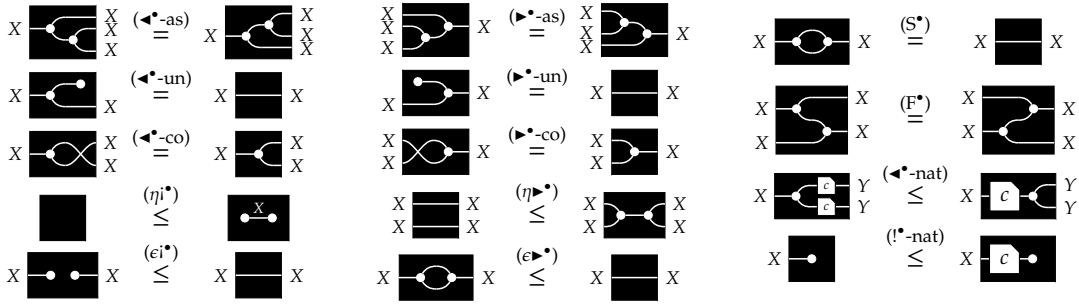

    \mylabel{ax:comMinusAssoc}{\ensuremath{\copier[-]}\text{-as}}
    \mylabel{ax:comMinusUnit}{\ensuremath{\copier[-]}\text{-un}}
    \mylabel{ax:comMinusComm}{\ensuremath{\copier[-]}\text{-co}}
    \mylabel{ax:monMinusAssoc}{\ensuremath{\cocopier[-]}\text{-as}}
    \mylabel{ax:monMinusUnit}{\ensuremath{\cocopier[-]}\text{-un}}
    \mylabel{ax:monMinusComm}{\ensuremath{\cocopier[-]}\text{-co}}
    \mylabel{ax:minusSpecFrob}{\text{S}\ensuremath{^\bullet}}
    \mylabel{ax:minusFrob}{\text{F}\ensuremath{^\bullet}}
    \mylabel{ax:comMinusLaxNat}{\ensuremath{\copier[-]}\text{-nat}}
    \mylabel{ax:discMinusLaxNat}{\ensuremath{\discard[-]}\text{-nat}}
    \mylabel{ax:minusCodiscDisc}{\ensuremath{\eta\codiscard[-]}}
    \mylabel{ax:minusDiscCodisc}{\ensuremath{\epsilon\codiscard[-]}}
    \mylabel{ax:minusCocopyCopy}{\ensuremath{\eta\cocopier[-]}}
    \mylabel{ax:minusCopyCocopy}{\ensuremath{\epsilon\cocopier[-]}}
    \[
            \begin{array}{@{}c c c c c@{}}
            \begin{array}{@{}c@{}c@{}c@{}}
                    
    \InputIfFileExists{axiomsNEW/cb/minus/comAssoc1.tikz}{}{\input{tikz/axiomsNEW/cb/minus/comAssoc1.tikz}}
 & \Leq{\ref*{ax:comMinusAssoc}}    & 
    \InputIfFileExists{axiomsNEW/cb/minus/comAssoc2.tikz}{}{\input{tikz/axiomsNEW/cb/minus/comAssoc2.tikz}}
 \\
                    
    \InputIfFileExists{axiomsNEW/cb/minus/comUnit.tikz}{}{\input{tikz/axiomsNEW/cb/minus/comUnit.tikz}}
   & \Leq{\ref*{ax:comMinusUnit}}     & \idCirc[-][X] \\
                    
    \InputIfFileExists{axiomsNEW/cb/minus/comComm.tikz}{}{\input{tikz/axiomsNEW/cb/minus/comComm.tikz}}
   & \Leq{\ref*{ax:comMinusComm}}     & \copierCirc[-][X] \\
                    \emptyCirc[-]                          & \Lleq{\ref*{ax:minusCodiscDisc}} & 
    \InputIfFileExists{axiomsNEW/cb/minus/codiscDisc.tikz}{}{\input{tikz/axiomsNEW/cb/minus/codiscDisc.tikz}}
 \\
                    
    \InputIfFileExists{axiomsNEW/bottom.tikz}{}{\input{tikz/axiomsNEW/bottom.tikz}}
             & \Lleq{\ref*{ax:minusDiscCodisc}} & \idCirc[-][X]
            \end{array} & \!\!\!\! &
            \begin{array}{@{}c@{}c@{}c@{}}
                    
    \InputIfFileExists{axiomsNEW/cb/minus/monAssoc1.tikz}{}{\input{tikz/axiomsNEW/cb/minus/monAssoc1.tikz}}
 & \Leq{\ref*{ax:monMinusAssoc}}    & 
    \InputIfFileExists{axiomsNEW/cb/minus/monAssoc2.tikz}{}{\input{tikz/axiomsNEW/cb/minus/monAssoc2.tikz}}
 \\
                    
    \InputIfFileExists{axiomsNEW/cb/minus/monUnit.tikz}{}{\input{tikz/axiomsNEW/cb/minus/monUnit.tikz}}
   & \Leq{\ref*{ax:monMinusUnit}}     & \idCirc[-][X] \\
                    
    \InputIfFileExists{axiomsNEW/cb/minus/monComm.tikz}{}{\input{tikz/axiomsNEW/cb/minus/monComm.tikz}}
   & \Leq{\ref*{ax:monMinusComm}}     & \cocopierCirc[-][X] \\
                    
    \InputIfFileExists{axiomsNEW/id2M.tikz}{}{\input{tikz/axiomsNEW/id2M.tikz}}
               & \Lleq{\ref*{ax:minusCocopyCopy}} & 
    \InputIfFileExists{axiomsNEW/cb/minus/cocopierCopier.tikz}{}{\input{tikz/axiomsNEW/cb/minus/cocopierCopier.tikz}}
 \\
                    
    \InputIfFileExists{axiomsNEW/cb/minus/specFrob.tikz}{}{\input{tikz/axiomsNEW/cb/minus/specFrob.tikz}}
  & \Lleq{\ref*{ax:minusCopyCocopy}} & \idCirc[-][X]
            \end{array} & \!\!\!\! &
            \begin{array}{@{}c@{}c@{}c@{}}
                    
    \InputIfFileExists{axiomsNEW/cb/minus/specFrob.tikz}{}{\input{tikz/axiomsNEW/cb/minus/specFrob.tikz}}
      & \Leq{\ref*{ax:minusSpecFrob}}    & \idCirc[-][X] \\[8pt]
                    
    \InputIfFileExists{axiomsNEW/cb/minus/frob1.tikz}{}{\input{tikz/axiomsNEW/cb/minus/frob1.tikz}}
         & \Leq{\ref*{ax:minusFrob}}        & 
    \InputIfFileExists{axiomsNEW/cb/minus/frob2.tikz}{}{\input{tikz/axiomsNEW/cb/minus/frob2.tikz}}
 \\[10pt]
                    
    \InputIfFileExists{axiomsNEW/cb/minus/copierLaxNat2.tikz}{}{\input{tikz/axiomsNEW/cb/minus/copierLaxNat2.tikz}}
 & \Lleq{\ref*{ax:comMinusLaxNat}}  & 
    \InputIfFileExists{axiomsNEW/cb/minus/copierLaxNat1.tikz}{}{\input{tikz/axiomsNEW/cb/minus/copierLaxNat1.tikz}}
 \\
                    \discardCirc[-][X][X]                      & \Lleq{\ref*{ax:discMinusLaxNat}} & 
    \InputIfFileExists{axiomsNEW/cb/minus/discardLaxNat.tikz}{}{\input{tikz/axiomsNEW/cb/minus/discardLaxNat.tikz}}

            \end{array}
            \end{array}
    \]
    \caption{Axioms of Cocartesian bicategories}\label{fig:cocb axioms}
\end{figure}

\begin{theorem}\label{th:spider}
    Any \emph{connected} diagram $c \colon X^n \to X^m$ made out of $\id[+], \symm[+], \copier[+], \discard[+], \cocopier[+]$ and $\codiscard[+]$ is equal to $\;\; n \left\{ 
    \InputIfFileExists{spider.tikz}{}{\input{tikz/spider.tikz}}
 \right\} m \;$.
\end{theorem}
\begin{proof}
    See~\cite{Lack2004a,Coecke2008}.
\end{proof}

\begin{remark}
    Theorem~\ref{th:spider} is known as the \emph{spider theorem} and it holds in any special Frobenius algebra and thus, in particular, in a cartesian bicategory. A direct consequence of the spider theorem is that we can ``rewire'' parts of a diagram involving the (co)monoid structures, as long as the connectivity is preserved. This is useful in several graphical derivations for arranging diagrams into a desired shape.
\end{remark}

\begin{lemma}\label{prop:cap property}
    For any $c,d \colon X \to Y$, $\!\!\boxCirc[+]{c}[X][Y] \!\!\!=\!\!\! \boxCirc[+]{d}[X][Y]\!\!$ if and only if $\cappedCirc[+]{c}[X][Y] = \cappedCirc[+]{d}[X][Y]$.
\end{lemma}
\begin{proof}
    One direction is trivial. The other direction follows from the Frobenius axioms.
\end{proof}

\begin{proposition}\label{prop:wrong way}
    For any $c \colon X \to Y$ the following inequality holds $
    \InputIfFileExists{wrongway/lhs.tikz}{}{\input{tikz/wrongway/lhs.tikz}}
 \leq 
    \InputIfFileExists{wrongway/rhs.tikz}{}{\input{tikz/wrongway/rhs.tikz}}
$.
\end{proposition}
\begin{proof}
    See Lemma 4.3 in~\cite{Bonchi2017c}.
\end{proof}

\begin{proposition}\label{prop:sliding through cap}
    For any $c \colon X \to Y$ the following equality holds $
    \InputIfFileExists{capslide/lhs.tikz}{}{\input{tikz/capslide/lhs.tikz}}
 = 
    \InputIfFileExists{capslide/rhs.tikz}{}{\input{tikz/capslide/rhs.tikz}}
$.
\end{proposition}
\begin{proof}
    \[ 
    \InputIfFileExists{capslide/rhs.tikz}{}{\input{tikz/capslide/rhs.tikz}}
 \stackrel{\eqref{eqdagger}}{=} 
    \begin{tikzpicture}
        \begin{pgfonlayer}{nodelayer}
            \node [style={dotStyle/+}] (135) at (4.25, 0.1) {$$};
            \node [style=none] (136) at (4.75, -1.625) {};
            \node [style=none] (137) at (4.75, 1.6) {};
            \node [style=none] (138) at (-1.25, 1.6) {};
            \node [style=none] (139) at (-1.25, -1.625) {};
            \node [style={dotStyle/+}] (144) at (3.75, 0.1) {};
            \node [style=none] (145) at (2.25, -1.15) {};
            \node [style=none] (146) at (-1.25, -1.15) {};
            \node [style=none] (147) at (-1.25, -0.65) {};
            \node [style=none] (150) at (2.25, 1.35) {};
            \node [style=label] (151) at (-1.725, -0.65) {$Y$};
            \node [style=label] (152) at (-1.725, -1.15) {$X$};
            \node [{boxStyle/+}] (153) at (0.75, 0.35) {$c$};
            \node [style={dotStyle/+}] (154) at (2.25, -0.15) {$$};
            \node [style={dotStyle/+}] (155) at (1.75, -0.15) {};
            \node [style=none] (156) at (1, -0.65) {};
            \node [style=none] (159) at (1, 0.35) {};
            \node [style={dotStyle/+}] (160) at (-0.75, 0.85) {$$};
            \node [style={dotStyle/+}] (161) at (-0.25, 0.85) {};
            \node [style=none] (162) at (0.5, 0.35) {};
            \node [style=none] (163) at (0.5, 1.35) {};
            \node [style=none] (164) at (2.25, 1.35) {};
        \end{pgfonlayer}
        \begin{pgfonlayer}{edgelayer}
            \draw [{bgStyle/+}] (138.center)
                 to (137.center)
                 to (136.center)
                 to (139.center)
                 to cycle;
            \draw [style={wStyle/+}, bend right] (145.center) to (144);
            \draw [style={wStyle/+}] (144) to (135);
            \draw [style={wStyle/+}, bend left] (150.center) to (144);
            \draw [style={wStyle/+}] (146.center) to (145.center);
            \draw [style={wStyle/+}, bend right] (156.center) to (155);
            \draw [style={wStyle/+}] (155) to (154);
            \draw [style={wStyle/+}] (153) to (159.center);
            \draw [style={wStyle/+}, bend left] (159.center) to (155);
            \draw [style={wStyle/+}, bend left] (162.center) to (161);
            \draw [style={wStyle/+}] (161) to (160);
            \draw [style={wStyle/+}, bend right] (163.center) to (161);
            \draw [style={wStyle/+}] (153) to (162.center);
            \draw [style={wStyle/+}] (163.center) to (164.center);
            \draw [style={wStyle/+}] (147.center) to (156.center);
            \draw [style={wStyle/+}] (164.center) to (150.center);
        \end{pgfonlayer}
    \end{tikzpicture}
    \stackrel{\text{\Cref{th:spider}}}{=} 
    \begin{tikzpicture}
        \begin{pgfonlayer}{nodelayer}
            \node [{boxStyle/+}] (134) at (0, 0.475) {$c$};
            \node [style={dotStyle/+}] (135) at (1.5, -0.025) {$$};
            \node [style=none] (136) at (2, -1) {};
            \node [style=none] (137) at (2, 1.225) {};
            \node [style=none] (138) at (-1.25, 1.225) {};
            \node [style=none] (139) at (-1.25, -1) {};
            \node [style={dotStyle/+}] (144) at (1, -0.025) {};
            \node [style=none] (145) at (0.25, -0.525) {};
            \node [style=none] (146) at (-1.25, -0.525) {};
            \node [style=none] (147) at (-1.25, 0.475) {};
            \node [style=none] (150) at (0.25, 0.475) {};
            \node [style=label] (151) at (-1.725, 0.475) {$Y$};
            \node [style=label] (152) at (-1.725, -0.525) {$X$};
            \node [style=none] (153) at (-0.25, -0.525) {};
        \end{pgfonlayer}
        \begin{pgfonlayer}{edgelayer}
            \draw [{bgStyle/+}] (138.center)
                 to (137.center)
                 to (136.center)
                 to (139.center)
                 to cycle;
            \draw [style={wStyle/+}, bend right] (145.center) to (144);
            \draw [style={wStyle/+}] (144) to (135);
            \draw [style={wStyle/+}] (134) to (150.center);
            \draw [style={wStyle/+}, bend left] (150.center) to (144);
            \draw [style={wStyle/+}] (145.center) to (153.center);
            \draw [style={wStyle/+}, in=0, out=-180] (153.center) to (147.center);
            \draw [style={wStyle/+}, in=0, out=180] (134) to (146.center);
        \end{pgfonlayer}
    \end{tikzpicture}    
    \structuralcong
    \begin{tikzpicture}
        \begin{pgfonlayer}{nodelayer}
            \node [{boxStyle/+}] (134) at (-0.5, -0.275) {$c$};
            \node [style={dotStyle/+}] (135) at (1.75, 0.225) {$$};
            \node [style=none] (136) at (2, -1) {};
            \node [style=none] (137) at (2, 1.225) {};
            \node [style=none] (138) at (-1.25, 1.225) {};
            \node [style=none] (139) at (-1.25, -1) {};
            \node [style={dotStyle/+}] (144) at (1.25, 0.225) {};
            \node [style=none] (145) at (0.75, -0.275) {};
            \node [style=none] (146) at (-1.25, -0.275) {};
            \node [style=none] (147) at (-0.25, 0.725) {};
            \node [style=none] (150) at (0.75, 0.725) {};
            \node [style=label] (151) at (-1.725, 0.725) {$Y$};
            \node [style=label] (152) at (-1.725, -0.275) {$X$};
            \node [style=none] (153) at (-1.25, 0.725) {};
        \end{pgfonlayer}
        \begin{pgfonlayer}{edgelayer}
            \draw [{bgStyle/+}] (138.center)
                 to (137.center)
                 to (136.center)
                 to (139.center)
                 to cycle;
            \draw [style={wStyle/+}, bend right=45] (145.center) to (144);
            \draw [style={wStyle/+}] (144) to (135);
            \draw [style={wStyle/+}, bend left=45] (150.center) to (144);
            \draw [style={wStyle/+}] (146.center) to (134);
            \draw [style={wStyle/+}, in=-180, out=0, looseness=0.75] (147.center) to (145.center);
            \draw [style={wStyle/+}, in=180, out=0] (134) to (150.center);
            \draw [style={wStyle/+}] (153.center) to (147.center);
        \end{pgfonlayer}
    \end{tikzpicture}    
    \stackrel{\eqref{ax:monPlusAssoc}}{=}    
    
    \InputIfFileExists{capslide/lhs.tikz}{}{\input{tikz/capslide/lhs.tikz}}
 \]
\end{proof}

\begin{table}[!htb]
    \caption{Properties of $\op{(\cdot)} \colon \Cat{C} \to \opposite{\Cat{C}}$}\label{table:daggerproperties}
    \hspace*{-1.2em}
    \centering
    \begin{tabular}{c}
    $
    \begin{array}{cccc}
        \toprule
        \multicolumn{2}{c}{
            \text{if }c\leq d\text{ then }\op{c} \leq \op{d}
        }
        &
        \multicolumn{2}{c}{
            \op{(\op{c})}= c
        }
        \\
    \op{(c \seq[+] d)} = \op{d} \seq[+] \op{c}
    &\op{(\id[+][X])}=\id[+][X]
    &\op{(\cocopier[+][X])}= \copier[+][X]
    &\op{(\codiscard[+][X])}= \discard[+][X]
    \\
    \op{(c \tensor[+] d)} = \op{c} \tensor[+] \op{d}
    &\op{(\symm[+][X][Y])} = \symm[+][Y][X]
    &\op{(\copier[+][X])}= \cocopier[+][X]
    &\op{(\discard[+][X])}= \codiscard[+][X]
    \\
    \multicolumn{2}{c}{
        \op{(c \wedge d)} = \op{c} \wedge \op{d}        
    }
    &
    \multicolumn{2}{c}{
        \op{\top} = \top       
    }
    \\
    \bottomrule
    \end{array}
    $
    \end{tabular}
    \end{table}

    \begin{proposition}[Extensional equality]\label{prop:func ext}
        For any $t_1, t_2 \colon X \to Y$ maps,
        \begin{center}
            if $ \seqCirc[+]{t_1}{t_2}[X][X][f][fop] = \topCirc[X][X]$ then $\funcCirc[+]{t_1}[X][Y] = \funcCirc[+]{t_2}[X][Y]$.
        \end{center}
    \end{proposition}
    \begin{proof}
        First observe that if $ \seqCirc[+]{t_1}{t_2}[X][X][f][fop] = \topCirc[X][X]$, then by the properties in Table~\ref{table:daggerproperties} the following holds
        \begin{equation}\label{ipotesi dagger}
            \seqCirc[+]{t_2}{t_1}[X][X][f][fop] = \op{(\seqCirc[+]{t_1}{t_2}[X][X][f][fop])} = \op{(\topCirc[X][X])} = \topCirc[X][X].
        \end{equation}
        To conclude we show the two inclusions separately:
        
        \noindent
        \begin{minipage}{0.48\textwidth}
            \input{tikz/funcExt/inclusion1.tex}
        \end{minipage}
        \quad
        \vline
        \begin{minipage}{0.48\textwidth}
            \input{tikz/funcExt/inclusion2.tex}
        \end{minipage}
    \end{proof}

%% file: tikz/wrongway/lhs.tikz
\begin{tikzpicture}
	\begin{pgfonlayer}{nodelayer}
		\node [{boxStyle/+}] (134) at (-0.5, 0.5) {$c$};
		\node [style=none] (135) at (1.25, 0) {};
		\node [style=none] (136) at (1.25, 1.225) {};
		\node [style=none] (137) at (1.25, -1) {};
		\node [style=none] (138) at (-1.25, -1) {};
		\node [style=none] (139) at (-1.25, 1.225) {};
		\node [style={dotStyle/+}] (144) at (0.5, 0) {};
		\node [style=none] (145) at (-0.25, -0.5) {};
		\node [style=none] (146) at (-1.25, 0.5) {};
		\node [style=none] (147) at (-1.25, -0.5) {};
		\node [style=none] (150) at (-0.25, 0.5) {};
	\end{pgfonlayer}
	\begin{pgfonlayer}{edgelayer}
		\draw [{bgStyle/+}] (138.center)
			 to (137.center)
			 to (136.center)
			 to (139.center)
			 to cycle;
		\draw [style={wStyle/+}, bend right] (145.center) to (144);
		\draw [style={wStyle/+}] (144) to (135.center);
		\draw [style={wStyle/+}] (134) to (146.center);
		\draw [style={wStyle/+}] (147.center) to (145.center);
		\draw [style={wStyle/+}] (134) to (150.center);
		\draw [style={wStyle/+}, bend left] (150.center) to (144);
	\end{pgfonlayer}
\end{tikzpicture}

%% file: tikz/wrongway/rhs.tikz
\begin{tikzpicture}
	\begin{pgfonlayer}{nodelayer}
		\node [{boxOpStyle/+}] (134) at (-3, -0.25) {$c$};
		\node [style={boxStyle/+}] (135) at (-1, 0.25) {$c$};
		\node [style=none] (136) at (-0.25, 1.225) {};
		\node [style=none] (137) at (-0.25, -1) {};
		\node [style=none] (138) at (-3.75, -1) {};
		\node [style=none] (139) at (-3.75, 1.225) {};
		\node [style={dotStyle/+}] (144) at (-2, 0.25) {};
		\node [style=none] (145) at (-2.75, 0.75) {};
		\node [style=none] (146) at (-3.75, 0.75) {};
		\node [style=none] (147) at (-3.75, -0.25) {};
		\node [style=none] (150) at (-2.75, -0.25) {};
		\node [style=none] (152) at (-0.25, 0.25) {};
	\end{pgfonlayer}
	\begin{pgfonlayer}{edgelayer}
		\draw [{bgStyle/+}] (138.center)
			 to (137.center)
			 to (136.center)
			 to (139.center)
			 to cycle;
		\draw [style={wStyle/+}, bend left] (145.center) to (144);
		\draw [style={wStyle/+}] (144) to (135);
		\draw [style={wStyle/+}] (134) to (150.center);
		\draw [style={wStyle/+}, bend right] (150.center) to (144);
		\draw [style={wStyle/+}] (134) to (147.center);
		\draw [style={wStyle/+}] (146.center) to (145.center);
		\draw [style={wStyle/+}] (135) to (152.center);
	\end{pgfonlayer}
\end{tikzpicture}

%% file: tikz/capslide/lhs.tikz
\begin{tikzpicture}
	\begin{pgfonlayer}{nodelayer}
		\node [{boxStyle/+}] (134) at (-0.5, -0.25) {$c$};
		\node [style={dotStyle/+}] (135) at (1, 0.25) {$$};
		\node [style=none] (136) at (1.5, 1.225) {};
		\node [style=none] (137) at (1.5, -1) {};
		\node [style=none] (138) at (-1.25, -1) {};
		\node [style=none] (139) at (-1.25, 1.225) {};
		\node [style={dotStyle/+}] (144) at (0.5, 0.25) {};
		\node [style=none] (145) at (-0.25, 0.75) {};
		\node [style=none] (146) at (-1.25, 0.75) {};
		\node [style=none] (147) at (-1.25, -0.25) {};
		\node [style=none] (150) at (-0.25, -0.25) {};
		\node [style=label] (151) at (-1.725, -0.25) {$X$};
		\node [style=label] (152) at (-1.725, 0.75) {$Y$};
	\end{pgfonlayer}
	\begin{pgfonlayer}{edgelayer}
		\draw [{bgStyle/+}] (138.center)
			 to (137.center)
			 to (136.center)
			 to (139.center)
			 to cycle;
		\draw [style={wStyle/+}, bend left] (145.center) to (144);
		\draw [style={wStyle/+}] (144) to (135);
		\draw [style={wStyle/+}] (134) to (150.center);
		\draw [style={wStyle/+}, bend right] (150.center) to (144);
		\draw [style={wStyle/+}] (134) to (147.center);
		\draw [style={wStyle/+}] (146.center) to (145.center);
	\end{pgfonlayer}
\end{tikzpicture}

%% file: tikz/capslide/rhs.tikz
\begin{tikzpicture}
	\begin{pgfonlayer}{nodelayer}
		\node [{boxOpStyle/+}] (134) at (-0.5, 0.475) {$c$};
		\node [style={dotStyle/+}] (135) at (1, -0.025) {$$};
		\node [style=none] (136) at (1.5, -1) {};
		\node [style=none] (137) at (1.5, 1.225) {};
		\node [style=none] (138) at (-1.25, 1.225) {};
		\node [style=none] (139) at (-1.25, -1) {};
		\node [style={dotStyle/+}] (144) at (0.5, -0.025) {};
		\node [style=none] (145) at (-0.25, -0.525) {};
		\node [style=none] (146) at (-1.25, -0.525) {};
		\node [style=none] (147) at (-1.25, 0.475) {};
		\node [style=none] (150) at (-0.25, 0.475) {};
		\node [style=label] (151) at (-1.725, 0.475) {$Y$};
		\node [style=label] (152) at (-1.725, -0.525) {$X$};
	\end{pgfonlayer}
	\begin{pgfonlayer}{edgelayer}
		\draw [{bgStyle/+}] (138.center)
			 to (137.center)
			 to (136.center)
			 to (139.center)
			 to cycle;
		\draw [style={wStyle/+}, bend right] (145.center) to (144);
		\draw [style={wStyle/+}] (144) to (135);
		\draw [style={wStyle/+}] (134) to (150.center);
		\draw [style={wStyle/+}, bend left] (150.center) to (144);
		\draw [style={wStyle/+}] (134) to (147.center);
		\draw [style={wStyle/+}] (146.center) to (145.center);
	\end{pgfonlayer}
\end{tikzpicture}

%% file: tikz/funcExt/inclusion1.tex
\begin{align*}
    \funcCirc[+]{t_1}[X][Y] &\leq 
    \begin{tikzpicture}
        \begin{pgfonlayer}{nodelayer}
            \node [style=none] (112) at (-2.5, 0.75) {};
            \node [style=none] (113) at (-2.5, -0.75) {};
            \node [style=none] (114) at (2.5, -0.75) {};
            \node [style=none] (115) at (2.5, 0.75) {};
            \node [style=none] (121) at (2.5, 0) {};
            \node [{funcStyle/+}] (123) at (1.5, 0) {$t_1$};
            \node [style=label] (124) at (-3, 0) {$X$};
            \node [style=label] (125) at (3, 0) {$Y$};
            \node [style=none] (126) at (-2.5, 0) {};
            \node [{dotStyle/+}] (127) at (-1.25, 0) {};
            \node [{dotStyle/+}] (128) at (-0.25, 0) {};
        \end{pgfonlayer}
        \begin{pgfonlayer}{edgelayer}
            \draw [{bgStyle/+}] (114.center)
                 to (113.center)
                 to (112.center)
                 to (115.center)
                 to cycle;
            \draw [{wStyle/+}] (123) to (121.center);
            \draw [style={wStyle/+}] (127) to (126.center);
            \draw [style={wStyle/+}] (128) to (123);
        \end{pgfonlayer}
    \end{tikzpicture}  \tag{\ref{ax:plusDiscCodisc}}    \\
    &= 
    \begin{tikzpicture}
        \begin{pgfonlayer}{nodelayer}
            \node [style=none] (112) at (-2.5, 0.75) {};
            \node [style=none] (113) at (-2.5, -0.75) {};
            \node [style=none] (114) at (2.5, -0.75) {};
            \node [style=none] (115) at (2.5, 0.75) {};
            \node [style=none] (121) at (2.5, 0) {};
            \node [{funcStyle/+}] (123) at (1.5, 0) {$t_1$};
            \node [style=label] (124) at (-3, 0) {$X$};
            \node [style=label] (125) at (3, 0) {$Y$};
            \node [style=none] (126) at (-2.5, 0) {};
            \node [{funcStyle/+}] (127) at (-1.5, 0) {$t_2$};
            \node [{funcOpStyle/+}] (128) at (0, 0) {$t_1$};
        \end{pgfonlayer}
        \begin{pgfonlayer}{edgelayer}
            \draw [{bgStyle/+}] (114.center)
                 to (113.center)
                 to (112.center)
                 to (115.center)
                 to cycle;
            \draw [{wStyle/+}] (123) to (121.center);
            \draw [style={wStyle/+}] (128) to (127);
            \draw [style={wStyle/+}] (127) to (126.center);
            \draw [style={wStyle/+}] (128) to (123);
        \end{pgfonlayer}
    \end{tikzpicture}   \tag{\ref{ipotesi dagger}} \\
    &\leq 
    \funcCirc[+]{t_2}[X][Y] \tag{Prop. \ref{prop:map adj}.3}
\end{align*}

%% file: tikz/funcExt/inclusion2.tex
\begin{align*}
    \funcCirc[+]{t_2}[X][Y] &\leq 
    \begin{tikzpicture}
        \begin{pgfonlayer}{nodelayer}
            \node [style=none] (112) at (-2.5, 0.75) {};
            \node [style=none] (113) at (-2.5, -0.75) {};
            \node [style=none] (114) at (2.5, -0.75) {};
            \node [style=none] (115) at (2.5, 0.75) {};
            \node [style=none] (121) at (2.5, 0) {};
            \node [{funcStyle/+}] (123) at (1.5, 0) {$t_2$};
            \node [style=label] (124) at (-3, 0) {$X$};
            \node [style=label] (125) at (3, 0) {$Y$};
            \node [style=none] (126) at (-2.5, 0) {};
            \node [{dotStyle/+}] (127) at (-1.25, 0) {};
            \node [{dotStyle/+}] (128) at (-0.25, 0) {};
        \end{pgfonlayer}
        \begin{pgfonlayer}{edgelayer}
            \draw [{bgStyle/+}] (114.center)
                 to (113.center)
                 to (112.center)
                 to (115.center)
                 to cycle;
            \draw [{wStyle/+}] (123) to (121.center);
            \draw [style={wStyle/+}] (127) to (126.center);
            \draw [style={wStyle/+}] (128) to (123);
        \end{pgfonlayer}
    \end{tikzpicture}  \tag{\ref{ax:plusDiscCodisc}}    \\
    &= 
    \begin{tikzpicture}
        \begin{pgfonlayer}{nodelayer}
            \node [style=none] (112) at (-2.5, 0.75) {};
            \node [style=none] (113) at (-2.5, -0.75) {};
            \node [style=none] (114) at (2.5, -0.75) {};
            \node [style=none] (115) at (2.5, 0.75) {};
            \node [style=none] (121) at (2.5, 0) {};
            \node [{funcStyle/+}] (123) at (1.5, 0) {$t_2$};
            \node [style=label] (124) at (-3, 0) {$X$};
            \node [style=label] (125) at (3, 0) {$Y$};
            \node [style=none] (126) at (-2.5, 0) {};
            \node [{funcStyle/+}] (127) at (-1.5, 0) {$t_1$};
            \node [{funcOpStyle/+}] (128) at (0, 0) {$t_2$};
        \end{pgfonlayer}
        \begin{pgfonlayer}{edgelayer}
            \draw [{bgStyle/+}] (114.center)
                 to (113.center)
                 to (112.center)
                 to (115.center)
                 to cycle;
            \draw [{wStyle/+}] (123) to (121.center);
            \draw [style={wStyle/+}] (128) to (127);
            \draw [style={wStyle/+}] (127) to (126.center);
            \draw [style={wStyle/+}] (128) to (123);
        \end{pgfonlayer}
    \end{tikzpicture}   \tag{Hyp.} \\
    &\leq 
    \funcCirc[+]{t_1}[X][Y] \tag{Prop. \ref{prop:map adj}.3}
\end{align*}

%% file: sections/apphyp.tex
\section{Appendix to Section \ref{sec:hyp}}\label{app:hyp}

\subsection{A few properties of Boolean hyperdoctrines}\label{app:boolhyp}
In this section we recall some useful properties of boolean hyperdoctrines, and we prove a lemma which will be crucial for establishing the precise connection with peircean bicategories. All the results we are going to show here are quite natural from the perspective of first-order classical logic, and their proofs are straightforward. 

First, it is well-known in first-order classical logic that the universal quantifier can be defined combining the existential quantifier with the negation. In the following lemma we provide a proof using the language of f.o. boolean hyperdoctrine of this fact.
\begin{lemma}\label{lem_bollean_hd_has_right_adjoints}
        Let $\hyperdoctrine{\Cat{C}}{P}$ be a boolean hyperdoctrine. Then for every arrow $\freccia{A}{f}{B}$, the functor 
        \[\forall_f(-):=\neg \exists_f \neg (-)\]
        provides a right adjoint to $P_f$. Moreover, if $\exists_f$ satisfies BCC then also $\forall_f$ satisfies BCC.
 \end{lemma}
    \begin{proof}
        The proof is a straightforward generalization of the ordinary proof in first-order classical logic. Indeed:
        \[\alpha \leq \forall_fP_f(\alpha)=\neg \exists_f\neg P_f(\alpha)\]
        holds if and only if 
        \[\exists_f \neg P_f(\alpha)\leq \neg \alpha\]
        because $P$ is boolean. But this holds because it is equivalent to
        \[\ P_f(\neg\alpha)\leq  P_f(\neg \alpha).\]
        To prove that 
        \[P_f\forall_f (\beta )\leq \beta\]
        we have to use again the assumption that $P$ is boolean. In fact we have that 
        \[P_f\forall_f (\beta ) =\neg P_f\exists_f (\neg \beta)\leq \neg \neg \beta\]
        because $\neg \beta \leq P_f\exists_f (\neg \beta)$ (since $\exists_f\dashv P_f$), and using the fact that $\neg \neg \beta=\beta$ we can conclude that $P_f\forall_f (\beta )\leq \beta$.
        Now we prove that if $\exists_f$ satisfies BCC then also $\forall_f$ satisfies BCC. So let us consider a pullback%
        \[\begin{tikzcd}
            D & C \\
            A & B
            \arrow["{g'}"', from=1-1, to=2-1]
            \arrow["f"', from=2-1, to=2-2]
            \arrow["g", from=1-2, to=2-2]
            \arrow["{f'}", from=1-1, to=1-2]
            \arrow["\lrcorner"{anchor=center, pos=0.125}, draw=none, from=1-1, to=2-2]
        \end{tikzcd}\]
        and suppose that $P_g\exists_f(\alpha)=\exists_{f'}P_{g'}(\alpha)$ for every $\alpha$ in $P(A)$. From this we can deduce that $\neg P_g\exists_f(\alpha)= P_g\neg \exists_f(\alpha)=\neg \exists_{f'}P_{g'}(\alpha)$, and then that  $P_g\forall_f(\neg \alpha)=\forall_{f'}P_{g'}(\neg \alpha)$ for every $\alpha$ element of $ P(A)$. Therefore, in particular it  holds for the element $\beta:=\neg \alpha$, and hence (using the fact that $P$ is boolean) we can conclude that 
        $P_g\forall_f(\alpha)=\forall_{f'}P_{g'}(\alpha)$.
    \end{proof}
\begin{remark}[Frobenius reciprocity]
    Employing the preservation of the implication $\to$ by the functor $P_f$, it is straightforward to check that every first-order boolean hyperdoctrine $\hyperdoctrine{\Cat{C}}{P}$ satisfies the so-called \emph{Frobenius reciprocity} (FR), namely:
    \[ \exists_f(P_f(\alpha)\wedge \beta)= \alpha \wedge \exists_f (\beta)\mbox{ and }  \forall_f(P_f(\alpha)\rightarrow \beta)= \alpha \rightarrow \forall_f (\beta)\]
    for every  morphism $\freccia{X}{f}{Y}$, $\alpha$ in $P(Y)$ and $\beta$ in $P (X)$. See \cite[Rem. 1.3]{TT}.
    However, it is not guarantee that of Beck-Chevalley conditions with respect to pullbacks along $f$. See \cite{MaiettiTrotta21} for more details.
\end{remark}

\begin{lemma}
    Let $\hyperdoctrine{\Cat{C}}{P}$ be a boolean hyperdoctrine and $\phi\in P(X\times Y)$ a functional and entire element from $X$ to$Y$. For all $\psi \in P(Y\times Z)$, it holds that 
    \[\exists_{\pr_{X\times Z}}(P_{\pr_{X\times Y}}(\phi)\wedge P_{\pr_{Y\times Z}}(\neg \psi)) =  \neg(\, \exists_{\pr_{X\times Z}}(P_{\pr_{X\times Y}}(\phi)\wedge P_{\pr_{Y\times Z}}(\psi))\,)\text{.}\]
    \end{lemma}

    \begin{proof}[Proof of \cref{lemma:boolhypcomp}]
    The proof happens to be straightforward if we employ the classical arguments of natural deduction. Here we provide a completely algebric proof.
    For sake of readability, here we employ the notation given by the internal language of $P$, writing $\exists y$ instead of the left adjoint $\exists_{\pr_Y}$ and using the predicates $\phi(x,y)$ and $\psi(y,z)$ to denote the elements $\phi\in P(X\times Y)$  and $\psi \in P(Y\times Z)$. Given the well-established correspondence between a doctrine and its internal language, proving that the two previous predicates a equivalent using the logical rules of first-order classical logic is equivalent to prove that they are equal in $P$. Hereafter we thus prove
    \[ \exists y. (\phi(x,y)\wedge \neg \psi(y,z))=\neg \exists y'. (\phi(x,y')\wedge  \psi(y',z))\] 
    
    Therefore, we start by proving 
    \begin{equation}\label{eq_primo_vero_neg_commuta_mappe}
        \exists y'. (\phi(x,y')\wedge \neg \psi(y',z))\vdash\neg \exists y. (\phi(x,y)\wedge  \psi(y,z)) 
    \end{equation}
    The crucial point to conclude that \eqref{eq_primo_vero_neg_commuta_mappe} holds is to show that 
    \begin{equation}\label{eq_seconda_prova_neg_commuta_mappe} 
        \phi (x,y)\wedge \psi (y,z)\vdash \neg \phi (x,y') \vee \psi (y',z).
    \end{equation}
    holds. In fact, from \eqref{eq_seconda_prova_neg_commuta_mappe} we can deduce the validity of \eqref{eq_primo_vero_neg_commuta_mappe} because \eqref{eq_seconda_prova_neg_commuta_mappe} implies that 
    \[ \exists y. (\phi (x,y)\wedge \psi (y,z))\vdash \neg \phi (x,y') \vee \psi (y',z)\]
    and hence, by applying to both sides $\neg$, we obtain
    \[ \phi (x,y') \wedge \neg  \psi (y',z)\vdash\neg \exists y. (\phi (x,y)\wedge \psi (y,z)) \]
    and hence that 
    \[ \exists y'.(\phi (x,y') \wedge \neg  \psi (y',z))\vdash\neg \exists y. (\phi (x,y)\wedge \psi (y,z)) \]
    
    So we have to prove the validity of \eqref{eq_seconda_prova_neg_commuta_mappe}.
    Now since we are working with boolean algebras, we have that $\phi (x,y)\wedge \psi (y,z)=\phi (x,y)\wedge \psi (y,z)\wedge (y=y'\vee y\neq y')$. Therefore, \eqref{eq_seconda_prova_neg_commuta_mappe}  is equivalent to 
    \begin{equation}\label{eq_quarta_prova_neg_commuta_mappe} 
        (\phi (x,y)\wedge \psi (y,z) \wedge y=y')\vee (\phi (x,y)\wedge \psi (y,z) \wedge y\neq y')\vdash \neg \phi (x,y') \vee \psi (y',z).
    \end{equation}
    To prove \eqref{eq_quarta_prova_neg_commuta_mappe}, it is enough to prove that both 
    \begin{equation}\label{eq_quinta}
        (\phi (x,y)\wedge \psi (y,z) \wedge y=y')\vdash \neg \phi (x,y') \vee \psi (y',z)
    \end{equation}
    and 
    \begin{equation}\label{eq_sesta}
        (\phi (x,y)\wedge \psi (y,z) \wedge y\neq y')\vdash \neg \phi (x,y') \vee \psi (y',z)
    \end{equation}
    hold. Now notice that \eqref{eq_sesta} holds trivially, because $\phi (x,y)\wedge \psi (y,z) \wedge y=y')\vdash \psi (y',z)$. To prove \eqref{eq_sesta} we have to employ the functionality of $\phi$. In fact, by definition of functionality we have that
    \[\phi (x,y)\wedge \phi (x,y')\vdash y=y'\]
    and then, using the fact that $y=y'$ is equivalent to $(y \neq  y')\to \bot$ (were $y\neq y'$ denotes $\neg (y=y')$) we can conclude that 
    \begin{equation}\label{eq_settima}
        \phi (x,y)\wedge  y\neq y' \vdash \neg \phi (x,y')
    \end{equation}
    and hence that $ \phi (x,y)\wedge  y\neq y' \vdash \neg \phi (x,y')\vee \psi(y',z)$. Therefeore, since $\phi (x,y)\wedge \psi (y,z) \wedge y\neq y'\vdash \phi (x,y)\wedge  y\neq y'$, we can conclude by transitivity that also \eqref{eq_sesta} holds.
    
    This concludes the proof that \eqref{eq_seconda_prova_neg_commuta_mappe}, and hence, \eqref{eq_primo_vero_neg_commuta_mappe}  hold.

    \bigskip
    Now we have to prove that
    \begin{equation}\label{eq_prima_secondo_verso}
       \neg \exists y. (\phi(x,y)\wedge  \psi(y,z)) \vdash  \exists y'. (\phi(x,y')\wedge \neg \psi(y',z))
    \end{equation}
    First, notice that \eqref{eq_prima_secondo_verso} is equivalent to
    \begin{equation}\label{eq_seconda_secondo_verso}
        \forall y'. (\neg \phi(x,y')\vee  \psi(y',z)) \vdash  \exists y. (\phi(x,y)\wedge  \psi(y,z)).
     \end{equation}
    Now, in order to prove \eqref{eq_seconda_secondo_verso}, we start by employing the assumption that $\phi$ is total, namely:
    \begin{equation}\label{eq_totality_phi}
        \top\vdash \exists y'.\phi(x,y').
    \end{equation}
    In fact, \eqref{eq_totality_phi} implies that 
    \[\forall y'.\neg \phi(x,y')\vdash \bot\]
    and since in every boolean algebra we have that $\psi (y',z)\wedge \neg \psi (y',z)\vdash \bot$, we have that
    \begin{equation}\label{eq_terza_secondo_verso}
        \forall y'.\neg \phi(x,y')=\forall y'.(\neg \phi(x,y')\vee \bot )=\forall y'.(\neg \phi(x,y')\vee(\psi (y',z)\wedge \neg \psi (y',z)))\vdash \bot 
    \end{equation}
    From \eqref{eq_terza_secondo_verso}, using the distributivity of $\vee$, and the fact that the universal quantifier of a disjunction is equivalent to the disjunction of two universal quantifiers (categorically, right adjoints preserve coproducts), we can conclude that 
    \[\forall y'.(\neg \phi(x,y')\vee\psi (y',z))\wedge \forall y.(\neg \phi(x,y) \wedge  \neg \psi (y,z))\vdash \bot\]
    and hence that
    \[\forall y'.(\neg \phi(x,y')\vee\psi (y',z))\vdash \neg  \forall y.(\neg \phi(x,y) \wedge  \neg \psi (y,z)).\]
    Now, since $\neg  \forall y.(\neg \phi(x,y) \wedge  \neg \psi (y,z))=\exists y. (\phi(x,y)\wedge \psi(y,z))$ we can conclude that \eqref{eq_seconda_secondo_verso}, and then \eqref{eq_prima_secondo_verso} hold.
    \end{proof}

%% file: sections/appadjunction.tex
\section{Appendix to \cref{sec:adjunction}}\label{app_section_adj}

In \cref{sec:adjunction} we have recalled the adjunction between the category of cartesian bicategories $\CartBic$  and that of elementary and existential doctrines $\EED$, i.e. \Cref{adj_CB_EED}, from \cite{bonchi2021doctrines}. 
The interested reader may find all details in Sections 5, 6 and 7 of \cite{bonchi2021doctrines} however, for its convenience, we recall in this appendix some interesting facts that are omitted in the main text.

We start by recalling some details about the $\REL(-)$ construction. The objects, arrows and composition of the category $\REL(P)$ associated to an elementary and existential doctrine $\doctrine{\Cat{C}}{P}$ are described in \cref{sec:adjunction}. To see why this is a cartesian bicategory it is convenient to first illustrate the monoidal product. For $\phi\colon X\to Y$ and $\psi\colon U \to V$ arrows of $\REL(P)$, namely $\phi \in P(X\times Y)$ and $\psi \in P(U\times V)$, the arrow $\phi \tensor[+][][] \psi\colon X \tensor[+][][] U \to Y\tensor[+][][] V$ is defined as 
\[ \phi \tensor[+][][] \psi \defeq P_{\langle\pi_X,\pi_Y\rangle}(\phi) \wedge  P_{\langle\pi_U,\pi_V\rangle}(\psi)\]
where $\langle\pi_X,\pi_Y\rangle$ and $\langle\pi_U,\pi_V\rangle$ are the projections from $X\times U \times Y \times Z $ to, respectively, $X\times Y$ and $U\times V$.

The rest of the structure of cartesian bicategory is inherited from the finite products base category $\Cat{C}$ of $P$ by means of the \emph{graph functor} $\Gamma_P \colon \Cat{C} \to \REL(P)$. In particular, the graph functor acts as the identity on objects and mapping each arrow $f\colon X \to Y$ in 
\[\Gamma_P(f) \defeq P_{f\times id_Y} (\delta_Y) \in P(X\times Y) = \REL(P)[X,Y]\text{.}\]
For instance, for the doctrine $\doctrine{\set}{\Pow}$  from \cref{ex_sub_on_regular}, the functor $\Gamma_{\Pow} \colon \set \to \Relp$ maps every function $f$ to its graph.
            
At this point, it is worth to observe that the arrow $P_{id_Y\times f}(\delta_Y)$ is right adjoint to $\Gamma_P(f)$ (see \cite[Pro. 23]{bonchi2021doctrines}) and thus $\Gamma_P(f)$ is a map. Therefore, $\Gamma_P$ restricts to a finite-product preserving functor $\Cat{C} \to \map(\REL (P))$. One can thus define the (co)monoid structure of $\REL (P)$ as
\[
\begin{array}{lcrlcr}
\copier[+][X] &\defeq& \Gamma_P(\copier[+][X]) \qquad & \qquad \cocopier[+][X] &\defeq& P_{id_{X\times X}\times \copier[+][X]}(\delta_{X\times X}) \\
\discard[+][X] &\defeq& \Gamma_P(\discard[+][X]) \qquad & \qquad \codiscard[+][X] &\defeq& P_{id_\unittensor \times \discard[+][X]}(\delta_{\unittensor}) \\
\end{array}
\]
where on the right hand side of the above equation $\copier[+][X]\colon X \to X \times X$ and $\discard[+][X]\colon X \to \unittensor$ are copier and discard of $\Cat{C}$ (they exist since $\Cat{C}$ has finite products).

\medskip

The graph functor $\freccia{\Cat{C}}{\Gamma_P}{\map(\REL(P))}$ is also used for defining the unit of the adjunction \eqref{adj_CB_EED}.
For every elementary and existential doctrine $\doctrine{\Cat{C}}{P}$, the morphism of elementary and existential doctrines $\freccia{P}{\eta_P}{\HM(\REL(P))}$ is 
   \begin{equation}\label{eq:unit}
    \begin{tikzcd}
      {\Cat{C}^{\op}} \\
      && \infsl \\
      {\mathsf{Map}(\mathsf{Rel}(P))^{\op}}
      \arrow["{{\REL(P)[-,1]}}"'{pos=0.3}, from=3-1, to=2-3]
      \arrow[""{name=0p, anchor=center, inner sep=0}, phantom, from=3-1, to=2-3, start anchor=center, end anchor=center]
      \arrow["P"{pos=0.45}, from=1-1, to=2-3]
      \arrow[""{name=1p, anchor=center, inner sep=0}, phantom, from=1-1, to=2-3, start anchor=center, end anchor=center]
      \arrow["{{\Gamma_P}^{\op}}"', from=1-1, to=3-1]
      \arrow["\rho"', shorten <=4pt, shorten >=4pt, from=1p, to=0p]
    \end{tikzcd}
\end{equation}
  where 
 
  \begin{itemize}
    \item $\freccia{\Cat{C}}{\Gamma_P}{\map(\REL(P))}$ is the graph-functor;
    \item each component $\freccia{P(X)}{\rho_X}{\REL(P)[X,\unittensor]}\defeq P(X\times I)$ is given by the isomorphism  $P(X)\cong P(X\times I)$ obtained by applying the functor $P$ to the right unitor $X\times I\cong X$ in $\Cat{C}$.
  \end{itemize}

%% file: sections/apppb.tex
\section{Appendix to Section~\ref{sec:pb}}\label{app:pb}

In this appendix we prove that the axioms of fo-bicategories hold in peircean bicategories. In the end we show the isomorphism $\FOBic \equiv \PeirceBic$ (\Cref{thm_equiv_FOBic_PeirceBic}). Before the main results, we need to prove a few useful properties of peircean bicategories.

\medskip

Since most of the proofs in this appendix are diagrammatic, it is worth remarking that negation behaves graphically as a \emph{colour switch}. Thus, for example $\nega{(\copierCirc[+][X])} = \copierCirc[-][X]$ and $\nega{(\discardCirc[-][X])} = \discardCirc[+][X]$.  For a generic arrow $\Circ{c}$, we depict its negation as $\Circ[-]{c}$. 

\medskip 
Moreover, it is convenient to visualize in diagrams~\eqref{eq:prop mappe} on two particular cases, namely when we take as map $\copier[+]$ or $\discard[+]$:
\begin{center}
    $\input{tikz/mapsCopier} \qquad\qquad \input{tikz/mapsDiscard.tex}$.
\end{center}

\medskip

\begin{lemma}\label{lemma:cup and bot}
    For all $c,d\colon X \to Y$, $c\vee d = \copier[-][X] \seq[-] (c \tensor[-] d) \seq[-] \cocopier[-][Y]$ and $\bot = \discard[-][X] \seq[-] \codiscard[-][Y]$, graphically rendered as follows
    \begin{center}
        $c \vee d = \unionCirc{c}{d}[X][Y] \qquad\qquad  \bot = \bottomCirc[X][Y]$
    \end{center}
    \end{lemma}
    \begin{proof} $ $
    
        \noindent\begin{minipage}{0.5\linewidth}
            \begin{align}
                c\vee d &= \nega{(\nega{c} \wedge \nega{d})} \tag{De Morgan}\\ 
                &= \nega{(   \copier[+][X] \seq[+] (\nega{c} \tensor[+] \nega{d}) \seq[+] \cocopier[+][Y]  )} \tag{\ref{eq:def:cap}}\\
                & =  \nega{(   \nega{\copier[-][X]} \seq[+] \nega{(c \tensor[-] d)} \seq[+] \nega{\cocopier[-][Y]}  )}  \tag{\ref{eq:defnegative}}\\
                & = \copier[-][X] \seq[-] (c \tensor[-] d) \seq[-] \cocopier[-][Y]  \tag{\ref{eq:defnegative}}
                \end{align}
        \end{minipage}
        \quad \vline
        \begin{minipage}{0.4\linewidth}
            \begin{align}
                \bot &= \nega{ \top} \tag{De Morgan}\\ 
                &= \nega{(   \discard[+][X] \seq[+] \codiscard[+][Y]  )} \tag{\ref{eq:def:cap}}\\
                & =  \nega{(   \nega{\discard[-][X]} \seq[+] \nega{\codiscard[-][Y]}   )}  \tag{\ref{eq:defnegative}}\\
                & = \discard[-][X] \seq[-] \codiscard[-][Y] \tag{\ref{eq:defnegative}}
                \end{align}
        \end{minipage}

        \qedhere
    \end{proof}
    
    Given that $(\Cat{C}, \copier[-], \cocopier[-])$ is a cocartesian bicategory, there is an isomorphism $\opp{(\cdot)} \colon \Cat{C} \to \opposite{\Cat{C}}$ that is the identity on objects and for all arrows $c \colon X \to Y$, $\opp{c} \colon Y \to X$ is defined as follows.
    \begin{equation}\label{eqddagger}
        \opp{c} \defeq \daggerCirc[b]{c}[Y][X]
    \end{equation}
    
    With this definition at hand it is immediate to show that $\nega{(\op{c})} = \opp{(\nega{c})}$. Moreover, we show that the two isomorphisms $\op{(\cdot)}$ and $\opp{(\cdot)}$ actually coincide (\Cref{cor:daggers are the same}).

    \begin{proposition}\label{prop:implications in boolean algebra}
        For any $c, d \colon X \to Y$ the following are equivalent:
        \begin{center}
            \begin{enumerate*}
                \item $c \leq d \qquad$ 
                \item $\nega{d} \leq \nega{c} \qquad$
                \item $\top \leq \nega{c} \vee d \qquad$
                \item $c \wedge \nega{d} \leq \bot \qquad$
            \end{enumerate*}
        \end{center}
    \end{proposition}
    \begin{proof}
        The four inclusions are equivalent since $\Cat{C}[X,Y]$ is a Boolean algebra.
    \end{proof}

    \begin{proposition}\label{prop:comaps counit}
        The following equality holds $\bottomCirc[X][Y] = 
    \InputIfFileExists{comapsCounits.tikz}{}{\input{tikz/comapsCounits.tikz}}
$
    \end{proposition}
    \begin{proof}
        We prove the two inclusions separately. The $\leq$ inclusion is trivial since $\bottomCirc[X][Y]$ is $\bot_{X,Y}$. For the other inclusion observe that $\bottomCirc[X][Y] \stackrel{\eqref{eq:prop mappe}}{=} 
    \InputIfFileExists{comapsCounits2.tikz}{}{\input{tikz/comapsCounits2.tikz}}
$ and thus what is left to prove is $
    \InputIfFileExists{comapsCounits.tikz}{}{\input{tikz/comapsCounits.tikz}}
 \leq 
    \InputIfFileExists{comapsCounits2.tikz}{}{\input{tikz/comapsCounits2.tikz}}
$. We prove it by means of \Cref{prop:cap property} as follows:
        \input{tikz/comapsCounitsProof.tex}
        where the last inequality holds since the left-hand side is $\bot_{X \tensor[+] Y, I}$.
    \end{proof}

    \begin{lemma}\label{lemma:dagger preserves bottom}
        The following equality holds, $\op{(\bot_{X,Y})} = \bot_{Y,X}$.
    \end{lemma}
    \begin{proof}
        \[
            \begin{array}{c@{\;\;}c@{\;}c@{\;}c@{\;}c@{\;}c@{\;}c}
                \op{(\bot_{X,Y})} 
                &\dstackrel{\eqref{eqdagger}}{\text{\Cref{lemma:cup and bot}}}{=}&
                \begin{tikzpicture}
                    \begin{pgfonlayer}{nodelayer}
                        \node [{dotStyle/+}] (107) at (-1, 0.375) {};
                        \node [style=none] (108) at (-0.5, 0.75) {};
                        \node [style=none] (109) at (-0.5, 0) {};
                        \node [style=none] (122) at (-2, -1.125) {};
                        \node [style=none] (123) at (-2, 1.125) {};
                        \node [style=none] (124) at (2.5, 1.125) {};
                        \node [style=none] (125) at (2.5, -1.125) {};
                        \node [{dotStyle/+}] (127) at (-1.625, 0.375) {};
                        \node [{dotStyle/+}] (128) at (1.5, -0.375) {};
                        \node [style=none] (129) at (1, 0) {};
                        \node [style=none] (130) at (1, -0.75) {};
                        \node [{dotStyle/+}] (131) at (2.125, -0.375) {};
                        \node [style=none] (132) at (2.5, 0.75) {};
                        \node [style=none] (133) at (-2, -0.75) {};
                        \node [style=label] (135) at (-2.475, -0.75) {$Y$};
                        \node [style=label] (136) at (2.975, 0.75) {$X$};
                        \node [style=none] (137) at (-0.5, -0.375) {};
                        \node [style=none] (138) at (-0.5, 0.375) {};
                        \node [style=none] (139) at (1, 0.375) {};
                        \node [style=none] (140) at (1, -0.375) {};
                        \node [style={dotStyle/-}] (141) at (0.5, 0) {};
                        \node [style={dotStyle/-}] (142) at (0, 0) {};
                    \end{pgfonlayer}
                    \begin{pgfonlayer}{edgelayer}
                        \draw [{bgStyle/+}] (124.center)
                            to (123.center)
                            to (122.center)
                            to (125.center)
                            to cycle;
                        \draw [{wStyle/+}, bend left] (109.center) to (107);
                        \draw [{wStyle/+}, bend left] (107) to (108.center);
                        \draw [{wStyle/+}] (107) to (127);
                        \draw [{wStyle/+}, bend right] (130.center) to (128);
                        \draw [{wStyle/+}, bend right] (128) to (129.center);
                        \draw [{wStyle/+}] (128) to (131);
                        \draw [{wStyle/+}] (132.center) to (108.center);
                        \draw [{wStyle/+}] (130.center) to (133.center);
                        \draw [{bgStyle/-}] (139.center)
                            to (138.center)
                            to (137.center)
                            to (140.center)
                            to cycle;
                        \draw [style={wStyle/-}] (129.center) to (141);
                        \draw [style={wStyle/-}] (142) to (109.center);
                    \end{pgfonlayer}
                \end{tikzpicture}
                &\stackrel{\text{\Cref{prop:comaps counit}}}{=}&
                \begin{tikzpicture}
                    \begin{pgfonlayer}{nodelayer}
                        \node [{dotStyle/+}] (107) at (-1, 0.375) {};
                        \node [style=none] (108) at (-0.5, 0.75) {};
                        \node [style=none] (109) at (-0.5, 0) {};
                        \node [style=none] (122) at (-2, -1.125) {};
                        \node [style=none] (123) at (-2, 1.125) {};
                        \node [style=none] (124) at (2.5, 1.125) {};
                        \node [style=none] (125) at (2.5, -1.125) {};
                        \node [{dotStyle/+}] (127) at (-1.625, 0.375) {};
                        \node [{dotStyle/+}] (128) at (1.5, -0.375) {};
                        \node [style={dotStyle/+}] (129) at (1, 0) {};
                        \node [style=none] (130) at (1, -0.75) {};
                        \node [{dotStyle/+}] (131) at (2.125, -0.375) {};
                        \node [style=none] (132) at (2.5, 0.75) {};
                        \node [style=none] (133) at (-2, -0.75) {};
                        \node [style=label] (135) at (-2.475, -0.75) {$Y$};
                        \node [style=label] (136) at (2.975, 0.75) {$X$};
                        \node [style=none] (137) at (-0.5, -0.375) {};
                        \node [style=none] (138) at (-0.5, 0.375) {};
                        \node [style=none] (139) at (0.5, 0.375) {};
                        \node [style=none] (140) at (0.5, -0.375) {};
                        \node [style={dotStyle/-}] (142) at (0, 0) {};
                    \end{pgfonlayer}
                    \begin{pgfonlayer}{edgelayer}
                        \draw [{bgStyle/+}] (124.center)
                            to (123.center)
                            to (122.center)
                            to (125.center)
                            to cycle;
                        \draw [{wStyle/+}, bend left] (109.center) to (107);
                        \draw [{wStyle/+}, bend left] (107) to (108.center);
                        \draw [{wStyle/+}] (107) to (127);
                        \draw [{wStyle/+}, bend right] (130.center) to (128);
                        \draw [{wStyle/+}, bend right] (128) to (129);
                        \draw [{wStyle/+}] (128) to (131);
                        \draw [{wStyle/+}] (132.center) to (108.center);
                        \draw [{wStyle/+}] (130.center) to (133.center);
                        \draw [{bgStyle/-}] (139.center)
                            to (138.center)
                            to (137.center)
                            to (140.center)
                            to cycle;
                        \draw [style={wStyle/-}] (142) to (109.center);
                    \end{pgfonlayer}
                \end{tikzpicture}            
                &\stackrel{\eqref{eq:prop mappe}}{=}&
                \begin{tikzpicture}
                    \begin{pgfonlayer}{nodelayer}
                        \node [{dotStyle/+}] (107) at (-1, 0.375) {};
                        \node [style=none] (108) at (-0.5, 0.75) {};
                        \node [style={dotStyle/+}] (109) at (-0.5, 0) {};
                        \node [style=none] (122) at (-2, -1.125) {};
                        \node [style=none] (123) at (-2, 1.125) {};
                        \node [style=none] (124) at (2.5, 1.125) {};
                        \node [style=none] (125) at (2.5, -1.125) {};
                        \node [{dotStyle/+}] (127) at (-1.625, 0.375) {};
                        \node [{dotStyle/+}] (128) at (1.5, -0.375) {};
                        \node [style={dotStyle/+}] (129) at (1, 0) {};
                        \node [style=none] (130) at (1, -0.75) {};
                        \node [{dotStyle/+}] (131) at (2.125, -0.375) {};
                        \node [style=none] (132) at (2.5, 0.75) {};
                        \node [style=none] (133) at (-2, -0.75) {};
                        \node [style=label] (135) at (-2.475, -0.75) {$Y$};
                        \node [style=label] (136) at (2.975, 0.75) {$X$};
                        \node [style=none] (137) at (-0.175, -0.375) {};
                        \node [style=none] (138) at (-0.175, 0.375) {};
                        \node [style=none] (139) at (0.675, 0.375) {};
                        \node [style=none] (140) at (0.675, -0.375) {};
                    \end{pgfonlayer}
                    \begin{pgfonlayer}{edgelayer}
                        \draw [{bgStyle/+}] (124.center)
                            to (123.center)
                            to (122.center)
                            to (125.center)
                            to cycle;
                        \draw [{wStyle/+}, bend left] (109) to (107);
                        \draw [{wStyle/+}, bend left] (107) to (108.center);
                        \draw [{wStyle/+}] (107) to (127);
                        \draw [{wStyle/+}, bend right] (130.center) to (128);
                        \draw [{wStyle/+}, bend right] (128) to (129);
                        \draw [{wStyle/+}] (128) to (131);
                        \draw [{wStyle/+}] (132.center) to (108.center);
                        \draw [{wStyle/+}] (130.center) to (133.center);
                        \draw [{bgStyle/-}] (139.center)
                            to (138.center)
                            to (137.center)
                            to (140.center)
                            to cycle;
                    \end{pgfonlayer}
                \end{tikzpicture}
                \\
                &\dstackrel{\eqref{ax:comPlusUnit}}{\eqref{ax:monPlusUnit}}{=}&
                \begin{tikzpicture}
                    \begin{pgfonlayer}{nodelayer}
                        \node [{dotStyle/+}] (107) at (-1, 0.75) {};
                        \node [style=none] (108) at (-0.5, 0.75) {};
                        \node [style=none] (122) at (-2, -1.125) {};
                        \node [style=none] (123) at (-2, 1.125) {};
                        \node [style=none] (124) at (2.5, 1.125) {};
                        \node [style=none] (125) at (2.5, -1.125) {};
                        \node [{dotStyle/+}] (128) at (1.5, -0.75) {};
                        \node [style=none] (130) at (1, -0.75) {};
                        \node [style=none] (132) at (2.5, 0.75) {};
                        \node [style=none] (133) at (-2, -0.75) {};
                        \node [style=label] (135) at (-2.475, -0.75) {$Y$};
                        \node [style=label] (136) at (2.975, 0.75) {$X$};
                        \node [style=none] (137) at (-0.175, -0.375) {};
                        \node [style=none] (138) at (-0.175, 0.375) {};
                        \node [style=none] (139) at (0.675, 0.375) {};
                        \node [style=none] (140) at (0.675, -0.375) {};
                    \end{pgfonlayer}
                    \begin{pgfonlayer}{edgelayer}
                        \draw [{bgStyle/+}] (124.center)
                             to (123.center)
                             to (122.center)
                             to (125.center)
                             to cycle;
                        \draw [{wStyle/+}] (107) to (108.center);
                        \draw [{wStyle/+}] (130.center) to (128);
                        \draw [{wStyle/+}] (132.center) to (108.center);
                        \draw [{wStyle/+}] (130.center) to (133.center);
                        \draw [{bgStyle/-}] (139.center)
                             to (138.center)
                             to (137.center)
                             to (140.center)
                             to cycle;
                    \end{pgfonlayer}
                \end{tikzpicture}                
                &\structuralcong&
                \begin{tikzpicture}
                    \begin{pgfonlayer}{nodelayer}
                        \node [{dotStyle/+}] (107) at (1.25, 0) {};
                        \node [style=none] (108) at (2, 0) {};
                        \node [style=none] (122) at (-2, -0.875) {};
                        \node [style=none] (123) at (-2, 0.875) {};
                        \node [style=none] (124) at (2.5, 0.875) {};
                        \node [style=none] (125) at (2.5, -0.875) {};
                        \node [{dotStyle/+}] (128) at (-0.75, 0) {};
                        \node [style=none] (130) at (-1.5, 0) {};
                        \node [style=none] (132) at (2.5, 0) {};
                        \node [style=none] (133) at (-2, 0) {};
                        \node [style=label] (135) at (-2.475, 0) {$Y$};
                        \node [style=label] (136) at (2.975, 0) {$X$};
                        \node [style=none] (137) at (-0.175, -0.375) {};
                        \node [style=none] (138) at (-0.175, 0.375) {};
                        \node [style=none] (139) at (0.675, 0.375) {};
                        \node [style=none] (140) at (0.675, -0.375) {};
                    \end{pgfonlayer}
                    \begin{pgfonlayer}{edgelayer}
                        \draw [{bgStyle/+}] (124.center)
                            to (123.center)
                            to (122.center)
                            to (125.center)
                            to cycle;
                        \draw [{wStyle/+}] (107) to (108.center);
                        \draw [{wStyle/+}] (130.center) to (128);
                        \draw [{wStyle/+}] (132.center) to (108.center);
                        \draw [{wStyle/+}] (130.center) to (133.center);
                        \draw [{bgStyle/-}] (139.center)
                            to (138.center)
                            to (137.center)
                            to (140.center)
                            to cycle;
                    \end{pgfonlayer}
                \end{tikzpicture} 
                &\stackrel{\eqref{eq:prop mappe}}{=}&
                \begin{tikzpicture}
                    \begin{pgfonlayer}{nodelayer}
                        \node [{dotStyle/+}] (107) at (1.25, 0) {};
                        \node [style=none] (108) at (2, 0) {};
                        \node [style=none] (122) at (-2, -0.875) {};
                        \node [style=none] (123) at (-2, 0.875) {};
                        \node [style=none] (124) at (2.5, 0.875) {};
                        \node [style=none] (125) at (2.5, -0.875) {};
                        \node [{dotStyle/-}] (128) at (-0.75, 0) {};
                        \node [style=none] (130) at (-1.425, 0) {};
                        \node [style=none] (132) at (2.5, 0) {};
                        \node [style=none] (133) at (-2, 0) {};
                        \node [style=label] (135) at (-2.475, 0) {$Y$};
                        \node [style=label] (136) at (2.975, 0) {$X$};
                        \node [style=none] (137) at (-1.425, -0.625) {};
                        \node [style=none] (138) at (-1.425, 0.625) {};
                        \node [style=none] (139) at (-0.075, 0.625) {};
                        \node [style=none] (140) at (-0.075, -0.625) {};
                    \end{pgfonlayer}
                    \begin{pgfonlayer}{edgelayer}
                        \draw [{bgStyle/+}] (124.center)
                             to (123.center)
                             to (122.center)
                             to (125.center)
                             to cycle;
                        \draw [{wStyle/+}] (107) to (108.center);
                        \draw [{wStyle/+}] (132.center) to (108.center);
                        \draw [{wStyle/+}] (130.center) to (133.center);
                        \draw [{bgStyle/-}] (139.center)
                             to (138.center)
                             to (137.center)
                             to (140.center)
                             to cycle;
                        \draw [style={wStyle/-}] (130.center) to (128);
                    \end{pgfonlayer}
                \end{tikzpicture}                 
                \\
                &\stackrel{\text{\Cref{prop:comaps counit}}}{=}&
                \begin{tikzpicture}
                    \begin{pgfonlayer}{nodelayer}
                        \node [{dotStyle/-}] (107) at (0.75, 0) {};
                        \node [style=none] (108) at (1.5, 0) {};
                        \node [style=none] (122) at (-2, -0.875) {};
                        \node [style=none] (123) at (-2, 0.875) {};
                        \node [style=none] (124) at (2.5, 0.875) {};
                        \node [style=none] (125) at (2.5, -0.875) {};
                        \node [{dotStyle/-}] (128) at (-0.25, 0) {};
                        \node [style=none] (130) at (-1, 0) {};
                        \node [style=none] (132) at (2.5, 0) {};
                        \node [style=none] (133) at (-2, 0) {};
                        \node [style=label] (135) at (-2.475, 0) {$Y$};
                        \node [style=label] (136) at (2.975, 0) {$X$};
                    \end{pgfonlayer}
                    \begin{pgfonlayer}{edgelayer}
                        \draw [{bgStyle/-}] (124.center)
                            to (123.center)
                            to (122.center)
                            to (125.center)
                            to cycle;
                        \draw [{wStyle/-}] (107) to (108.center);
                        \draw [{wStyle/-}] (130.center) to (128);
                        \draw [{wStyle/-}] (132.center) to (108.center);
                        \draw [{wStyle/-}] (130.center) to (133.center);
                    \end{pgfonlayer}
                \end{tikzpicture}
                &\stackrel{\text{\Cref{lemma:cup and bot}}}{=}&
                \bot_{Y,X}
            \end{array}
        \]
    \end{proof}
    
    \begin{lemma}\label{lemma:dagger preserves negation}
    For any $c \colon X \to Y$, $\op{(\nega{c})} = \nega{(\op{c})}$.
    \end{lemma}
    \begin{proof}
        We prove the two inclusions separately. 
        We prove $\leq$, on the left, by means of \Cref{prop:implications in boolean algebra}.4. For $\geq$, on the right, we use the properties of $\op{(\cdot)}$ in~\Cref{table:daggerproperties} and the inclusion proved on the left.
        
        \noindent\begin{minipage}{0.55\linewidth}
            \begin{align*}
                \op{(\nega{c})} \wedge \nega{\nega{(\op{c})}} 
                &= \op{(\nega{c})} \wedge \op{c} \tag{\Cref{def:peircean-bicategory}.1} \\
                &= \op{(\nega{c} \wedge c)} \tag{\Cref{table:daggerproperties}} \\
                &= \op{(\bot_{X,Y})} \tag{\Cref{def:peircean-bicategory}.1} \\
                &= \bot_{Y,X} \tag{\Cref{lemma:dagger preserves bottom}}
            \end{align*}
        \end{minipage}
        \quad \vline
        \begin{minipage}{0.4\linewidth}
            \begin{align*}
                \nega{(\op{c})} &= \op{\op{\nega{(\op{c})}}} \tag{\Cref{table:daggerproperties}} \\
                &\leq \op{\nega{\op{\op{c}}}} \tag{$\op{(\nega{c})} \leq \nega{(\op{c})}$} \\
                &= \op{(\nega{c})} \tag{\Cref{table:daggerproperties}}
            \end{align*}
        \end{minipage}

        \qedhere
    \end{proof}

    \begin{lemma}\label{cor:daggers are the same}
        For any $c \colon X \to Y$, $\daggerCirc[+]{c}[Y][X] = \daggerCirc[b]{c}[Y][X]$.
    \end{lemma}
    \begin{proof}
        \begin{align*}
            \op{c} 
            \; \stackrel{\text{\Cref{def:peircean-bicategory}.1}}{=} \;
            \op{\left(\nega{\nega{c}}\right)}
            \; \stackrel{\text{\Cref{lemma:dagger preserves negation}}}{=} \;
            \nega{\left( \op{(\nega{c})} \right)}
            \; \stackrel{\eqref{eqddagger}}{=} \;
            \opp{\left(\nega{\nega{c}}\right)}
            \; \stackrel{\text{\Cref{def:peircean-bicategory}.1}}{=} \;
            \opp{c} 
        \end{align*}
    \end{proof}

    \Cref{cor:daggers are the same} is instrumental in proving the following result that extends~\eqref{eq:prop mappe} to \emph{comaps}, i.e. all those arrows $f$, such that $f^\dagger$ is a map.

    \begin{lemma}\label{lemma:comaps}
        For all maps $f \colon X \to Y$ and arrows $c \colon Z \to Y$, $\nega{c} \seq[+] \op{f} = \nega{(c \seq[+] \op{f})}$.
    \end{lemma}
    \begin{proof}
        \[
            \begin{array}{c@{}c@{}c@{}c@{}c@{}c@{}c}
                \seqCirc[+]{c}{f}[Z][X][n][fop]
                &\stackrel{\eqref{eqdagger}}{=}&
                \begin{tikzpicture}
                    \begin{pgfonlayer}{nodelayer}
                        \node [style=none] (81) at (-2, -0.6) {};
                        \node [style=none] (112) at (-2, 1.4) {};
                        \node [style=none] (113) at (-2, -1.35) {};
                        \node [style=none] (114) at (2.5, -1.35) {};
                        \node [style=none] (115) at (2.5, 1.4) {};
                        \node [style=none] (121) at (2.5, 0.9) {};
                        \node [{boxStyle/-}] (122) at (-1.25, -0.6) {$c$};
                        \node [{funcStyle/+}] (123) at (0.75, 0.15) {$f$};
                        \node [style=label] (124) at (-2.5, -0.6) {$Z$};
                        \node [style=label] (125) at (3, 0.9) {$X$};
                        \node [style={dotStyle/+}] (126) at (1.75, -0.225) {};
                        \node [style={dotStyle/+}] (127) at (2.25, -0.225) {};
                        \node [style=none] (128) at (1.25, 0.15) {};
                        \node [style=none] (129) at (1.25, -0.6) {};
                        \node [style={dotStyle/+}] (130) at (-0.25, 0.525) {};
                        \node [style={dotStyle/+}] (131) at (-0.75, 0.525) {};
                        \node [style=none] (132) at (0.25, 0.9) {};
                        \node [style=none] (133) at (0.25, 0.15) {};
                    \end{pgfonlayer}
                    \begin{pgfonlayer}{edgelayer}
                        \draw [{bgStyle/+}] (114.center)
                            to (113.center)
                            to (112.center)
                            to (115.center)
                            to cycle;
                        \draw [{wStyle/+}] (81.center) to (122);
                        \draw [style={wStyle/+}, bend right] (129.center) to (126);
                        \draw [style={wStyle/+}, bend left] (128.center) to (126);
                        \draw [style={wStyle/+}] (126) to (127);
                        \draw [style={wStyle/+}, bend left] (133.center) to (130);
                        \draw [style={wStyle/+}, bend right] (132.center) to (130);
                        \draw [style={wStyle/+}] (130) to (131);
                        \draw [style={wStyle/+}] (133.center) to (123);
                        \draw [style={wStyle/+}] (123) to (128.center);
                        \draw [style={wStyle/+}] (132.center) to (121.center);
                        \draw [style={wStyle/+}] (129.center) to (122);
                    \end{pgfonlayer}
                \end{tikzpicture}       
                &\stackrel{\text{\Cref{prop:sliding through cap}}}{=}&
                \begin{tikzpicture}
                    \begin{pgfonlayer}{nodelayer}
                        \node [style=none] (81) at (-2, -0.75) {};
                        \node [style=none] (112) at (-2, 1.25) {};
                        \node [style=none] (113) at (-2, -1.25) {};
                        \node [style=none] (114) at (2.75, -1.25) {};
                        \node [style=none] (115) at (2.75, 1.25) {};
                        \node [style=none] (121) at (2.75, 0.75) {};
                        \node [{funcStyle/+}] (123) at (-0.25, 0) {$f$};
                        \node [style=label] (124) at (-2.5, -0.75) {$Z$};
                        \node [style=label] (125) at (3.25, 0.75) {$X$};
                        \node [style={dotStyle/+}] (126) at (2, -0.375) {};
                        \node [style={dotStyle/+}] (127) at (2.5, -0.375) {};
                        \node [style=none] (128) at (1.5, 0) {};
                        \node [style=none] (129) at (1.5, -0.75) {};
                        \node [style={dotStyle/+}] (130) at (-1.25, 0.375) {};
                        \node [style={dotStyle/+}] (131) at (-1.75, 0.375) {};
                        \node [style=none] (132) at (-0.75, 0.75) {};
                        \node [style=none] (133) at (-0.75, 0) {};
                        \node [{boxOpStyle/-}] (134) at (1, 0) {$c$};
                    \end{pgfonlayer}
                    \begin{pgfonlayer}{edgelayer}
                        \draw [{bgStyle/+}] (114.center)
                            to (113.center)
                            to (112.center)
                            to (115.center)
                            to cycle;
                        \draw [style={wStyle/+}, bend right] (129.center) to (126);
                        \draw [style={wStyle/+}, bend left] (128.center) to (126);
                        \draw [style={wStyle/+}] (126) to (127);
                        \draw [style={wStyle/+}, bend left] (133.center) to (130);
                        \draw [style={wStyle/+}, bend right] (132.center) to (130);
                        \draw [style={wStyle/+}] (130) to (131);
                        \draw [style={wStyle/+}] (133.center) to (123);
                        \draw [style={wStyle/+}] (132.center) to (121.center);
                        \draw [style={wStyle/+}] (123) to (134);
                        \draw [style={wStyle/+}] (134) to (128.center);
                        \draw [style={wStyle/+}] (81.center) to (129.center);
                    \end{pgfonlayer}
                \end{tikzpicture}   
                &\stackrel{\eqref{eq:prop mappe}}{=}&                                            
                \begin{tikzpicture}
                    \begin{pgfonlayer}{nodelayer}
                        \node [style=none] (81) at (-2.25, -1) {};
                        \node [style=none] (112) at (-2.25, 1.25) {};
                        \node [style=none] (113) at (-2.25, -1.25) {};
                        \node [style=none] (114) at (3, -1.25) {};
                        \node [style=none] (115) at (3, 1.25) {};
                        \node [style=none] (121) at (3, 1) {};
                        \node [{funcStyle/-}] (123) at (-0.25, 0) {$f$};
                        \node [style=label] (124) at (-2.75, -1) {$Z$};
                        \node [style=label] (125) at (3.5, 1) {$X$};
                        \node [style={dotStyle/+}] (126) at (2.25, -0.5) {};
                        \node [style={dotStyle/+}] (127) at (2.75, -0.5) {};
                        \node [style=none] (128) at (1.75, 0) {};
                        \node [style=none] (129) at (1.75, -1) {};
                        \node [style={dotStyle/+}] (130) at (-1.5, 0.525) {};
                        \node [style={dotStyle/+}] (131) at (-2, 0.525) {};
                        \node [style=none] (132) at (-1, 1) {};
                        \node [style=none] (133) at (-1, 0) {};
                        \node [{boxOpStyle/-}] (134) at (1, 0) {$c$};
                        \node [style=none] (135) at (-1, 0.75) {};
                        \node [style=none] (136) at (-1, -0.75) {};
                        \node [style=none] (137) at (1.75, -0.75) {};
                        \node [style=none] (138) at (1.75, 0.75) {};
                    \end{pgfonlayer}
                    \begin{pgfonlayer}{edgelayer}
                        \draw [{bgStyle/+}] (114.center)
                            to (113.center)
                            to (112.center)
                            to (115.center)
                            to cycle;
                        \draw [style={wStyle/+}, bend right] (129.center) to (126);
                        \draw [style={wStyle/+}, bend left] (128.center) to (126);
                        \draw [style={wStyle/+}] (126) to (127);
                        \draw [style={wStyle/+}, bend left] (133.center) to (130);
                        \draw [style={wStyle/+}, bend right] (132.center) to (130);
                        \draw [style={wStyle/+}] (130) to (131);
                        \draw [style={wStyle/+}] (132.center) to (121.center);
                        \draw [style={wStyle/+}] (81.center) to (129.center);
                        \draw [{bgStyle/-}] (137.center)
                            to (136.center)
                            to (135.center)
                            to (138.center)
                            to cycle;
                        \draw [style={wStyle/-}] (128.center) to (134);
                        \draw [style={wStyle/-}] (134) to (123);
                        \draw [style={wStyle/-}] (123) to (133.center);
                    \end{pgfonlayer}
                \end{tikzpicture} 
                \\
                &\stackrel{\text{\Cref{cor:daggers are the same}}}{=}&
                \begin{tikzpicture}
                    \begin{pgfonlayer}{nodelayer}
                        \node [style=none] (81) at (-2, -1) {};
                        \node [style=none] (112) at (-2, 1.25) {};
                        \node [style=none] (113) at (-2, -1.25) {};
                        \node [style=none] (114) at (2.75, -1.25) {};
                        \node [style=none] (115) at (2.75, 1.25) {};
                        \node [style=none] (121) at (2.75, 1) {};
                        \node [{funcStyle/-}] (123) at (-0.25, 0) {$f$};
                        \node [style=label] (124) at (-2.5, -1) {$Z$};
                        \node [style=label] (125) at (3.25, 1) {$X$};
                        \node [style={dotStyle/-}] (126) at (2, -0.5) {};
                        \node [style={dotStyle/-}] (127) at (2.5, -0.5) {};
                        \node [style=none] (128) at (1.5, 0) {};
                        \node [style=none] (129) at (1.5, -1) {};
                        \node [style={dotStyle/-}] (130) at (-1.25, 0.525) {};
                        \node [style={dotStyle/-}] (131) at (-1.75, 0.525) {};
                        \node [style=none] (132) at (-0.75, 1) {};
                        \node [style=none] (133) at (-0.75, 0) {};
                        \node [{boxOpStyle/-}] (134) at (1, 0) {$c$};
                    \end{pgfonlayer}
                    \begin{pgfonlayer}{edgelayer}
                        \draw [{bgStyle/-}] (114.center)
                             to (113.center)
                             to (112.center)
                             to (115.center)
                             to cycle;
                        \draw [style={wStyle/-}, bend right] (129.center) to (126);
                        \draw [style={wStyle/-}, bend left] (128.center) to (126);
                        \draw [style={wStyle/-}] (126) to (127);
                        \draw [style={wStyle/-}, bend left] (133.center) to (130);
                        \draw [style={wStyle/-}, bend right] (132.center) to (130);
                        \draw [style={wStyle/-}] (130) to (131);
                        \draw [style={wStyle/-}] (132.center) to (121.center);
                        \draw [style={wStyle/-}] (81.center) to (129.center);
                        \draw [style={wStyle/-}] (128.center) to (134);
                        \draw [style={wStyle/-}] (134) to (123);
                        \draw [style={wStyle/-}] (123) to (133.center);
                    \end{pgfonlayer}
                \end{tikzpicture}                       
                &\stackrel{\eqref{eqddagger}}{=}&
                \seqCirc[-]{c}{f}[Z][X][n][nfop]
            \end{array}                       
        \]
    \end{proof}

    It is convenient to visualize in diagrams \Cref{lemma:comaps} on two particular cases, namely when we take as comap $\cocopier[+]$ or $\codiscard[+]$:
    \begin{center}
        $\input{tikz/mapsCocopier} \qquad\qquad \input{tikz/mapsCodiscard.tex}$.
    \end{center}

    \medskip
    
    The structure in~\eqref{eq:defnegative} can also be used to define a compact closed structure in $(\Cat{C}, \copier[-], \cocopier[-])$. In other words, in a peircean bicategory one can \emph{bend} the wires in two ways. These bending operations are shown to be the same (\Cref{cor:cc are the same}). 
    
    \begin{proposition}\label{prop:partial cc}
        For any $c \colon X \tensor[+] Y \to \unittensor$, $
    \InputIfFileExists{partialCC/lhs.tikz}{}{\input{tikz/partialCC/lhs.tikz}}
 \leq 
    \InputIfFileExists{partialCC/rhs.tikz}{}{\input{tikz/partialCC/rhs.tikz}}
$.
    \end{proposition}
    \begin{proof}
        We prove it by means of \Cref{prop:implications in boolean algebra}.3 as follows.
        \[
            \begin{array}{c@{\,}c@{\,}c@{\,}c@{\,}c}
                \nega{\left( 
    \InputIfFileExists{partialCC/lhs.tikz}{}{\input{tikz/partialCC/lhs.tikz}}
 \right)} \vee 
    \InputIfFileExists{partialCC/rhs.tikz}{}{\input{tikz/partialCC/rhs.tikz}}
 
                &
                \dstackrel{\eqref{eq:defnegative}}{\text{\Cref{lemma:cup and bot}}}{=}
                & 
                
    \InputIfFileExists{partialCC/step1.tikz}{}{\input{tikz/partialCC/step1.tikz}}

                &
                \stackrel{\text{\Cref{th:spider}}}{=}
                &
                
    \InputIfFileExists{partialCC/step2.tikz}{}{\input{tikz/partialCC/step2.tikz}}
 \\[20pt]
                &
                \stackrel{\text{\Cref{def:peircean-bicategory}.1}}{=}
                &
                
    \InputIfFileExists{partialCC/step3.tikz}{}{\input{tikz/partialCC/step3.tikz}}

                &
                \stackrel{\eqref{eq:prop mappe}}{=}
                &
                
    \InputIfFileExists{partialCC/step4.tikz}{}{\input{tikz/partialCC/step4.tikz}}
 \\
                &
                \stackrel{\;\;\;\;\;\;\eqref{ax:comMinusUnit}\;\;\;\;\;\;}{=}
                &
                
    \InputIfFileExists{partialCC/step5.tikz}{}{\input{tikz/partialCC/step5.tikz}}

                &
                \structuralcong
                &
                
    \InputIfFileExists{partialCC/step6.tikz}{}{\input{tikz/partialCC/step6.tikz}}
 \\[12pt]
                &
                \stackrel{\text{\Cref{lemma:comaps}}}{=}
                &
                
    \InputIfFileExists{partialCC/step7.tikz}{}{\input{tikz/partialCC/step7.tikz}}

                &
                \stackrel{\eqref{eq:def:cap}}{=}
                &
                \top_{X,Y}      
            \end{array}
        \]
    \end{proof}
    
    \begin{lemma}\label{lemma:cc preserves negation}
       For any $c \colon X \to Y$, $\cappedCirc[+]{c}[X][Y][b] = \nega{\left( \cappedCirc[+]{c}[X][Y] \,\,\right)}$.
    \end{lemma}
    \begin{proof}
        We prove the two inclusions separately. 
        
        We prove $\leq$ by means of~\Cref{prop:implications in boolean algebra}.4 as follows.
        \[
        \begin{array}{c@{\,}c@{\,}c@{\,}c@{\,}c@{\,}c}
            \cappedCirc[+]{c}[X][Y][b] \wedge \nega{\nega{\left( \cappedCirc[+]{c}[X][Y] \,\,\right)}} 
            &
            \stackrel{\text{\Cref{def:peircean-bicategory}.1}}{=} 
            &
            \cappedCirc[+]{c}[X][Y][b] \wedge \cappedCirc[+]{c}[X][Y]
            &
            \stackrel{\eqref{eq:def:cap}}{=}
            &
            
    \InputIfFileExists{CCneg/step1.tikz}{}{\input{tikz/CCneg/step1.tikz}}

            \\
            &
            \stackrel{\;\;\;\;\text{\Cref{th:spider}}\;\;\;\;}{=} 
            &
            
    \InputIfFileExists{CCneg/step2.tikz}{}{\input{tikz/CCneg/step2.tikz}}

            &
            \stackrel{\text{\Cref{def:peircean-bicategory}.1}}{=}
            &
            
    \InputIfFileExists{CCneg/step3.tikz}{}{\input{tikz/CCneg/step3.tikz}}

            \\
            &
            \stackrel{\text{\Cref{lemma:comaps}}}{=} 
            &
            
    \InputIfFileExists{CCneg/step4.tikz}{}{\input{tikz/CCneg/step4.tikz}}

            &
            \stackrel{\eqref{ax:monPlusUnit}}{=} 
            &
            
    \InputIfFileExists{CCneg/step5.tikz}{}{\input{tikz/CCneg/step5.tikz}}

            \\
            &
            \;\;\;\;\;\;\;\; \structuralcong \;\;\;\;\;\;\;\;
            &
            
    \InputIfFileExists{CCneg/step6.tikz}{}{\input{tikz/CCneg/step6.tikz}}

            &
            \stackrel{\eqref{eq:prop mappe}}{=}  
            &
            
    \InputIfFileExists{CCneg/step7.tikz}{}{\input{tikz/CCneg/step7.tikz}}

            \\
            &
            \stackrel{\text{\Cref{lemma:cup and bot}}}{=} 
            &
            \bot_{X \tensor[+] Y, \unittensor}  
        \end{array}
        \]
        
        For $\geq$ we proceed as follows.
        \begin{align*}
            \nega{\left( \!\cappedCirc[+]{c}[X][Y] \,\,\right)} 
            &\stackrel{\phantom{\Cref{th:spider}}}{=}
            \cappedCirc[-]{c}[X][Y][b]
            \stackrel{\text{\Cref{th:spider}}}{=}\!\!
            \begin{tikzpicture}
                \begin{pgfonlayer}{nodelayer}
                    \node [style=none] (109) at (-2.325, 1) {};
                    \node [style=none] (122) at (-2.325, -0.875) {};
                    \node [style=none] (123) at (-2.325, 1.625) {};
                    \node [style=none] (124) at (0.575, 1.625) {};
                    \node [style=none] (125) at (0.575, -0.875) {};
                    \node [{dotStyle/+}] (128) at (-0.25, 0.625) {};
                    \node [style=none] (129) at (-0.75, 1) {};
                    \node [style=none] (130) at (-0.75, 0.25) {};
                    \node [{dotStyle/+}] (131) at (0.2, 0.625) {};
                    \node [style=none] (133) at (-1, 0.25) {};
                    \node [{boxStyle/-}] (134) at (-1, 1) {$c$};
                    \node [style=label] (135) at (-3.05, 1) {$X$};
                    \node [style=label] (136) at (-3.05, -1.25) {$Y$};
                    \node [{dotStyle/+}] (137) at (-1.5, -0.125) {};
                    \node [style=none] (138) at (-1, 0.25) {};
                    \node [style=none] (139) at (-1, -0.5) {};
                    \node [{dotStyle/+}] (140) at (-1.95, -0.125) {};
                    \node [style=none] (141) at (0.575, -0.5) {};
                    \node [style=none] (142) at (-2.575, -1.625) {};
                    \node [style=none] (143) at (-2.575, 1.875) {};
                    \node [style=none] (144) at (1.9, 1.875) {};
                    \node [style=none] (145) at (1.9, -1.625) {};
                    \node [style=none] (146) at (-2.575, 1) {};
                    \node [{dotStyle/-}] (147) at (1.075, -0.875) {};
                    \node [style=none] (148) at (0.575, -0.5) {};
                    \node [style=none] (149) at (0.575, -1.25) {};
                    \node [{dotStyle/-}] (150) at (1.525, -0.875) {};
                    \node [style=none] (151) at (-2.575, -1.25) {};
                \end{pgfonlayer}
                \begin{pgfonlayer}{edgelayer}
                    \draw [{bgStyle/-}] (144.center)
                         to (143.center)
                         to (142.center)
                         to (145.center)
                         to cycle;
                    \draw [{bgStyle/+}] (124.center)
                         to (123.center)
                         to (122.center)
                         to (125.center)
                         to cycle;
                    \draw [{wStyle/+}, bend right] (130.center) to (128);
                    \draw [{wStyle/+}, bend right] (128) to (129.center);
                    \draw [{wStyle/+}] (128) to (131);
                    \draw [{wStyle/+}] (130.center) to (133.center);
                    \draw [{wStyle/+}] (109.center) to (134);
                    \draw [{wStyle/+}] (129.center) to (134);
                    \draw [{wStyle/+}, bend left] (139.center) to (137);
                    \draw [{wStyle/+}, bend left] (137) to (138.center);
                    \draw [{wStyle/+}] (137) to (140);
                    \draw [style={wStyle/+}] (141.center) to (139.center);
                    \draw [{wStyle/-}, bend right] (149.center) to (147);
                    \draw [{wStyle/-}, bend right] (147) to (148.center);
                    \draw [{wStyle/-}] (147) to (150);
                    \draw [style={wStyle/-}] (149.center) to (151.center);
                    \draw [style={wStyle/-}] (146.center) to (109.center);
                    \draw [style={wStyle/-}] (148.center) to (141.center);
                    \draw [style={wStyle/+}] (138.center) to (133.center);
                \end{pgfonlayer}
            \end{tikzpicture}               
            \stackrel{\text{\Cref{prop:partial cc}}}{\leq}\!\!
            \begin{tikzpicture}
                \begin{pgfonlayer}{nodelayer}
                    \node [style=none] (109) at (-1.75, 1) {};
                    \node [style=none] (122) at (-1.75, -0.125) {};
                    \node [style=none] (123) at (-1.75, 1.625) {};
                    \node [style=none] (124) at (0.575, 1.625) {};
                    \node [style=none] (125) at (0.575, -0.125) {};
                    \node [{dotStyle/+}] (128) at (-0.25, 0.625) {};
                    \node [style=none] (129) at (-0.75, 1) {};
                    \node [style=none] (130) at (-0.75, 0.25) {};
                    \node [{dotStyle/+}] (131) at (0.2, 0.625) {};
                    \node [style=none] (133) at (-1.75, 0.25) {};
                    \node [{boxStyle/-}] (134) at (-1, 1) {$c$};
                    \node [style=label] (135) at (-3.55, 1) {$X$};
                    \node [style=label] (136) at (-3.55, -1.25) {$Y$};
                    \node [{dotStyle/-}] (137) at (-2.25, -0.125) {};
                    \node [style=none] (138) at (-1.75, 0.25) {};
                    \node [style=none] (139) at (-1.75, -0.5) {};
                    \node [{dotStyle/-}] (140) at (-2.7, -0.125) {};
                    \node [style=none] (141) at (-1, -0.5) {};
                    \node [style=none] (142) at (-3.075, -1.625) {};
                    \node [style=none] (143) at (-3.075, 1.875) {};
                    \node [style=none] (144) at (0.825, 1.875) {};
                    \node [style=none] (145) at (0.825, -1.625) {};
                    \node [style=none] (146) at (-3.075, 1) {};
                    \node [{dotStyle/-}] (147) at (-0.25, -0.875) {};
                    \node [style=none] (148) at (-0.75, -0.5) {};
                    \node [style=none] (149) at (-0.75, -1.25) {};
                    \node [{dotStyle/-}] (150) at (0.2, -0.875) {};
                    \node [style=none] (151) at (-3.075, -1.25) {};
                \end{pgfonlayer}
                \begin{pgfonlayer}{edgelayer}
                    \draw [{bgStyle/-}] (144.center)
                         to (143.center)
                         to (142.center)
                         to (145.center)
                         to cycle;
                    \draw [{bgStyle/+}] (124.center)
                         to (123.center)
                         to (122.center)
                         to (125.center)
                         to cycle;
                    \draw [{wStyle/+}, bend right] (130.center) to (128);
                    \draw [{wStyle/+}, bend right] (128) to (129.center);
                    \draw [{wStyle/+}] (128) to (131);
                    \draw [{wStyle/+}] (130.center) to (133.center);
                    \draw [{wStyle/+}] (109.center) to (134);
                    \draw [{wStyle/+}] (129.center) to (134);
                    \draw [{wStyle/-}, bend left] (139.center) to (137);
                    \draw [{wStyle/-}, bend left] (137) to (138.center);
                    \draw [{wStyle/-}, bend right] (149.center) to (147);
                    \draw [{wStyle/-}, bend right] (147) to (148.center);
                    \draw [{wStyle/-}] (147) to (150);
                    \draw [style={wStyle/-}] (149.center) to (151.center);
                    \draw [style={wStyle/-}] (146.center) to (109.center);
                    \draw [style={wStyle/-}] (148.center) to (141.center);
                    \draw [style={wStyle/-}] (133.center) to (138.center);
                    \draw [style={wStyle/-}] (137) to (140);
                    \draw [style={wStyle/-}] (139.center) to (141.center);
                \end{pgfonlayer}
            \end{tikzpicture}                    
            \\
            &\stackrel{\text{\Cref{th:spider}}}{=}\!\!
            \cappedCirc[+]{c}[X][Y][b]    
        \end{align*}
        \qedhere
    \end{proof}

    \begin{lemma}\label{cor:cc are the same}
        For any $c \colon X \to Y$, $\cappedCirc[+]{c}[X][Y] = \cappedCirc[-]{c}[X][Y]$.
    \end{lemma}
    \begin{proof}
        \begin{align*}
            \cappedCirc[+]{c}[X][Y] \stackrel{\text{\Cref{def:peircean-bicategory}.1}}{=} \nega{\nega{\left( \cappedCirc[+]{c}[X][Y] \right)}} \stackrel{\text{\Cref{lemma:cc preserves negation}}}{=} \nega{\left( \cappedCirc[+]{c}[X][Y][b] \right)} = \cappedCirc[-]{c}[X][Y]
        \end{align*}
    \end{proof}
    
    Now we need to prove that the axioms of fo-bicategories hold in a peircean bicategory. We do not show a proof for all of them, but only for a few representative ones. The rest are either derivable or proved in a similar manner.

    \begin{proposition}\label{prop:lin distr and over or}
        For any $c,d,e \colon X \to Y$, $c \wedge (d \vee e) \leq (c \wedge d) \vee e$.
    \end{proposition}
    \begin{proof}
        \begin{align*}
            c \wedge (d \vee e) &= (c \wedge d) \vee (c \wedge e) \tag{Definition~\ref{def:peircean-bicategory}.1} \\
                                  &\leq (c \wedge d) \vee (\top \wedge e) \tag{\Cref{lemma:meet semilattice}}\\
                                  &= (c \wedge d) \vee e \tag{Definition~\ref{def:peircean-bicategory}.1}
        \end{align*}
    \end{proof}

    \begin{lemma}\label{lemma:restricted tensor lin distr}
        For any $c \colon X \to \unittensor, d \colon Y \to \unittensor, e \colon Z \to \unittensor$, $c \tensor[+] (d \tensor[-] e) \leq (c \tensor[+] d) \tensor[-] e$.
    \end{lemma}
    \begin{proof}
        \input{tikz/tensorLinDistr}
    \end{proof}

    \newpage
    
    \begin{lemma}\label{lemma:restricted seq lin distr}
        For any $c \colon X \tensor[+] Y \to \unittensor, d \colon Y \tensor[+] Z \to \unittensor, e \colon Z \tensor[+] W \to \unittensor$, the following inequality holds
        \[
            
    \InputIfFileExists{restrictedLinDistr/lhs.tikz}{}{\input{tikz/restrictedLinDistr/lhs.tikz}}
 \;\;  \leq 
    \InputIfFileExists{restrictedLinDistr/rhs.tikz}{}{\input{tikz/restrictedLinDistr/rhs.tikz}}

        \]
    \end{lemma}
    \begin{proof}
        \input{tikz/restrictedLinDistr/proof}
    \end{proof}
    
    \newpage

    \begin{lemma}[Linear distributivities]\label{lemma:lin distr}
        For any $c \colon X \to Y, d \colon Y \to Z, e \colon Z \to W$, the following inequalities hold
        \begin{center}
            \begin{enumerate*}
                \item $c \seq[+] (d \seq[-] e) \leq (c \seq[+] d) \seq[-] e \qquad$
                \item $(c \seq[-] d) \seq[+] e \leq c \seq[-] (d \seq[+] e)$
            \end{enumerate*}
        \end{center}
    \end{lemma}
    \begin{proof}
        We prove $1.$ below by means of \Cref{prop:cap property}. The proof for $2.$ is analogous.
        \input{tikz/linDistr/proof}
    \end{proof}

    \begin{proposition}\label{prop:restricted EM}
        For any $c \colon X \to Y$, the following inequalities hold
        \begin{center}
            \begin{enumerate*}
                \item $
    \InputIfFileExists{restrictedEM.tikz}{}{\input{tikz/restrictedEM.tikz}}
 \, \leq \!\!\capCirc[-][X]\qquad\qquad$ 
                \item $\capCirc[+][X] \, \leq \!\!
    \InputIfFileExists{restrictedNC.tikz}{}{\input{tikz/restrictedNC.tikz}}
$
            \end{enumerate*}
        \end{center}
    \end{proposition}
    \begin{proof}
        We prove $1.$ below. The proof for $2.$ is analogous.
        \input{tikz/exMid/proof}
    \end{proof}

    \begin{lemma}[Linear adjoints]\label{lemma:lin adj}
        For any $c \colon X \to Y$, the following inequalities hold
        \begin{center}
            \begin{enumerate*}
                \item $c \seq[+] \op{(\nega{c})} \leq \id[-][X]\qquad$
                \item $\id[+][Y] \leq \op{(\nega{c})} \seq[-] c\qquad$
                \item $\op{(\nega{c})} \seq[+] c  \leq \id[-][Y]\qquad$
                \item $\id[+][X] \leq c \seq[-] \op{(\nega{c})}$
            \end{enumerate*}
        \end{center}
    \end{lemma}
    \begin{proof}
        We prove $1.$ below. $2.$ is proved similarly, exploiting \Cref{prop:restricted EM}.2. The proofs for $3.$ and $4.$ are analogous.
        \input{tikz/linAdjoints/proof}
    \end{proof}

\begin{lemma}[Linear strengths I]\label{lemma:lin str}
    For any $a,b,c,d$ properly typed, the following equalities hold
    \begin{center}
        \begin{enumerate*}
            \item $
    \InputIfFileExists{axiomsNEW/linStr1_1.tikz}{}{\input{tikz/axiomsNEW/linStr1_1.tikz}}
 \leq 
    \InputIfFileExists{axiomsNEW/linStr1_2.tikz}{}{\input{tikz/axiomsNEW/linStr1_2.tikz}}
\qquad\qquad$
            \item $
    \InputIfFileExists{axiomsNEW/linStr1_1.tikz}{}{\input{tikz/axiomsNEW/linStr1_1.tikz}}
 \leq 
    \InputIfFileExists{axiomsNEW/linStr2_2.tikz}{}{\input{tikz/axiomsNEW/linStr2_2.tikz}}
$
            \item $
    \InputIfFileExists{axiomsNEW/linStr3_1.tikz}{}{\input{tikz/axiomsNEW/linStr3_1.tikz}}
 \leq 
    \InputIfFileExists{axiomsNEW/linStr3_2.tikz}{}{\input{tikz/axiomsNEW/linStr3_2.tikz}}
\qquad\qquad$
            \item $
    \InputIfFileExists{axiomsNEW/linStr4_1.tikz}{}{\input{tikz/axiomsNEW/linStr4_1.tikz}}
 \leq 
    \InputIfFileExists{axiomsNEW/linStr3_2.tikz}{}{\input{tikz/axiomsNEW/linStr3_2.tikz}}
$
        \end{enumerate*}
    \end{center}
\end{lemma}
\begin{proof}
    We prove $1.$ by means of \Cref{prop:implications in boolean algebra}.4 below. The proof for $2.$ is analogous. $3.$ and $4.$ follow from $1.$ and $2.$ and \Cref{prop:implications in boolean algebra}.2.
    \input{tikz/linStrenghts/proof}
\end{proof}

\begin{proposition}\label{prop:and under or}
    For any $c,d \colon X \to Y$, $c \wedge d \leq c \vee d$.
\end{proposition}
\begin{proof}
    The following holds since $\Cat{C}[X,Y]$ is a Boolean algebra and a $\wedge$-semilattice with $\top$:
    \begin{align*}
        c \wedge d
        &=
        c \wedge (d \vee d) \tag{Idempotency of $\vee$}  \\ 
        &\leq
        (c \wedge d) \vee d \tag{\Cref{prop:lin distr and over or}}  \\ 
        &\leq
        (c \wedge \top) \vee d \tag{$\top$ is the top element} \\ 
        &=
        c \vee d \tag{$\top$ is the unit of $\wedge$}
    \end{align*}
\end{proof}

\begin{lemma}\label{lemma:tensor and under or}
    For any $c \colon X \to Y, d \colon Z \to W$, $c \tensor[+] d \leq c \tensor[-] d$.
\end{lemma}
\begin{proof}
    \input{tikz/tensorsLemma/proof}
\end{proof}

\begin{corollary}[Linear strenghts II]\label{cor:lin str id}
    For all objects $X,Y$, the following inequalities hold 
    \begin{center}
        $\id[-][X] \tensor[+] \id[-][Y] \leq \id[-][X] \tensor[-] \id[-][Y] \qquad\quad \id[+][X] \tensor[+] \id[+][Y] \leq \id[+][X] \tensor[-] \id[+][Y]$  
    \end{center}
\end{corollary}
\begin{proof}
    Immediate by \Cref{lemma:tensor and under or}.
\end{proof}

\begin{proposition}\label{prop:monoid rewiring}
    The following equality holds $
    \InputIfFileExists{rewiredMonoid.tikz}{}{\input{tikz/rewiredMonoid.tikz}}
 = \cocopierCirc[-][X]$.
\end{proposition}
\begin{proof}
    \input{tikz/frobAux/proof}
\end{proof}

\begin{lemma}[Linear Frobenius]\label{lemma:lin frob}
    The following equalities hold
    \begin{center}
        \begin{enumerate*}
            \item $
    \InputIfFileExists{axiomsNEW/bwS2.tikz}{}{\input{tikz/axiomsNEW/bwS2.tikz}}
 = 
    \InputIfFileExists{axiomsNEW/bwZ2.tikz}{}{\input{tikz/axiomsNEW/bwZ2.tikz}}
\qquad$
            \item $
    \InputIfFileExists{axiomsNEW/bwS.tikz}{}{\input{tikz/axiomsNEW/bwS.tikz}}
 = 
    \InputIfFileExists{axiomsNEW/bwZ.tikz}{}{\input{tikz/axiomsNEW/bwZ.tikz}}
$
            \item $
    \InputIfFileExists{axiomsNEW/wbS2.tikz}{}{\input{tikz/axiomsNEW/wbS2.tikz}}
 = 
    \InputIfFileExists{axiomsNEW/wbZ2.tikz}{}{\input{tikz/axiomsNEW/wbZ2.tikz}}
\qquad$
            \item $
    \InputIfFileExists{axiomsNEW/wbS.tikz}{}{\input{tikz/axiomsNEW/wbS.tikz}}
 = 
    \InputIfFileExists{axiomsNEW/wbZ.tikz}{}{\input{tikz/axiomsNEW/wbZ.tikz}}
$
        \end{enumerate*}
    \end{center}
\end{lemma}
\begin{proof}
    We prove $1.$ below. The proof for $2.$ is analogous. $3.$ and $4.$ follow from $1.$ and $2.$ and \Cref{prop:implications in boolean algebra}.2.
    \input{tikz/linFrob/proof}
\end{proof}

\begin{proof}[Proof of \Cref{thm_equiv_FOBic_PeirceBic}]
 By Propositions \ref{prop:enrichment} and \ref{prop:maps} every fo-bicategory is a peircean bicategory. By Proposition \ref{prop:map} every morphism of fo-bicategories is a morphism of peircean bicategories.

    Now observe that a peircean bicategory $\Cat{C}$ is a fo-bicategory, since:
    \begin{itemize}
        \item it is a cartesian bicategory by definition;
        \item it is a cocartesian bicategory via the isomorphism $\nega{}\colon (\co{\Cat{C}}, \copier[+], \cocopier[+]) \to (\Cat{C}, \copier[-], \cocopier[-])$;
        \item it is a closed linear bicategory, since:
        \begin{itemize}
            \item \eqref{ax:leftLinDistr} and \eqref{ax:rightLinDistr} hold by \Cref{lemma:lin distr};
            \item \eqref{ax:linStrn1}, \eqref{ax:linStrn2}, \eqref{ax:linStrn3} and \eqref{ax:linStrn4} hold by \Cref{lemma:lin str} and \eqref{ax:tensorPlusIdMinus} and \eqref{ax:tensorMinusIdPlus} hold by \Cref{cor:lin str id};
            \item every arrow $c \colon X \to Y$, and in particular also $\symm[+], \symm[-], \copier[+], \discard[+], \cocopier[+], \codiscard[+], \copier[-], \discard[-], \cocopier[-]$ and $\codiscard[-]$, has both a left and right linear adjoint by \Cref{lemma:lin adj};
        \end{itemize}
        \item \eqref{ax:bwFrob}, \eqref{ax:bwFrob2}, \eqref{ax:wbFrob} and \eqref{ax:wbFrob2} hold by \Cref{lemma:lin frob}.
    \end{itemize}

    A morphism of peircean bicategories $F \colon \Cat{C} \to \Cat{D}$ preserves the structure defined in~\eqref{eq:defnegative} since it preserves negation, e.g.
    \begin{align*}
        F(c \seq[-] d) &= F(\nega{(\nega{c} \seq[+] \nega{d})}) \tag{\ref{eq:defnegative}} \\
                       &= \nega{F(\nega{c} \seq[+] \nega{d})} \tag{$F$ preserves negation}  \\
                       &= \nega{(F(\nega{c}) \seq[+] F(\nega{d}))} \tag{$F$ is a strong symmetric monoidal functor} \\
                       &= \nega{(\nega{F(c)} \seq[+] \nega{F(d)})} \tag{$F$ preserves negation} \\
                       &= F(c) \seq[-] F(d) \tag{\ref{eq:defnegative}}
    \end{align*}
    In particular, $F$ is both a morphism of cartesian and cocartesian bicategories and thus a morphism of fo-bicategories.
\end{proof}

%% file: tikz/restrictedLinDistr/lhs.tikz
\begin{tikzpicture}
	\begin{pgfonlayer}{nodelayer}
		\node [style={dotStyle/-}] (81) at (0.175, -0.5) {};
		\node [{dotStyle/-}] (107) at (0.675, -0.5) {};
		\node [style=none] (108) at (1.175, -0.9) {};
		\node [style=none] (109) at (1.175, -0.075) {};
		\node [style=none] (110) at (-0.325, -1.4) {};
		\node [style=none] (111) at (-0.325, 0.425) {};
		\node [style=none] (112) at (2.425, 0.875) {};
		\node [style=none] (113) at (2.425, -1.85) {};
		\node [style=none] (114) at (-0.325, -1.85) {};
		\node [style=none] (115) at (-0.325, 0.875) {};
		\node [style=label] (116) at (-2.25, -1.4) {$W$};
		\node [style={boxStyle/+}] (120) at (1.675, -1.15) {$e$};
		\node [style={boxStyle/+}] (121) at (1.675, 0.175) {$d$};
		\node [style=none] (122) at (1.675, -1.4) {};
		\node [style=none] (123) at (1.675, -0.9) {};
		\node [style=none] (124) at (1.675, -0.075) {};
		\node [style=none] (125) at (1.675, 0.425) {};
		\node [style={dotStyle/+}] (126) at (-1.325, 0.825) {};
		\node [{dotStyle/+}] (127) at (-0.825, 0.825) {};
		\node [style=none] (128) at (-0.325, 0.425) {};
		\node [style=none] (129) at (-0.325, 1.25) {};
		\node [style=none] (131) at (-1.825, 1.75) {};
		\node [style=none] (132) at (2.85, 2.45) {};
		\node [style=none] (133) at (2.85, -2.225) {};
		\node [style=none] (134) at (-1.825, -2.225) {};
		\node [style=none] (135) at (-1.825, 2.45) {};
		\node [style=label] (137) at (-2.25, 1.75) {$X$};
		\node [style={boxStyle/+}] (139) at (1.675, 1.5) {$c$};
		\node [style=none] (141) at (-0.325, 0.425) {};
		\node [style=none] (142) at (1.675, 1.25) {};
		\node [style=none] (143) at (1.675, 1.75) {};
		\node [style=none] (144) at (-1.825, -1.4) {};
	\end{pgfonlayer}
	\begin{pgfonlayer}{edgelayer}
		\draw [{bgStyle/+}] (134.center)
			 to (133.center)
			 to (132.center)
			 to (135.center)
			 to cycle;
		\draw [{bgStyle/-}] (114.center)
			 to (113.center)
			 to (112.center)
			 to (115.center)
			 to cycle;
		\draw [{wStyle/-}, bend right=45] (109.center) to (107);
		\draw [{wStyle/-}, bend right=45] (107) to (108.center);
		\draw [{wStyle/-}] (81) to (107);
		\draw [style={wStyle/-}] (109.center) to (124.center);
		\draw [style={wStyle/-}] (123.center) to (108.center);
		\draw [style={wStyle/-}] (110.center) to (122.center);
		\draw [style={wStyle/-}] (111.center) to (125.center);
		\draw [{wStyle/+}, bend right=45] (129.center) to (127);
		\draw [{wStyle/+}, bend right=45] (127) to (128.center);
		\draw [{wStyle/+}] (126) to (127);
		\draw [style={wStyle/+}] (129.center) to (142.center);
		\draw [style={wStyle/+}] (141.center) to (128.center);
		\draw [style={wStyle/+}] (131.center) to (143.center);
		\draw [style={wStyle/+}] (110.center) to (144.center);
	\end{pgfonlayer}
\end{tikzpicture}

%% file: tikz/restrictedLinDistr/rhs.tikz
\begin{tikzpicture}
	\begin{pgfonlayer}{nodelayer}
		\node [style={dotStyle/+}] (81) at (0.175, 0.725) {};
		\node [{dotStyle/+}] (107) at (0.675, 0.725) {};
		\node [style=none] (108) at (1.175, 1.125) {};
		\node [style=none] (109) at (1.175, 0.3) {};
		\node [style=none] (110) at (-0.325, 1.625) {};
		\node [style=none] (111) at (-0.325, -0.2) {};
		\node [style=none] (112) at (2.425, -0.65) {};
		\node [style=none] (113) at (2.425, 2.075) {};
		\node [style=none] (114) at (-0.325, 2.075) {};
		\node [style=none] (115) at (-0.325, -0.65) {};
		\node [style=label] (116) at (-2.25, 1.625) {$X$};
		\node [style={boxStyle/+}] (120) at (1.675, 1.375) {$c$};
		\node [style={boxStyle/+}] (121) at (1.675, 0.05) {$d$};
		\node [style=none] (122) at (1.675, 1.625) {};
		\node [style=none] (123) at (1.675, 1.125) {};
		\node [style=none] (124) at (1.675, 0.3) {};
		\node [style=none] (125) at (1.675, -0.2) {};
		\node [style={dotStyle/-}] (126) at (-1.325, -0.6) {};
		\node [{dotStyle/-}] (127) at (-0.825, -0.6) {};
		\node [style=none] (128) at (-0.325, -0.2) {};
		\node [style=none] (129) at (-0.325, -1.025) {};
		\node [style=none] (131) at (-1.825, -1.525) {};
		\node [style=none] (132) at (2.85, -2.225) {};
		\node [style=none] (133) at (2.85, 2.45) {};
		\node [style=none] (134) at (-1.825, 2.45) {};
		\node [style=none] (135) at (-1.825, -2.225) {};
		\node [style=label] (137) at (-2.25, -1.525) {$W$};
		\node [style={boxStyle/+}] (139) at (1.675, -1.275) {$e$};
		\node [style=none] (141) at (-0.325, -0.2) {};
		\node [style=none] (142) at (1.675, -1.025) {};
		\node [style=none] (143) at (1.675, -1.525) {};
		\node [style=none] (144) at (-1.825, 1.625) {};
	\end{pgfonlayer}
	\begin{pgfonlayer}{edgelayer}
		\draw [{bgStyle/-}] (134.center)
			 to (133.center)
			 to (132.center)
			 to (135.center)
			 to cycle;
		\draw [{bgStyle/+}] (114.center)
			 to (113.center)
			 to (112.center)
			 to (115.center)
			 to cycle;
		\draw [{wStyle/+}, bend left=45] (109.center) to (107);
		\draw [{wStyle/+}, bend left=45] (107) to (108.center);
		\draw [{wStyle/+}] (81) to (107);
		\draw [style={wStyle/+}] (109.center) to (124.center);
		\draw [style={wStyle/+}] (123.center) to (108.center);
		\draw [style={wStyle/+}] (110.center) to (122.center);
		\draw [style={wStyle/+}] (111.center) to (125.center);
		\draw [{wStyle/-}, bend left=45] (129.center) to (127);
		\draw [{wStyle/-}, bend left=45] (127) to (128.center);
		\draw [{wStyle/-}] (126) to (127);
		\draw [style={wStyle/-}] (129.center) to (142.center);
		\draw [style={wStyle/-}] (141.center) to (128.center);
		\draw [style={wStyle/-}] (131.center) to (143.center);
		\draw [style={wStyle/-}] (110.center) to (144.center);
	\end{pgfonlayer}
\end{tikzpicture}

%% file: tikz/exMid/proof.tex
\[
    \begin{array}{c@{\,}c@{\,}c@{\,}c@{\,}c@{\,}c@{\,}c@{\,}c@{\,}c}
        
    \InputIfFileExists{restrictedEM.tikz}{}{\input{tikz/restrictedEM.tikz}}
 
        &\structuralcong&
        \begin{tikzpicture}
            \begin{pgfonlayer}{nodelayer}
                \node [style={dotStyle/+}] (81) at (0.5, 0) {};
                \node [{dotStyle/+}] (107) at (0, 0) {};
                \node [style=none] (108) at (-0.75, -0.5) {};
                \node [style=none] (109) at (-0.75, 0.5) {};
                \node [style=none] (110) at (-3, -0.5) {};
                \node [style=none] (111) at (-3, 0.5) {};
                \node [style=none] (112) at (0.75, 1.125) {};
                \node [style=none] (113) at (0.75, -1.125) {};
                \node [style=none] (114) at (-2.5, -1.125) {};
                \node [style=none] (115) at (-2.5, 1.125) {};
                \node [style=label] (116) at (-3.425, -0.5) {$X$};
                \node [style=label] (117) at (-3.425, 0.5) {$X$};
                \node [style={boxStyle/-}] (119) at (-1.5, 0.5) {$c$};
                \node [style={boxStyle/+}] (120) at (-1.5, -0.5) {$c$};
                \node [style=none] (121) at (1, 1.375) {};
                \node [style=none] (122) at (1, -1.375) {};
                \node [style=none] (123) at (-3, -1.375) {};
                \node [style=none] (124) at (-3, 1.375) {};
                \node [style=none] (125) at (-2.5, -0.5) {};
                \node [style=none] (126) at (-2.5, 0.5) {};
            \end{pgfonlayer}
            \begin{pgfonlayer}{edgelayer}
                \draw [{bgStyle/-}] (123.center)
                     to (122.center)
                     to (121.center)
                     to (124.center)
                     to cycle;
                \draw [{bgStyle/+}] (114.center)
                     to (113.center)
                     to (112.center)
                     to (115.center)
                     to cycle;
                \draw [{wStyle/+}, bend left] (109.center) to (107);
                \draw [{wStyle/+}, bend left] (107) to (108.center);
                \draw [{wStyle/+}] (81) to (107);
                \draw (109.center) to (119);
                \draw [style={wStyle/+}] (108.center) to (120);
                \draw [style={wStyle/+}] (119) to (126.center);
                \draw [style={wStyle/+}] (125.center) to (120);
                \draw [style={wStyle/-}] (125.center) to (110.center);
                \draw [style={wStyle/-}] (111.center) to (126.center);
            \end{pgfonlayer}
        \end{tikzpicture}        
        &\stackrel{\eqref{ax:minusCocopyCopy}}{\leq}&
        \begin{tikzpicture}
            \begin{pgfonlayer}{nodelayer}
                \node [style={dotStyle/+}] (81) at (0.25, 0) {};
                \node [{dotStyle/+}] (107) at (-0.25, 0) {};
                \node [style=none] (108) at (-1, -0.5) {};
                \node [style=none] (109) at (-1, 0.5) {};
                \node [style=none] (110) at (-4.25, -0.5) {};
                \node [style=none] (111) at (-4.25, 0.5) {};
                \node [style=none] (112) at (0.5, 1.125) {};
                \node [style=none] (113) at (0.5, -1.125) {};
                \node [style=none] (114) at (-2.25, -1.125) {};
                \node [style=none] (115) at (-2.25, 1.125) {};
                \node [style=label] (116) at (-4.675, -0.5) {$X$};
                \node [style=label] (117) at (-4.675, 0.5) {$X$};
                \node [style={boxStyle/-}] (119) at (-1.5, 0.5) {$c$};
                \node [style={boxStyle/+}] (120) at (-1.5, -0.5) {$c$};
                \node [style=none] (121) at (0.75, 1.375) {};
                \node [style=none] (122) at (0.75, -1.375) {};
                \node [style=none] (123) at (-4.25, -1.375) {};
                \node [style=none] (124) at (-4.25, 1.375) {};
                \node [style=none] (125) at (-2.25, -0.5) {};
                \node [style=none] (126) at (-2.25, 0.5) {};
                \node [style={dotStyle/-}] (127) at (-3.5, 0) {};
                \node [style={dotStyle/-}] (128) at (-3, 0) {};
            \end{pgfonlayer}
            \begin{pgfonlayer}{edgelayer}
                \draw [{bgStyle/-}] (123.center)
                    to (122.center)
                    to (121.center)
                    to (124.center)
                    to cycle;
                \draw [{bgStyle/+}] (114.center)
                    to (113.center)
                    to (112.center)
                    to (115.center)
                    to cycle;
                \draw [{wStyle/+}, bend left] (109.center) to (107);
                \draw [{wStyle/+}, bend left] (107) to (108.center);
                \draw [{wStyle/+}] (81) to (107);
                \draw (109.center) to (119);
                \draw [style={wStyle/+}] (108.center) to (120);
                \draw [style={wStyle/+}] (119) to (126.center);
                \draw [style={wStyle/+}] (125.center) to (120);
                \draw [style={wStyle/-}] (128) to (127);
                \draw [style={wStyle/-}, bend left] (128) to (126.center);
                \draw [style={wStyle/-}, bend right] (128) to (125.center);
                \draw [style={wStyle/-}, bend left] (127) to (110.center);
                \draw [style={wStyle/-}, bend right] (127) to (111.center);
            \end{pgfonlayer}
        \end{tikzpicture}
        &\stackrel{\eqref{eq:prop mappe}}{=}&
        \begin{tikzpicture}
            \begin{pgfonlayer}{nodelayer}
                \node [style={dotStyle/+}] (81) at (0.5, 0) {};
                \node [{dotStyle/+}] (107) at (0, 0) {};
                \node [style=none] (108) at (-0.75, -0.5) {};
                \node [style=none] (109) at (-0.75, 0.5) {};
                \node [style=none] (110) at (-4, -0.5) {};
                \node [style=none] (111) at (-4, 0.5) {};
                \node [style=none] (112) at (0.75, 1.125) {};
                \node [style=none] (113) at (0.75, -1.125) {};
                \node [style=none] (114) at (-2.75, -1.125) {};
                \node [style=none] (115) at (-2.75, 1.125) {};
                \node [style=label] (116) at (-4.425, -0.5) {$X$};
                \node [style=label] (117) at (-4.425, 0.5) {$X$};
                \node [style={boxStyle/-}] (119) at (-1.15, 0.5) {$c$};
                \node [style={boxStyle/+}] (120) at (-1.15, -0.5) {$c$};
                \node [style=none] (121) at (1, 1.375) {};
                \node [style=none] (122) at (1, -1.375) {};
                \node [style=none] (123) at (-4, -1.375) {};
                \node [style=none] (124) at (-4, 1.375) {};
                \node [style=none] (125) at (-1.5, -0.5) {};
                \node [style=none] (126) at (-1.5, 0.5) {};
                \node [style={dotStyle/-}] (127) at (-3.25, 0) {};
                \node [style={dotStyle/+}] (128) at (-2.25, 0) {};
                \node [style=none] (129) at (-2.75, 0) {};
            \end{pgfonlayer}
            \begin{pgfonlayer}{edgelayer}
                \draw [{bgStyle/-}] (123.center)
                     to (122.center)
                     to (121.center)
                     to (124.center)
                     to cycle;
                \draw [{bgStyle/+}] (114.center)
                     to (113.center)
                     to (112.center)
                     to (115.center)
                     to cycle;
                \draw [{wStyle/+}, bend left] (109.center) to (107);
                \draw [{wStyle/+}, bend left] (107) to (108.center);
                \draw [{wStyle/+}] (81) to (107);
                \draw (109.center) to (119);
                \draw [style={wStyle/+}] (108.center) to (120);
                \draw [style={wStyle/+}] (119) to (126.center);
                \draw [style={wStyle/+}] (125.center) to (120);
                \draw [style={wStyle/+}, bend left] (128) to (126.center);
                \draw [style={wStyle/+}, bend right] (128) to (125.center);
                \draw [style={wStyle/-}, bend left] (127) to (110.center);
                \draw [style={wStyle/-}, bend right] (127) to (111.center);
                \draw [style={wStyle/-}] (129.center) to (127);
                \draw [style={wStyle/+}] (129.center) to (128);
            \end{pgfonlayer}
        \end{tikzpicture}        
        &
        &
        \\
        &\stackrel{\text{\Cref{def:peircean-bicategory}.1}}{=}&
        \begin{tikzpicture}
            \begin{pgfonlayer}{nodelayer}
                \node [style={dotStyle/+}] (81) at (-0.75, 0) {};
                \node [{dotStyle/-}] (107) at (-1.75, 0) {};
                \node [style=none] (110) at (-4.25, -0.5) {};
                \node [style=none] (111) at (-4.25, 0.5) {};
                \node [style=none] (112) at (-0.5, 1.125) {};
                \node [style=none] (113) at (-0.5, -1.125) {};
                \node [style=none] (114) at (-3, -1.125) {};
                \node [style=none] (115) at (-3, 1.125) {};
                \node [style=label] (116) at (-4.675, -0.5) {$X$};
                \node [style=label] (117) at (-4.675, 0.5) {$X$};
                \node [style=none] (121) at (-0.25, 1.375) {};
                \node [style=none] (122) at (-0.25, -1.375) {};
                \node [style=none] (123) at (-4.25, -1.375) {};
                \node [style=none] (124) at (-4.25, 1.375) {};
                \node [style={dotStyle/-}] (127) at (-3.5, 0) {};
                \node [style={dotStyle/-}] (128) at (-2.25, 0) {};
                \node [style=none] (129) at (-3, 0) {};
                \node [style=none] (130) at (-1.25, 0.625) {};
                \node [style=none] (131) at (-1.25, -0.625) {};
                \node [style=none] (132) at (-2.75, -0.625) {};
                \node [style=none] (133) at (-2.75, 0.625) {};
                \node [style=none] (134) at (-1.25, 0) {};
                \node [style=none] (135) at (-2.75, 0) {};
            \end{pgfonlayer}
            \begin{pgfonlayer}{edgelayer}
                \draw [{bgStyle/-}] (123.center)
                    to (122.center)
                    to (121.center)
                    to (124.center)
                    to cycle;
                \draw [{bgStyle/+}] (114.center)
                    to (113.center)
                    to (112.center)
                    to (115.center)
                    to cycle;
                \draw [style={wStyle/-}, bend left] (127) to (110.center);
                \draw [style={wStyle/-}, bend right] (127) to (111.center);
                \draw [style={wStyle/-}] (129.center) to (127);
                \draw [{bgStyle/-}] (132.center)
                    to (131.center)
                    to (130.center)
                    to (133.center)
                    to cycle;
                \draw [style={wStyle/+}] (135.center) to (129.center);
                \draw [style={wStyle/+}] (134.center) to (81);
                \draw [style={wStyle/-}] (134.center) to (107);
                \draw [style={wStyle/-}] (128) to (135.center);
            \end{pgfonlayer}
        \end{tikzpicture}
        &\stackrel{\text{\Cref{lemma:comaps}}}{=}&
        \begin{tikzpicture}
            \begin{pgfonlayer}{nodelayer}
                \node [style={dotStyle/+}] (81) at (0.25, 0) {};
                \node [style=none] (110) at (-4, -0.5) {};
                \node [style=none] (111) at (-4, 0.5) {};
                \node [style=none] (112) at (0.75, 1.125) {};
                \node [style=none] (113) at (0.75, -1.125) {};
                \node [style=none] (114) at (-2.75, -1.125) {};
                \node [style=none] (115) at (-2.75, 1.125) {};
                \node [style=label] (116) at (-4.425, -0.5) {$X$};
                \node [style=label] (117) at (-4.425, 0.5) {$X$};
                \node [style=none] (121) at (1, 1.375) {};
                \node [style=none] (122) at (1, -1.375) {};
                \node [style=none] (123) at (-4, -1.375) {};
                \node [style=none] (124) at (-4, 1.375) {};
                \node [style={dotStyle/-}] (127) at (-3.25, 0) {};
                \node [style={dotStyle/-}] (128) at (-1.5, 0) {};
                \node [style=none] (129) at (-2.75, 0) {};
                \node [style=none] (130) at (-1, 0.625) {};
                \node [style=none] (131) at (-1, -0.625) {};
                \node [style=none] (132) at (-2.25, -0.625) {};
                \node [style=none] (133) at (-2.25, 0.625) {};
                \node [style={dotStyle/+}] (134) at (-0.5, 0) {};
                \node [style=none] (135) at (-2.25, 0) {};
            \end{pgfonlayer}
            \begin{pgfonlayer}{edgelayer}
                \draw [{bgStyle/-}] (123.center)
                     to (122.center)
                     to (121.center)
                     to (124.center)
                     to cycle;
                \draw [{bgStyle/+}] (114.center)
                     to (113.center)
                     to (112.center)
                     to (115.center)
                     to cycle;
                \draw [style={wStyle/-}, bend left] (127) to (110.center);
                \draw [style={wStyle/-}, bend right] (127) to (111.center);
                \draw [style={wStyle/-}] (129.center) to (127);
                \draw [{bgStyle/-}] (132.center)
                     to (131.center)
                     to (130.center)
                     to (133.center)
                     to cycle;
                \draw [style={wStyle/+}] (135.center) to (129.center);
                \draw [style={wStyle/+}] (134) to (81);
                \draw [style={wStyle/-}] (128) to (135.center);
            \end{pgfonlayer}
        \end{tikzpicture}        
        &\stackrel{\eqref{ax:plusCodiscDisc}}{\leq}&
        \begin{tikzpicture}
            \begin{pgfonlayer}{nodelayer}
                \node [style=none] (110) at (-4, -0.5) {};
                \node [style=none] (111) at (-4, 0.5) {};
                \node [style=none] (112) at (0.75, 1.125) {};
                \node [style=none] (113) at (0.75, -1.125) {};
                \node [style=none] (114) at (-2.75, -1.125) {};
                \node [style=none] (115) at (-2.75, 1.125) {};
                \node [style=label] (116) at (-4.425, -0.5) {$X$};
                \node [style=label] (117) at (-4.425, 0.5) {$X$};
                \node [style=none] (121) at (1, 1.375) {};
                \node [style=none] (122) at (1, -1.375) {};
                \node [style=none] (123) at (-4, -1.375) {};
                \node [style=none] (124) at (-4, 1.375) {};
                \node [style={dotStyle/-}] (127) at (-3.25, 0) {};
                \node [style={dotStyle/-}] (128) at (-1.5, 0) {};
                \node [style=none] (129) at (-2.75, 0) {};
                \node [style=none] (130) at (-1, 0.625) {};
                \node [style=none] (131) at (-1, -0.625) {};
                \node [style=none] (132) at (-2.25, -0.625) {};
                \node [style=none] (133) at (-2.25, 0.625) {};
                \node [style=none] (135) at (-2.25, 0) {};
            \end{pgfonlayer}
            \begin{pgfonlayer}{edgelayer}
                \draw [{bgStyle/-}] (123.center)
                     to (122.center)
                     to (121.center)
                     to (124.center)
                     to cycle;
                \draw [{bgStyle/+}] (114.center)
                     to (113.center)
                     to (112.center)
                     to (115.center)
                     to cycle;
                \draw [style={wStyle/-}, bend left] (127) to (110.center);
                \draw [style={wStyle/-}, bend right] (127) to (111.center);
                \draw [style={wStyle/-}] (129.center) to (127);
                \draw [{bgStyle/-}] (132.center)
                     to (131.center)
                     to (130.center)
                     to (133.center)
                     to cycle;
                \draw [style={wStyle/+}] (135.center) to (129.center);
                \draw [style={wStyle/-}] (128) to (135.center);
            \end{pgfonlayer}
        \end{tikzpicture}         
        &\structuralcong&
        \capCirc[-][X]            
    \end{array}
\]

%% file: tikz/linAdjoints/proof.tex
\[
\begin{array}{c@{}c@{}c@{}c@{}c@{}c@{}c}
    c \seq[+] \op{(\nega{c})} 
    &\stackrel{\eqref{eqdagger}}{=}&
    \begin{tikzpicture}
        \begin{pgfonlayer}{nodelayer}
            \node [style=none] (81) at (-1.425, -0.75) {};
            \node [style=none] (112) at (-1.425, 1.75) {};
            \node [style=none] (113) at (-1.425, -1.5) {};
            \node [style=none] (114) at (2.175, -1.5) {};
            \node [style=none] (115) at (2.175, 1.75) {};
            \node [style=none] (121) at (2.175, 1.25) {};
            \node [{boxStyle/+}] (122) at (-0.675, -0.75) {$c$};
            \node [{boxStyle/-}] (123) at (0.75, 0.25) {$c$};
            \node [style=label] (124) at (-1.925, -0.75) {$X$};
            \node [style=label] (125) at (2.675, 1.25) {$X$};
            \node [style={dotStyle/+}] (126) at (1.5, -0.25) {};
            \node [style={dotStyle/+}] (127) at (1.925, -0.25) {};
            \node [style=none] (128) at (1, 0.25) {};
            \node [style=none] (129) at (1, -0.75) {};
            \node [style={dotStyle/+}] (130) at (0, 0.75) {};
            \node [style={dotStyle/+}] (131) at (-0.425, 0.75) {};
            \node [style=none] (132) at (0.5, 1.25) {};
            \node [style=none] (133) at (0.5, 0.25) {};
        \end{pgfonlayer}
        \begin{pgfonlayer}{edgelayer}
            \draw [{bgStyle/+}] (114.center)
                 to (113.center)
                 to (112.center)
                 to (115.center)
                 to cycle;
            \draw [{wStyle/+}] (81.center) to (122);
            \draw [style={wStyle/+}] (128.center) to (123);
            \draw [style={wStyle/+}, bend left=45] (128.center) to (126);
            \draw [style={wStyle/+}] (126) to (127);
            \draw [style={wStyle/+}, bend left=45] (126) to (129.center);
            \draw [style={wStyle/+}] (129.center) to (122);
            \draw [style={wStyle/+}, bend right=45] (132.center) to (130);
            \draw [style={wStyle/+}] (130) to (131);
            \draw [style={wStyle/+}, bend right=45] (130) to (133.center);
            \draw [style={wStyle/+}] (123) to (133.center);
            \draw [style={wStyle/+}] (121.center) to (132.center);
        \end{pgfonlayer}
    \end{tikzpicture}    
    &\stackrel{\text{\Cref{prop:restricted EM}.1}}{\leq}&
    \begin{tikzpicture}
        \begin{pgfonlayer}{nodelayer}
            \node [style=none] (81) at (-1.675, -0.75) {};
            \node [style=none] (112) at (-1.675, 1.75) {};
            \node [style=none] (113) at (-1.675, -1.5) {};
            \node [style=none] (114) at (0.925, -1.5) {};
            \node [style=none] (115) at (0.925, 1.75) {};
            \node [style=none] (121) at (0.925, 1.25) {};
            \node [style=label] (124) at (-2.175, -0.75) {$X$};
            \node [style=label] (125) at (1.425, 1.25) {$X$};
            \node [style={dotStyle/-}] (126) at (0, -0.25) {};
            \node [style={dotStyle/-}] (127) at (0.425, -0.25) {};
            \node [style=none] (128) at (-0.5, 0.25) {};
            \node [style=none] (129) at (-0.5, -0.75) {};
            \node [style={dotStyle/+}] (130) at (-1, 0.75) {};
            \node [style={dotStyle/+}] (131) at (-1.425, 0.75) {};
            \node [style=none] (132) at (-0.5, 1.25) {};
            \node [style=none] (133) at (-0.5, 0.25) {};
            \node [style=none] (134) at (-0.5, 0.5) {};
            \node [style=none] (135) at (-0.5, -1) {};
            \node [style=none] (136) at (0.675, -1) {};
            \node [style=none] (137) at (0.675, 0.5) {};
            \node [style=none] (138) at (-0.5, 0.25) {};
            \node [style=none] (139) at (-0.5, -0.75) {};
        \end{pgfonlayer}
        \begin{pgfonlayer}{edgelayer}
            \draw [{bgStyle/+}] (114.center)
                 to (113.center)
                 to (112.center)
                 to (115.center)
                 to cycle;
            \draw [{bgStyle/-}] (136.center)
                 to (135.center)
                 to (134.center)
                 to (137.center)
                 to cycle;
            \draw [style={wStyle/-}, bend left=45] (128.center) to (126);
            \draw [style={wStyle/-}] (126) to (127);
            \draw [style={wStyle/-}, bend left=45] (126) to (129.center);
            \draw [style={wStyle/+}, bend right=45] (132.center) to (130);
            \draw [style={wStyle/+}] (130) to (131);
            \draw [style={wStyle/+}, bend right=45] (130) to (133.center);
            \draw [style={wStyle/+}] (121.center) to (132.center);
            \draw [style={wStyle/+}] (138.center) to (133.center);
            \draw [style={wStyle/+}] (139.center) to (81.center);
        \end{pgfonlayer}
    \end{tikzpicture}    
    &\stackrel{\text{\Cref{cor:cc are the same}}}{=}&
    \begin{tikzpicture}
        \begin{pgfonlayer}{nodelayer}
            \node [style=none] (81) at (-1.675, -1) {};
            \node [style=none] (112) at (-1.675, 1.5) {};
            \node [style=none] (113) at (-1.675, -1.5) {};
            \node [style=none] (114) at (1.675, -1.5) {};
            \node [style=none] (115) at (1.675, 1.5) {};
            \node [style=none] (121) at (1.675, 1) {};
            \node [style=label] (124) at (-2.175, -1) {$X$};
            \node [style=label] (125) at (2.175, 1) {$X$};
            \node [style={dotStyle/+}] (126) at (1, -0.5) {};
            \node [style={dotStyle/+}] (127) at (1.425, -0.5) {};
            \node [style=none] (128) at (0.5, 0) {};
            \node [style=none] (129) at (0.5, -1) {};
            \node [style={dotStyle/+}] (130) at (-1, 0.5) {};
            \node [style={dotStyle/+}] (131) at (-1.425, 0.5) {};
            \node [style=none] (132) at (-0.5, 1) {};
            \node [style=none] (133) at (-0.5, 0) {};
            \node [style=none] (134) at (-0.5, 0.25) {};
            \node [style=none] (135) at (-0.5, -0.25) {};
            \node [style=none] (136) at (0.5, -0.25) {};
            \node [style=none] (137) at (0.5, 0.25) {};
            \node [style=none] (138) at (-0.5, 0) {};
            \node [style=none] (139) at (-0.5, -1) {};
        \end{pgfonlayer}
        \begin{pgfonlayer}{edgelayer}
            \draw [{bgStyle/+}] (114.center)
                 to (113.center)
                 to (112.center)
                 to (115.center)
                 to cycle;
            \draw [{bgStyle/-}] (136.center)
                 to (135.center)
                 to (134.center)
                 to (137.center)
                 to cycle;
            \draw [style={wStyle/+}, bend left=45] (128.center) to (126);
            \draw [style={wStyle/+}] (126) to (127);
            \draw [style={wStyle/+}, bend left=45] (126) to (129.center);
            \draw [style={wStyle/+}, bend right=45] (132.center) to (130);
            \draw [style={wStyle/+}] (130) to (131);
            \draw [style={wStyle/+}, bend right=45] (130) to (133.center);
            \draw [style={wStyle/+}] (121.center) to (132.center);
            \draw [style={wStyle/+}] (138.center) to (133.center);
            \draw [style={wStyle/+}] (139.center) to (81.center);
            \draw [style={wStyle/-}] (128.center) to (138.center);
            \draw [style={wStyle/+}] (129.center) to (139.center);
        \end{pgfonlayer}
    \end{tikzpicture}    
    \\
    &\stackrel{\text{\Cref{cor:daggers are the same}}}{=}&
    \begin{tikzpicture}
        \begin{pgfonlayer}{nodelayer}
            \node [style=none] (81) at (-1.675, -1) {};
            \node [style=none] (112) at (-1.675, 1.5) {};
            \node [style=none] (113) at (-1.675, -1.5) {};
            \node [style=none] (114) at (1.925, -1.5) {};
            \node [style=none] (115) at (1.925, 1.5) {};
            \node [style=none] (121) at (1.925, 1) {};
            \node [style=label] (124) at (-2.175, -1) {$X$};
            \node [style=label] (125) at (2.425, 1) {$X$};
            \node [style={dotStyle/-}] (126) at (1.25, -0.5) {};
            \node [style={dotStyle/-}] (127) at (1.675, -0.5) {};
            \node [style=none] (128) at (0.75, 0) {};
            \node [style=none] (129) at (0.75, -1) {};
            \node [style={dotStyle/-}] (130) at (-1, 0.5) {};
            \node [style={dotStyle/-}] (131) at (-1.425, 0.5) {};
            \node [style=none] (132) at (-0.5, 1) {};
            \node [style=none] (133) at (-0.5, 0) {};
            \node [style=none] (138) at (-0.5, 0) {};
            \node [style=none] (139) at (-0.5, -1) {};
        \end{pgfonlayer}
        \begin{pgfonlayer}{edgelayer}
            \draw [{bgStyle/-}] (114.center)
                 to (113.center)
                 to (112.center)
                 to (115.center)
                 to cycle;
            \draw [style={wStyle/-}, bend left=45] (128.center) to (126);
            \draw [style={wStyle/-}] (126) to (127);
            \draw [style={wStyle/-}, bend left=45] (126) to (129.center);
            \draw [style={wStyle/-}, bend right=45] (132.center) to (130);
            \draw [style={wStyle/-}] (130) to (131);
            \draw [style={wStyle/-}, bend right=45] (130) to (133.center);
            \draw [style={wStyle/-}] (121.center) to (132.center);
            \draw [style={wStyle/-}] (138.center) to (133.center);
            \draw [style={wStyle/-}] (139.center) to (81.center);
            \draw [style={wStyle/-}] (128.center) to (138.center);
            \draw [style={wStyle/-}] (129.center) to (139.center);
        \end{pgfonlayer}
    \end{tikzpicture}    
    &\stackrel{\text{\Cref{th:spider}}}{=}&
    \idCirc[-][X]
    &=&
    \id[-][X]           
\end{array}
\]

%% file: tikz/frobAux/proof.tex
We prove it by means of \Cref{prop:cap property} as follows.
\begin{align*}
    \begin{tikzpicture}
        \begin{pgfonlayer}{nodelayer}
            \node [{dotStyle/+}] (107) at (-1, 0.75) {};
            \node [style=none] (108) at (-0.5, 1.125) {};
            \node [style=none] (109) at (-0.5, 0.375) {};
            \node [style=none] (122) at (-1.925, -2) {};
            \node [style=none] (123) at (-1.925, 1.5) {};
            \node [style=none] (124) at (3, 1.5) {};
            \node [style=none] (125) at (3, -2) {};
            \node [style={dotStyle/+}] (127) at (-1.425, 0.75) {};
            \node [{dotStyle/-}] (128) at (0, 0) {};
            \node [style=none] (129) at (-0.5, 0.375) {};
            \node [style=none] (130) at (-0.5, -0.375) {};
            \node [style=none] (131) at (0.425, 0) {};
            \node [style=none] (132) at (0.5, 1.125) {};
            \node [style=none] (133) at (-1.175, -0.375) {};
            \node [style=none] (134) at (-0.5, -0.5) {};
            \node [style=none] (135) at (-0.5, 0.5) {};
            \node [style=none] (136) at (0.425, 0.5) {};
            \node [style=none] (137) at (0.425, -0.5) {};
            \node [style=none] (138) at (0.425, 0) {};
            \node [style=label] (139) at (-2.4, -0.925) {$X$};
            \node [style=label] (140) at (-2.4, -0.3) {$X$};
            \node [style=label] (142) at (-2.4, -1.525) {$X$};
            \node [{dotStyle/+}] (143) at (1.15, -0.4) {};
            \node [style=none] (144) at (0.4, -0.85) {};
            \node [style=none] (145) at (0.4, 0) {};
            \node [style={dotStyle/+}] (146) at (1.575, -0.4) {};
            \node [style=none] (147) at (-1.175, -0.85) {};
            \node [style=none] (148) at (-1.925, -0.375) {};
            \node [style=none] (149) at (-1.925, -0.85) {};
            \node [{dotStyle/+}] (153) at (2.225, -0.175) {};
            \node [style=none] (154) at (0.75, -1.475) {};
            \node [style=none] (155) at (0.75, 1.125) {};
            \node [style={dotStyle/+}] (156) at (2.65, -0.175) {};
            \node [style=none] (157) at (-1.925, -1.475) {};
        \end{pgfonlayer}
        \begin{pgfonlayer}{edgelayer}
            \draw [{bgStyle/+}] (124.center)
                 to (123.center)
                 to (122.center)
                 to (125.center)
                 to cycle;
            \draw [{bgStyle/-}] (136.center)
                 to (135.center)
                 to (134.center)
                 to (137.center)
                 to cycle;
            \draw [{wStyle/+}, bend left] (109.center) to (107);
            \draw [{wStyle/+}, bend left] (107) to (108.center);
            \draw [{wStyle/+}] (107) to (127);
            \draw [{wStyle/-}, bend right] (130.center) to (128);
            \draw [{wStyle/-}, bend right] (128) to (129.center);
            \draw [{wStyle/-}] (128) to (131.center);
            \draw [{wStyle/+}] (132.center) to (108.center);
            \draw [{wStyle/+}] (130.center) to (133.center);
            \draw [style={wStyle/+}] (138.center) to (131.center);
            \draw [{wStyle/+}, bend left] (145.center) to (143);
            \draw [{wStyle/+}, bend left] (143) to (144.center);
            \draw [{wStyle/+}] (143) to (146);
            \draw [{wStyle/+}] (147.center) to (144.center);
            \draw [style={wStyle/+}] (138.center) to (145.center);
            \draw [style={wStyle/+}, in=0, out=-180] (133.center) to (149.center);
            \draw [style={wStyle/+}, in=-180, out=0] (148.center) to (147.center);
            \draw [{wStyle/+}, bend left=45] (155.center) to (153);
            \draw [{wStyle/+}, bend left=45] (153) to (154.center);
            \draw [{wStyle/+}] (153) to (156);
            \draw (132.center) to (155.center);
            \draw (157.center) to (154.center);
        \end{pgfonlayer}
    \end{tikzpicture}
    &\stackrel{\text{\Cref{th:spider}}}{=}
    \begin{tikzpicture}
        \begin{pgfonlayer}{nodelayer}
            \node [style=none] (109) at (0.5, 0.625) {};
            \node [style=none] (122) at (-2, -1) {};
            \node [style=none] (123) at (-2, 1) {};
            \node [style=none] (124) at (3, 1) {};
            \node [style=none] (125) at (3, -1) {};
            \node [{dotStyle/-}] (128) at (1, 0.25) {};
            \node [style=none] (129) at (0.5, 0.625) {};
            \node [style=none] (130) at (0.5, -0.125) {};
            \node [style=none] (131) at (1.425, 0.25) {};
            \node [style=none] (133) at (0.325, -0.125) {};
            \node [style=none] (134) at (0.5, -0.25) {};
            \node [style=none] (135) at (0.5, 0.75) {};
            \node [style=none] (136) at (1.425, 0.75) {};
            \node [style=none] (137) at (1.425, -0.25) {};
            \node [style=none] (138) at (1.425, 0.25) {};
            \node [style=label] (139) at (-2.475, 0) {$X$};
            \node [style=label] (140) at (-2.475, 0.625) {$X$};
            \node [style=label] (142) at (-2.475, -0.6) {$X$};
            \node [{dotStyle/+}] (143) at (2.15, -0.15) {};
            \node [style=none] (144) at (1.4, -0.6) {};
            \node [style=none] (145) at (1.4, 0.25) {};
            \node [style={dotStyle/+}] (146) at (2.575, -0.15) {};
            \node [style=none] (147) at (0.325, -0.6) {};
            \node [style=none] (148) at (-0.425, -0.125) {};
            \node [style=none] (149) at (-0.425, -0.6) {};
            \node [style=none] (158) at (-0.45, 0.625) {};
            \node [style=none] (161) at (-1.25, 0.625) {};
            \node [style=none] (163) at (-1.25, -0.125) {};
            \node [style=none] (164) at (-1.25, -0.6) {};
            \node [style=none] (165) at (-2, -0.125) {};
            \node [style=none] (166) at (-2, -0.6) {};
            \node [style=none] (167) at (-2, 0.625) {};
        \end{pgfonlayer}
        \begin{pgfonlayer}{edgelayer}
            \draw [{bgStyle/+}] (124.center)
                 to (123.center)
                 to (122.center)
                 to (125.center)
                 to cycle;
            \draw [{bgStyle/-}] (136.center)
                 to (135.center)
                 to (134.center)
                 to (137.center)
                 to cycle;
            \draw [{wStyle/-}, bend right] (130.center) to (128);
            \draw [{wStyle/-}, bend right] (128) to (129.center);
            \draw [{wStyle/-}] (128) to (131.center);
            \draw [{wStyle/+}] (130.center) to (133.center);
            \draw [style={wStyle/+}] (138.center) to (131.center);
            \draw [{wStyle/+}, bend left] (145.center) to (143);
            \draw [{wStyle/+}, bend left] (143) to (144.center);
            \draw [{wStyle/+}] (143) to (146);
            \draw [{wStyle/+}] (147.center) to (144.center);
            \draw [style={wStyle/+}] (138.center) to (145.center);
            \draw [style={wStyle/+}, in=0, out=-180] (133.center) to (149.center);
            \draw [style={wStyle/+}, in=-180, out=0] (148.center) to (147.center);
            \draw (158.center) to (129.center);
            \draw [style={wStyle/+}, in=0, out=-180] (163.center) to (166.center);
            \draw [style={wStyle/+}, in=-180, out=0] (165.center) to (164.center);
            \draw (164.center) to (149.center);
            \draw (167.center) to (161.center);
            \draw [in=-180, out=0] (163.center) to (158.center);
            \draw [in=-180, out=0, looseness=1.25] (161.center) to (148.center);
        \end{pgfonlayer}
    \end{tikzpicture}
    \stackrel{\text{\Cref{cor:cc are the same}}}{=}
    \begin{tikzpicture}
        \begin{pgfonlayer}{nodelayer}
            \node [style=none] (109) at (0.3, 0.625) {};
            \node [style=none] (122) at (-2, -1) {};
            \node [style=none] (123) at (-2, 1) {};
            \node [style=none] (124) at (2.575, 1) {};
            \node [style=none] (125) at (2.575, -1) {};
            \node [{dotStyle/-}] (128) at (0.8, 0.25) {};
            \node [style=none] (129) at (0.3, 0.625) {};
            \node [style=none] (130) at (0.3, -0.125) {};
            \node [style=none] (131) at (0.825, 0.25) {};
            \node [style=none] (133) at (0.3, -0.125) {};
            \node [style=none] (134) at (0.3, -0.75) {};
            \node [style=none] (135) at (0.3, 0.75) {};
            \node [style=none] (136) at (2.325, 0.75) {};
            \node [style=none] (137) at (2.325, -0.75) {};
            \node [style=none] (138) at (0.825, 0.25) {};
            \node [style=label] (139) at (-2.475, 0) {$X$};
            \node [style=label] (140) at (-2.475, 0.625) {$X$};
            \node [style=label] (142) at (-2.475, -0.6) {$X$};
            \node [{dotStyle/-}] (143) at (1.55, -0.15) {};
            \node [style=none] (144) at (0.8, -0.6) {};
            \node [style=none] (145) at (0.8, 0.25) {};
            \node [style={dotStyle/-}] (146) at (1.975, -0.15) {};
            \node [style=none] (147) at (0.3, -0.6) {};
            \node [style=none] (148) at (-0.425, -0.125) {};
            \node [style=none] (149) at (-0.425, -0.6) {};
            \node [style=none] (158) at (-0.45, 0.625) {};
            \node [style=none] (161) at (-1.25, 0.625) {};
            \node [style=none] (163) at (-1.25, -0.125) {};
            \node [style=none] (164) at (-1.25, -0.6) {};
            \node [style=none] (165) at (-2, -0.125) {};
            \node [style=none] (166) at (-2, -0.6) {};
            \node [style=none] (167) at (-2, 0.625) {};
        \end{pgfonlayer}
        \begin{pgfonlayer}{edgelayer}
            \draw [{bgStyle/+}] (124.center)
                 to (123.center)
                 to (122.center)
                 to (125.center)
                 to cycle;
            \draw [{bgStyle/-}] (136.center)
                 to (135.center)
                 to (134.center)
                 to (137.center)
                 to cycle;
            \draw [{wStyle/-}, bend right] (130.center) to (128);
            \draw [{wStyle/-}, bend right] (128) to (129.center);
            \draw [{wStyle/-}] (128) to (131.center);
            \draw [{wStyle/+}] (130.center) to (133.center);
            \draw [style={wStyle/+}] (138.center) to (131.center);
            \draw [{wStyle/-}, bend left] (145.center) to (143);
            \draw [{wStyle/-}, bend left] (143) to (144.center);
            \draw [{wStyle/-}] (143) to (146);
            \draw [{wStyle/-}] (147.center) to (144.center);
            \draw [style={wStyle/+}, in=0, out=-180] (133.center) to (149.center);
            \draw [style={wStyle/+}, in=-180, out=0] (148.center) to (147.center);
            \draw (158.center) to (129.center);
            \draw [style={wStyle/+}, in=0, out=-180] (163.center) to (166.center);
            \draw [style={wStyle/+}, in=-180, out=0] (165.center) to (164.center);
            \draw (164.center) to (149.center);
            \draw (167.center) to (161.center);
            \draw [in=-180, out=0] (163.center) to (158.center);
            \draw [in=-180, out=0, looseness=1.25] (161.center) to (148.center);
        \end{pgfonlayer}
    \end{tikzpicture}    
    \stackrel{\eqref{eq:prop mappe}}{=}
    \begin{tikzpicture}
        \begin{pgfonlayer}{nodelayer}
            \node [style=none] (109) at (0.5, 0.625) {};
            \node [style=none] (122) at (-2, -1) {};
            \node [style=none] (123) at (-2, 1) {};
            \node [style=none] (124) at (2.525, 1) {};
            \node [style=none] (125) at (2.525, -1) {};
            \node [{dotStyle/-}] (128) at (1, 0.25) {};
            \node [style=none] (129) at (0.5, 0.625) {};
            \node [style=none] (130) at (0.5, -0.125) {};
            \node [style=none] (131) at (1.025, 0.25) {};
            \node [style=none] (133) at (0.325, -0.125) {};
            \node [style=none] (138) at (1.025, 0.25) {};
            \node [style=label] (139) at (-2.475, 0) {$X$};
            \node [style=label] (140) at (-2.475, 0.625) {$X$};
            \node [style=label] (142) at (-2.475, -0.6) {$X$};
            \node [{dotStyle/-}] (143) at (1.75, -0.15) {};
            \node [style=none] (144) at (1, -0.6) {};
            \node [style=none] (145) at (1, 0.25) {};
            \node [style={dotStyle/-}] (146) at (2.175, -0.15) {};
            \node [style=none] (147) at (0.325, -0.6) {};
            \node [style=none] (148) at (-0.425, -0.125) {};
            \node [style=none] (149) at (-0.425, -0.6) {};
            \node [style=none] (158) at (-0.45, 0.625) {};
            \node [style=none] (161) at (-1.25, 0.625) {};
            \node [style=none] (163) at (-1.25, -0.125) {};
            \node [style=none] (164) at (-1.25, -0.6) {};
            \node [style=none] (165) at (-2, -0.125) {};
            \node [style=none] (166) at (-2, -0.6) {};
            \node [style=none] (167) at (-2, 0.625) {};
        \end{pgfonlayer}
        \begin{pgfonlayer}{edgelayer}
            \draw [{bgStyle/-}] (124.center)
                 to (123.center)
                 to (122.center)
                 to (125.center)
                 to cycle;
            \draw [{wStyle/-}, bend right] (130.center) to (128);
            \draw [{wStyle/-}, bend right] (128) to (129.center);
            \draw [{wStyle/-}] (128) to (131.center);
            \draw [{wStyle/-}] (130.center) to (133.center);
            \draw [style={wStyle/-}] (138.center) to (131.center);
            \draw [{wStyle/-}, bend left] (145.center) to (143);
            \draw [{wStyle/-}, bend left] (143) to (144.center);
            \draw [{wStyle/-}] (143) to (146);
            \draw [{wStyle/-}] (147.center) to (144.center);
            \draw [style={wStyle/-}, in=0, out=-180] (133.center) to (149.center);
            \draw [style={wStyle/-}, in=-180, out=0] (148.center) to (147.center);
            \draw [style={wStyle/-}](158.center) to (129.center);
            \draw [style={wStyle/-}, in=0, out=-180] (163.center) to (166.center);
            \draw [style={wStyle/-}, in=-180, out=0] (165.center) to (164.center);
            \draw [style={wStyle/-}] (164.center) to (149.center);
            \draw [style={wStyle/-}] (167.center) to (161.center);
            \draw [style={wStyle/-}][in=-180, out=0] (163.center) to (158.center);
            \draw [style={wStyle/-}] [in=-180, out=0, looseness=1.25] (161.center) to (148.center);
        \end{pgfonlayer}
    \end{tikzpicture}        
    \\
    &\stackrel{\text{\Cref{th:spider}}}{=}
    \begin{tikzpicture}
        \begin{pgfonlayer}{nodelayer}
            \node [style=none] (109) at (-2, 0.625) {};
            \node [style=none] (122) at (-2, -1) {};
            \node [style=none] (123) at (-2, 1) {};
            \node [style=none] (124) at (0.025, 1) {};
            \node [style=none] (125) at (0.025, -1) {};
            \node [{dotStyle/-}] (128) at (-1.5, 0.25) {};
            \node [style=none] (129) at (-2, 0.625) {};
            \node [style=none] (130) at (-2, -0.125) {};
            \node [style=none] (131) at (-1.475, 0.25) {};
            \node [style=none] (138) at (-1.475, 0.25) {};
            \node [style=label] (139) at (-2.475, 0) {$X$};
            \node [style=label] (140) at (-2.475, 0.625) {$X$};
            \node [style=label] (142) at (-2.475, -0.6) {$X$};
            \node [{dotStyle/-}] (143) at (-0.75, -0.15) {};
            \node [style=none] (144) at (-1.5, -0.6) {};
            \node [style=none] (145) at (-1.5, 0.25) {};
            \node [style={dotStyle/-}] (146) at (-0.325, -0.15) {};
            \node [style=none] (147) at (-2, -0.6) {};
            \node [style=none] (165) at (-2, -0.125) {};
            \node [style=none] (166) at (-2, -0.6) {};
            \node [style=none] (167) at (-2, 0.625) {};
        \end{pgfonlayer}
        \begin{pgfonlayer}{edgelayer}
            \draw [{bgStyle/-}] (124.center)
                 to (123.center)
                 to (122.center)
                 to (125.center)
                 to cycle;
            \draw [{wStyle/-}, bend right] (130.center) to (128);
            \draw [{wStyle/-}, bend right] (128) to (129.center);
            \draw [{wStyle/-}] (128) to (131.center);
            \draw [style={wStyle/-}] (138.center) to (131.center);
            \draw [{wStyle/-}, bend left] (145.center) to (143);
            \draw [{wStyle/-}, bend left] (143) to (144.center);
            \draw [{wStyle/-}] (143) to (146);
            \draw [{wStyle/-}] (147.center) to (144.center);
        \end{pgfonlayer}
    \end{tikzpicture}            
    \stackrel{\text{\Cref{cor:cc are the same}}}{=}
    \begin{tikzpicture}
        \begin{pgfonlayer}{nodelayer}
            \node [style=none] (109) at (-1.75, 0.625) {};
            \node [style=none] (122) at (-2, -1) {};
            \node [style=none] (123) at (-2, 1) {};
            \node [style=none] (124) at (0.75, 1) {};
            \node [style=none] (125) at (0.75, -1) {};
            \node [{dotStyle/-}] (128) at (-1.25, 0.25) {};
            \node [style=none] (129) at (-1.75, 0.625) {};
            \node [style=none] (130) at (-1.75, -0.125) {};
            \node [style=none] (131) at (-0.825, 0.25) {};
            \node [style=none] (133) at (-1.925, -0.125) {};
            \node [style=none] (134) at (-1.75, -0.25) {};
            \node [style=none] (135) at (-1.75, 0.75) {};
            \node [style=none] (136) at (-0.825, 0.75) {};
            \node [style=none] (137) at (-0.825, -0.25) {};
            \node [style=none] (138) at (-0.825, 0.25) {};
            \node [style=label] (139) at (-2.475, 0) {$X$};
            \node [style=label] (140) at (-2.475, 0.625) {$X$};
            \node [style=label] (142) at (-2.475, -0.6) {$X$};
            \node [{dotStyle/+}] (143) at (-0.1, -0.15) {};
            \node [style=none] (144) at (-0.85, -0.6) {};
            \node [style=none] (145) at (-0.85, 0.25) {};
            \node [style={dotStyle/+}] (146) at (0.325, -0.15) {};
            \node [style=none] (147) at (-1.925, -0.6) {};
            \node [style=none] (165) at (-2, -0.125) {};
            \node [style=none] (166) at (-2, -0.6) {};
            \node [style=none] (167) at (-2, 0.625) {};
        \end{pgfonlayer}
        \begin{pgfonlayer}{edgelayer}
            \draw [{bgStyle/+}] (124.center)
                 to (123.center)
                 to (122.center)
                 to (125.center)
                 to cycle;
            \draw [{bgStyle/-}] (136.center)
                 to (135.center)
                 to (134.center)
                 to (137.center)
                 to cycle;
            \draw [{wStyle/-}, bend right] (130.center) to (128);
            \draw [{wStyle/-}, bend right] (128) to (129.center);
            \draw [{wStyle/-}] (128) to (131.center);
            \draw [{wStyle/+}] (130.center) to (133.center);
            \draw [style={wStyle/+}] (138.center) to (131.center);
            \draw [{wStyle/+}, bend left] (145.center) to (143);
            \draw [{wStyle/+}, bend left] (143) to (144.center);
            \draw [{wStyle/+}] (143) to (146);
            \draw [{wStyle/+}] (147.center) to (144.center);
            \draw [style={wStyle/+}] (138.center) to (145.center);
            \draw (167.center) to (129.center);
            \draw (133.center) to (165.center);
            \draw (166.center) to (147.center);
        \end{pgfonlayer}
    \end{tikzpicture}    
\end{align*}

%% file: tikz/linFrob/proof.tex
\begin{align*}
    \begin{tikzpicture}
        \begin{pgfonlayer}{nodelayer}
            \node [{dotStyle/+}] (107) at (0.5, -0.375) {};
            \node [style=none] (108) at (0, -0.75) {};
            \node [style=none] (109) at (0, 0) {};
            \node [style=none] (122) at (0.925, 1.125) {};
            \node [style=none] (123) at (0.925, -1.125) {};
            \node [style=none] (124) at (-1.175, -1.125) {};
            \node [style=none] (125) at (-1.175, 1.125) {};
            \node [style=none] (127) at (0.925, -0.375) {};
            \node [{dotStyle/-}] (128) at (-0.5, 0.375) {};
            \node [style=none] (129) at (0, 0) {};
            \node [style=none] (130) at (0, 0.75) {};
            \node [style=none] (131) at (-0.925, 0.375) {};
            \node [style=none] (132) at (-1.175, -0.75) {};
            \node [style=none] (133) at (0.925, 0.75) {};
            \node [style=none] (134) at (0, 0.875) {};
            \node [style=none] (135) at (0, -0.125) {};
            \node [style=none] (136) at (-0.925, -0.125) {};
            \node [style=none] (137) at (-0.925, 0.875) {};
            \node [style=none] (138) at (-1.175, 0.375) {};
            \node [style=label] (139) at (1.4, -0.375) {$X$};
            \node [style=label] (140) at (1.4, 0.75) {$X$};
            \node [style=label] (141) at (-1.65, -0.75) {$X$};
            \node [style=label] (142) at (-1.65, 0.375) {$X$};
        \end{pgfonlayer}
        \begin{pgfonlayer}{edgelayer}
            \draw [{bgStyle/+}] (124.center)
                 to (123.center)
                 to (122.center)
                 to (125.center)
                 to cycle;
            \draw [{bgStyle/-}] (136.center)
                 to (135.center)
                 to (134.center)
                 to (137.center)
                 to cycle;
            \draw [{wStyle/+}, bend left] (109.center) to (107);
            \draw [{wStyle/+}, bend left] (107) to (108.center);
            \draw [{wStyle/+}] (107) to (127.center);
            \draw [{wStyle/-}, bend right] (130.center) to (128);
            \draw [{wStyle/-}, bend right] (128) to (129.center);
            \draw [{wStyle/-}] (128) to (131.center);
            \draw [{wStyle/+}] (132.center) to (108.center);
            \draw [{wStyle/+}] (130.center) to (133.center);
            \draw [style={wStyle/+}] (138.center) to (131.center);
        \end{pgfonlayer}
    \end{tikzpicture}
    &\stackrel{\text{\Cref{th:spider}}}{=}
    \begin{tikzpicture}
        \begin{pgfonlayer}{nodelayer}
            \node [style=none] (122) at (-2.275, -1.975) {};
            \node [style=none] (123) at (-2.275, 2) {};
            \node [style=none] (124) at (4.075, 2) {};
            \node [style=none] (125) at (4.075, -2) {};
            \node [{dotStyle/-}] (128) at (0.3, 0) {};
            \node [style=none] (129) at (-0.2, 0.375) {};
            \node [style=none] (130) at (-0.2, -0.375) {};
            \node [style=none] (134) at (-1.975, -1.175) {};
            \node [style=none] (135) at (-1.975, 1.75) {};
            \node [style=none] (136) at (1.825, 1.75) {};
            \node [style=none] (137) at (1.825, -1.175) {};
            \node [style=label] (140) at (4.525, 1.475) {$X$};
            \node [style=label] (142) at (-2.75, -0.85) {$X$};
            \node [{dotStyle/-}] (143) at (1.05, -0.4) {};
            \node [style=none] (144) at (0.3, -0.85) {};
            \node [style=none] (145) at (0.3, 0) {};
            \node [style={dotStyle/-}] (146) at (1.475, -0.4) {};
            \node [style=none] (147) at (-1.975, -0.85) {};
            \node [style=none] (166) at (-2.275, -0.85) {};
            \node [{dotStyle/-}] (170) at (-1.15, 0.975) {};
            \node [style=none] (171) at (-0.4, 0.375) {};
            \node [style=none] (172) at (-0.4, 1.475) {};
            \node [style={dotStyle/-}] (173) at (-1.575, 0.975) {};
            \node [{dotStyle/-}] (174) at (-1.15, 0.225) {};
            \node [style=none] (175) at (-0.4, -0.375) {};
            \node [style=none] (176) at (-0.4, 0.725) {};
            \node [style={dotStyle/-}] (177) at (-1.575, 0.225) {};
            \node [style=none] (178) at (1.825, 1.475) {};
            \node [style=none] (179) at (1.825, 0.725) {};
            \node [style=none] (180) at (4.075, 1.475) {};
            \node [style=none] (181) at (1.825, 0.725) {};
            \node [style=label] (182) at (-2.75, -1.575) {$X$};
            \node [style=none] (183) at (-2.275, -1.575) {};
            \node [style={dotStyle/+}] (184) at (3.25, -0.45) {};
            \node [style=none] (185) at (4.075, -0.45) {};
            \node [style=none] (186) at (1.825, -1.575) {};
            \node [style=label] (187) at (4.525, -0.45) {$X$};
        \end{pgfonlayer}
        \begin{pgfonlayer}{edgelayer}
            \draw [{bgStyle/+}] (124.center)
                 to (123.center)
                 to (122.center)
                 to (125.center)
                 to cycle;
            \draw [{bgStyle/-}] (136.center)
                 to (135.center)
                 to (134.center)
                 to (137.center)
                 to cycle;
            \draw [{wStyle/-}, bend right] (130.center) to (128);
            \draw [{wStyle/-}, bend right] (128) to (129.center);
            \draw [{wStyle/-}, bend left] (145.center) to (143);
            \draw [{wStyle/-}, bend left] (143) to (144.center);
            \draw [{wStyle/-}] (143) to (146);
            \draw [style={wStyle/-}] (144.center) to (147.center);
            \draw [style={wStyle/-}] (145.center) to (128);
            \draw [{wStyle/-}, bend right] (172.center) to (170);
            \draw [{wStyle/-}, bend right] (170) to (171.center);
            \draw [{wStyle/-}] (170) to (173);
            \draw [{wStyle/-}, bend right] (176.center) to (174);
            \draw [{wStyle/-}, bend right] (174) to (175.center);
            \draw [{wStyle/-}] (174) to (177);
            \draw [style={wStyle/-}] (130.center) to (175.center);
            \draw [style={wStyle/-}] (129.center) to (171.center);
            \draw [style={wStyle/-}] (172.center) to (178.center);
            \draw [style={wStyle/-}] (179.center) to (176.center);
            \draw [style={wStyle/+}] (147.center) to (166.center);
            \draw [style={wStyle/+}] (179.center) to (181.center);
            \draw [style={wStyle/+}] (178.center) to (180.center);
            \draw [style={wStyle/+}] (183.center) to (186.center);
            \draw [style={wStyle/+}, bend right] (186.center) to (184);
            \draw [style={wStyle/+}, bend left] (181.center) to (184);
            \draw [style={wStyle/+}] (184) to (185.center);
        \end{pgfonlayer}
    \end{tikzpicture}
    \stackrel{\text{\Cref{cor:daggers are the same}}}{=}
    \begin{tikzpicture}
        \begin{pgfonlayer}{nodelayer}
            \node [style=none] (122) at (-2.275, -1.975) {};
            \node [style=none] (123) at (-2.275, 2) {};
            \node [style=none] (124) at (3.375, 2) {};
            \node [style=none] (125) at (3.375, -2) {};
            \node [{dotStyle/-}] (128) at (-0.15, 0) {};
            \node [style=none] (129) at (-0.65, 0.375) {};
            \node [style=none] (130) at (-0.65, -0.375) {};
            \node [style=none] (134) at (-0.65, -0.5) {};
            \node [style=none] (135) at (-0.65, 0.5) {};
            \node [style=none] (136) at (0.375, 0.5) {};
            \node [style=none] (137) at (0.375, -0.5) {};
            \node [style=label] (140) at (3.825, 1.475) {$X$};
            \node [style=label] (142) at (-2.75, -0.85) {$X$};
            \node [{dotStyle/+}] (143) at (1.1, -0.4) {};
            \node [style=none] (144) at (0.35, -0.85) {};
            \node [style=none] (145) at (0.35, 0) {};
            \node [style={dotStyle/+}] (146) at (1.525, -0.4) {};
            \node [style=none] (166) at (-2.275, -0.85) {};
            \node [{dotStyle/+}] (170) at (-1.4, 0.975) {};
            \node [style=none] (171) at (-0.65, 0.375) {};
            \node [style=none] (172) at (-0.65, 1.475) {};
            \node [style={dotStyle/+}] (173) at (-1.825, 0.975) {};
            \node [{dotStyle/+}] (174) at (-1.4, 0.225) {};
            \node [style=none] (175) at (-0.65, -0.375) {};
            \node [style=none] (176) at (-0.65, 0.725) {};
            \node [style={dotStyle/+}] (177) at (-1.825, 0.225) {};
            \node [style=none] (179) at (1.125, 0.725) {};
            \node [style=none] (180) at (3.375, 1.475) {};
            \node [style=none] (181) at (1.125, 0.725) {};
            \node [style=label] (182) at (-2.75, -1.575) {$X$};
            \node [style=none] (183) at (-2.275, -1.575) {};
            \node [style={dotStyle/+}] (184) at (2.55, -0.45) {};
            \node [style=none] (185) at (3.375, -0.45) {};
            \node [style=none] (186) at (1.125, -1.575) {};
            \node [style=label] (187) at (3.825, -0.45) {$X$};
        \end{pgfonlayer}
        \begin{pgfonlayer}{edgelayer}
            \draw [{bgStyle/+}] (124.center)
                 to (123.center)
                 to (122.center)
                 to (125.center)
                 to cycle;
            \draw [{bgStyle/-}] (136.center)
                 to (135.center)
                 to (134.center)
                 to (137.center)
                 to cycle;
            \draw [{wStyle/-}, bend right] (130.center) to (128);
            \draw [{wStyle/-}, bend right] (128) to (129.center);
            \draw [{wStyle/+}, bend left] (145.center) to (143);
            \draw [{wStyle/+}, bend left] (143) to (144.center);
            \draw [style={wStyle/-}] (145.center) to (128);
            \draw [{wStyle/+}, bend right] (172.center) to (170);
            \draw [{wStyle/+}, bend right] (170) to (171.center);
            \draw [{wStyle/+}, bend right] (176.center) to (174);
            \draw [{wStyle/+}, bend right] (174) to (175.center);
            \draw [style={wStyle/-}] (130.center) to (175.center);
            \draw [style={wStyle/-}] (129.center) to (171.center);
            \draw [style={wStyle/+}] (179.center) to (181.center);
            \draw [style={wStyle/+}] (183.center) to (186.center);
            \draw [style={wStyle/+}, bend right] (186.center) to (184);
            \draw [style={wStyle/+}, bend left] (181.center) to (184);
            \draw [style={wStyle/+}] (184) to (185.center);
            \draw [style={wStyle/+}] (172.center) to (180.center);
            \draw [style={wStyle/+}] (181.center) to (176.center);
            \draw [style={wStyle/+}] (166.center) to (144.center);
            \draw [style={wStyle/+}] (146) to (143);
            \draw [style={wStyle/+}] (173) to (170);
            \draw [style={wStyle/+}] (177) to (174);
        \end{pgfonlayer}
    \end{tikzpicture}
    \\
    &\stackrel{\text{\Cref{th:spider}}}{=}
    \begin{tikzpicture}
        \begin{pgfonlayer}{nodelayer}
            \node [style=none] (122) at (-2.575, -1.95) {};
            \node [style=none] (123) at (-2.575, 2) {};
            \node [style=none] (124) at (2.025, 2) {};
            \node [style=none] (125) at (2.025, -1.975) {};
            \node [{dotStyle/-}] (128) at (-0.15, 0) {};
            \node [style=none] (129) at (-0.65, 0.375) {};
            \node [style=none] (130) at (-0.65, -0.375) {};
            \node [style=none] (134) at (-0.65, -0.5) {};
            \node [style=none] (135) at (-0.65, 0.5) {};
            \node [style=none] (136) at (0.375, 0.5) {};
            \node [style=none] (137) at (0.375, -0.5) {};
            \node [style=label] (140) at (2.475, 1.475) {$X$};
            \node [style=label] (142) at (-3.05, -0.375) {$X$};
            \node [{dotStyle/+}] (143) at (1.1, -0.4) {};
            \node [style=none] (144) at (0.35, -0.85) {};
            \node [style=none] (145) at (0.35, 0) {};
            \node [style={dotStyle/+}] (146) at (1.525, -0.4) {};
            \node [style=none] (166) at (-0.675, -0.85) {};
            \node [{dotStyle/+}] (170) at (-1.4, 0.975) {};
            \node [style=none] (171) at (-0.65, 0.375) {};
            \node [style=none] (172) at (-0.65, 1.475) {};
            \node [style={dotStyle/+}] (173) at (-1.825, 0.975) {};
            \node [{dotStyle/+}] (174) at (-1.9, -1.125) {};
            \node [style=none] (175) at (-1.4, -1.45) {};
            \node [style=none] (176) at (-1.4, -0.85) {};
            \node [style=none] (177) at (-2.575, -1.125) {};
            \node [style=none] (180) at (2.025, 1.475) {};
            \node [style=label] (182) at (-3.05, -1.1) {$X$};
            \node [style=none] (185) at (-2.575, -0.375) {};
            \node [style=none] (186) at (2.025, -1.475) {};
            \node [style=label] (187) at (2.475, -1.45) {$X$};
            \node [style=none] (188) at (-1.4, -0.375) {};
            \node [style=none] (189) at (-1.4, -0.85) {};
        \end{pgfonlayer}
        \begin{pgfonlayer}{edgelayer}
            \draw [{bgStyle/+}] (124.center)
                 to (123.center)
                 to (122.center)
                 to (125.center)
                 to cycle;
            \draw [{bgStyle/-}] (136.center)
                 to (135.center)
                 to (134.center)
                 to (137.center)
                 to cycle;
            \draw [{wStyle/-}, bend right] (130.center) to (128);
            \draw [{wStyle/-}, bend right] (128) to (129.center);
            \draw [{wStyle/+}, bend left] (145.center) to (143);
            \draw [{wStyle/+}, bend left] (143) to (144.center);
            \draw [style={wStyle/-}] (145.center) to (128);
            \draw [{wStyle/+}, bend right] (172.center) to (170);
            \draw [{wStyle/+}, bend right] (170) to (171.center);
            \draw [{wStyle/+}, bend right] (176.center) to (174);
            \draw [{wStyle/+}, bend right] (174) to (175.center);
            \draw [style={wStyle/-}] (129.center) to (171.center);
            \draw [style={wStyle/+}] (172.center) to (180.center);
            \draw [style={wStyle/+}] (166.center) to (144.center);
            \draw [style={wStyle/+}] (146) to (143);
            \draw [style={wStyle/+}] (173) to (170);
            \draw [style={wStyle/+}] (177.center) to (174);
            \draw [style={wStyle/+}, in=0, out=-180] (130.center) to (189.center);
            \draw [style={wStyle/+}, in=-180, out=0] (188.center) to (166.center);
            \draw [style={wStyle/+}] (176.center) to (189.center);
            \draw [style={wStyle/+}] (175.center) to (186.center);
            \draw [style={wStyle/+}] (185.center) to (188.center);
        \end{pgfonlayer}
    \end{tikzpicture}  
    \stackrel{\text{\Cref{prop:monoid rewiring}}}{=}
    \begin{tikzpicture}
        \begin{pgfonlayer}{nodelayer}
            \node [{dotStyle/+}] (107) at (-0.5, -0.375) {};
            \node [style=none] (108) at (0, -0.75) {};
            \node [style=none] (109) at (0, 0) {};
            \node [style=none] (122) at (-0.925, 1.125) {};
            \node [style=none] (123) at (-0.925, -1.125) {};
            \node [style=none] (124) at (1.175, -1.125) {};
            \node [style=none] (125) at (1.175, 1.125) {};
            \node [style=none] (127) at (-0.925, -0.375) {};
            \node [{dotStyle/-}] (128) at (0.5, 0.375) {};
            \node [style=none] (129) at (0, 0) {};
            \node [style=none] (130) at (0, 0.75) {};
            \node [style=none] (131) at (0.925, 0.375) {};
            \node [style=none] (132) at (1.175, -0.75) {};
            \node [style=none] (133) at (-0.925, 0.75) {};
            \node [style=none] (134) at (0, 0.875) {};
            \node [style=none] (135) at (0, -0.125) {};
            \node [style=none] (136) at (0.925, -0.125) {};
            \node [style=none] (137) at (0.925, 0.875) {};
            \node [style=none] (138) at (1.175, 0.375) {};
            \node [style=label] (139) at (-1.4, -0.375) {$X$};
            \node [style=label] (140) at (-1.4, 0.75) {$X$};
            \node [style=label] (141) at (1.65, -0.75) {$X$};
            \node [style=label] (142) at (1.65, 0.375) {$X$};
        \end{pgfonlayer}
        \begin{pgfonlayer}{edgelayer}
            \draw [{bgStyle/+}] (124.center)
                 to (123.center)
                 to (122.center)
                 to (125.center)
                 to cycle;
            \draw [{bgStyle/-}] (136.center)
                 to (135.center)
                 to (134.center)
                 to (137.center)
                 to cycle;
            \draw [{wStyle/+}, bend right] (109.center) to (107);
            \draw [{wStyle/+}, bend right] (107) to (108.center);
            \draw [{wStyle/+}] (107) to (127.center);
            \draw [{wStyle/-}, bend left] (130.center) to (128);
            \draw [{wStyle/-}, bend left] (128) to (129.center);
            \draw [{wStyle/-}] (128) to (131.center);
            \draw [{wStyle/+}] (132.center) to (108.center);
            \draw [{wStyle/+}] (130.center) to (133.center);
            \draw [style={wStyle/+}] (138.center) to (131.center);
        \end{pgfonlayer}
    \end{tikzpicture}                    
\end{align*}

%% file: sections/apprestrictingadjunction.tex
\section{Appendix to \cref{sec:restrictingadjunction}}\label{app:resadj}

\begin{lemma}\label{ex_func_ent_rel_in_CB}
                Let $\Cat{C}$ be a cartesian bicategory, and consider the doctrine $\doctrine{\mathsf{Map}(\Cat{C})}{\Cat{C}[-,I]}$. Then an element $c\in \Cat{C}[X\times Y,I]$ is:
                \begin{itemize}
                    \item functional if and only if the corresponding arrow in $\Cat{C}[X, Y]$, i.e. $\rewiredCirc[+]{c}[X][Y]$, satisfies the leftmost inequality in \eqref{eq:def:map};
                    \item entire if and only if the corresponding arrow in $\Cat{C}[X, Y]$, i.e. $\rewiredCirc[+]{c}[X][Y]$, satisfies the rightmost inequality in \eqref{eq:def:map};
                                    \end{itemize}
            Therefore,  $c$ is both functional and entire if and only if $\rewiredCirc[+]{c}[X][Y]$ is a map.
            
            \end{lemma}
\begin{proof}
    By definition, an element $c\in \Cat{C}[X\times Y,I]$ is functional from $X$ to $Y$ if and only if 
    \begin{equation}\label{functional}\tag{functional}
        \functionalLHS{c}{X}{Y} \leq \functionalRHS{X}{Y} %
    \end{equation}
    and it is entire from $X$ to $Y$ if and only if 
    \begin{equation}\label{entire}\tag{entire}
        \entireLHS{X} \leq \entireRHS{c}{X} %
    \end{equation}
    Now observe that \eqref{functional} holds if and only if the following holds
    \begin{equation}\label{functionalalt}\tag{functional-alt}
        \functionalLHS{c}{X}{Y} \leq \functionalRHSalt{c}{X}{Y} %
    \end{equation}
    To conclude, we need to show that \eqref{functionalalt} and \eqref{entire} are equivalent, respectively, to
    \[
        \rewiredFunctionalLHS{c}{X}{Y} \leq \rewiredFunctionalRHS{c}{X}{Y} \quad\text{ and }\quad \entireLHS{X} \leq \rewiredEntireRHS{c}{X}.
    \]
    For \eqref{functionalalt} we use the fact that in any cartesian bicategory $\myCirc{a}{X}{Y} \leq \myCirc{b}{X}{Y}$ if and only if $\rewiredCirc[+]{a}[X][Y] \leq \rewiredCirc[+]{b}[X][Y]$, and thus
    \[
        \functionalLHS{c}{X}{Y} \leq \functionalRHSalt{c}{X}{Y} \;\text{ iff }\; \functionalLHScupped{c}{X}{Y} \leq \functionalRHScupped{c}{X}{Y}
    \]
    which amounts to $\rewiredFunctionalLHS{c}{X}{Y} \leq \rewiredFunctionalRHS{c}{X}{Y}$ using~\Cref{th:spider}.
    
    For \eqref{entire} it suffices to observe that $\rewiredEntireRHS{c}{X} \stackrel{\eqref{ax:comPlusUnit}}{=} \entireRHS{c}{X}$.
\end{proof}

\begin{proof}[Proof of \cref{prop_maps_rel_P}]
By \Cref{adj_CB_EED} we have that for every elementary and existential doctrine $P$, the unit of the adjunction $\REL(-)\dashv \HM(-)$ is morphism of elementary and existential doctrines $\freccia{P}{\eta_P\defeq (\Gamma_P,\rho)}{\HM(\REL(P))}$. Moreover, each component of $\rho$ is an iso by definition (see \Cref{app_section_adj} for a detailed description).  Therefore,  $\rho$ preserves and reflects functional and entire relations  (see \Cref{rem_EED_morph_sens_Fe_and_Te_into_Fe_and_Te}), i.e. $\rho_{X\times Y}(\phi)\in \REL(P)[X\times Y,I]$ is functional and entire in $\REL(P)$ if and only if $\phi$ is functional and entire in $P$. On the other hand, by \Cref{ex_func_ent_rel_in_CB} we have that an arrow $\psi\in \REL(P)[X,Y]$, i.e. an element $\psi\in P(X\times Y)$, is a map if and only if the corresponding element in $\REL(P)[A\times B,I]$, i.e. $\rho_{X\times Y} (\psi)$, is functional and entire with respect to the doctrine $\REL(P)[-,I]$. Therefore, we can conclude that $\phi$ is a map in $\mathsf{Rel}(P)$ if and only if $\phi$ is an entire and functional element of $ P$.
\end{proof}

\begin{proof}[Proof of \cref{thm:main}]
First, we want to prove that the inclusion $\HM \colon \CartBic \hookrightarrow \EED$ in \eqref{adj_CB_EED} restricts to an inclusion of categories 
$\PeirceBic\hookrightarrow\BHD$. By \cref{prop_inclusion_PD_BHD}, one only needs to check for morphisms in $\PeirceBic$. Given a morphism of peircean bicategories $F\colon \Cat{C} \to \Cat{D}$, $\HM(F)$ is the morphism of elementary and existential doctrines $(\tilde{F},\mathfrak{b}^F)$ defined in \cref{sec:adjunction}. In order to conclude that it is a morphism of boolean doctrines, it is enough to show that $\mathfrak{b}^F_X$ is a morphism of boolean algebras for all objects $X$. Since  $(\tilde{F},\mathfrak{b}^F)$ is a morphism of doctrines, $\mathfrak{b}^F_X$ is a morphism of inf-semilattices. Thus it is enough to show that $\mathfrak{b}^F_X$ preserve negation. But this is trivial since, for all  $c\in \Cat{C}[X,I]$, 
\begin{align}
\mathfrak{b}^F_X(\nega{c}) & =F(\nega{c}) \tag{Def. $\mathfrak{b}^F$} \\
&= \nega{F(c)}  \tag{morphism of Peircean, \cref{def:peircean-bicategory}} \\
&=\nega{\mathfrak{b}^F_X(c)}  \tag{Def. $\mathfrak{b}^F$}
\end{align}

Now, to prove that $\REL$ restrict to a functor $\REL \colon \BHD \to \PeirceBic$, by \cref{thm_adj_BHD_PB}, one only needs to check that for all morphisms of boolean hyperdoctrines $(F,\mathfrak{b})\colon P \to Q$, $\REL(F,\mathfrak{b})\colon \REL(P) \to \REL(Q)$ is a morphism of peicean bicategories. Since by \eqref{adj_CB_EED}, $\REL(F,\mathfrak{b})$ is a morphism of cartesian bicategories, one only needs to check that it preserves the negation. But this is obvious since for all arrows  $\phi \in \REL(P)[X,Y]$, $\REL(F,\mathfrak{b})(\phi)$ is --by definition-- $\mathfrak{b}_{X\times Y}(\phi)$ and $\mathfrak{b}_{X\times Y}$ is a morphism of boolean algebras.

To conclude, one only needs to check the unit and the counit of the adjunction in \eqref{adj_CB_EED}. The counit is an isomorphism of cartesian bicategories (see Equation (9) in \cite{bonchi2021doctrines}), and then it provides an isomorphism of peircean bicategories $\Cat{C}\cong \REL(\Cat{C}[-,I])$ whenever $\Cat{C}$ is a peircean bicategory.
The unit of the adjunction $\eta_P\colon P\to \REL(P)[-,I]$ is the morphism of elementary and existential doctrines $(\Gamma_P, \rho)$ illustrated in \eqref{eq:unit}. To conclude that $\eta_P$ is a morphism of boolean hyperdoctrine whenever $P$ is a boolean hyperdoctrine, one has only to prove that $\rho$ is a morphism of boolean algebras, but this is trivial since $\rho$ is always an isomorphism of inf-semilattices.

\end{proof}

%% file: sections/apptabulation.tex
\section{Appendix to \Cref{sec:the two equivalences}}\label{app_section_tablutation_and equivalences}

\begin{proof}[Proof of \cref{thm:theequivalence}]
    By \Cref{eq:theequivalence} we have that  the $\HM$ and $\REL$ functors provide an equivalence between the categories $\CartBic$ and  $\overline{\EED}$. Now, since every peircean category is in particular a cartesian bicategory, we have that every boolean hyperdoctrine arising from a peircean bicategory satisfies (RUC) and it has comprehensive diagonals. Then, we have that the functor $\HM\colon \PeirceBic \hookrightarrow \BHD$ factors through the canonical inclusion $\overline{\BHD}\hookrightarrow \BHD$:
    
    \[\begin{tikzcd}
        \PeirceBic && \BHD \\
        & \overline{\BHD}
        \arrow["\HM",hook, from=1-1, to=1-3]
        \arrow["\HM"',hook, from=1-1, to=2-2]
        \arrow[hook, from=2-2, to=1-3]
    \end{tikzcd}\]
    By \Cref{thm:main}, we have that $\HM\colon \PeirceBic \hookrightarrow \BHD$ is fully and faithful (since the counit of the adjunction is an iso), so it remains to prove that it is essentially surjective (with respect to the objects of $\overline{\BHD}$). By the equivalence presented in \Cref{eq:theequivalence}, we know that every boolean hyperdoctrine (that is in particular an elementary and existential doctrine) satisfying (RUC) and having comprehensive diagonals,  is isomorphic to an elementary and existential doctrine $\doctrine{\map(\Cat{C})}{\Cat{C}[-,I]}$ for some cartesian bicategory $\Cat{C}$. Thus,  we can conclude that $\doctrine{\map(\Cat{C})}{\Cat{C}[-,I]}$ is a boolean hyperdoctrine and, by \Cref{cor_cart_bic_is_peircea_iff_hyperdoc}, that $\Cat{C}$ is a peircean bicategory. This concludes  the proof that  $\PeirceBic \equiv \overline{\BHD}$.
    
    \end{proof}
\section{Appendix to \Cref{sec:tabulation}}
We summarize some useful properties of tabulations, which are crucial to establish the precise connection with the notion of full comprehension. We refer to  \cite[Lem. 3.3]{carboni1987cartesian} for the following result:
\begin{lemma}\label{lem_carboni_tabulations}
    Let $\freccia{X_r}{i}{X}$ be a tabulation of $\freccia{X}{r}{I}$. Then
    \begin{itemize}
        \item for every map $\freccia{Z}{f}{X}$ such that $f^\dag \seq[+]\discard[+]_{Z}\leq r$ there exists a unique map $\freccia{Z}{h}{X_r}$ such that $f=h\seq[+]i$;
        \item if $f^\dagger \seq[+]\discard[+]_{Z}=r$, then $h^\dagger \seq[+]\discard[+]_{Z}=\discard[+]_{X}$.
    \end{itemize}
       
\end{lemma}

Now we recall another useful lemma, regarding doctrines with full comprehensions. We refer to \cite[Pro. 7.10]{MaiettiTrotta21} for the following result:
\begin{lemma}\label{lem_a_is_exists_comp_top}
    Let $\doctrine{\Cat{C}}{P}$ be a doctrine with full comprehensions. Then
    every comprehension is a mono for every element $\alpha$ of $P(X)$ we have that $\alpha=\exists_{\comp{\alpha}}(\top_{X_{\alpha}})$.
\end{lemma}

\begin{proof}[Proof of \cref{thm_func_comp_cb_iff_comprehensions}]
    Let $(\Cat{C}, \copier[+], \cocopier[+])$ be functionally complete, and let us consider the doctrine $\doctrine{\map(\Cat{C})}{\Cat{C}[-,I]}$ and an element $\freccia{X}{r}{I}$ of $\Cat{C}[X,I]$. We claim that the tabulation $\freccia{X_r}{i}{X}$ of $r$ is the comprehension of $r$. 
    
    First, notice that $\Cat{C}[i,I](r)=\top_{X_r}$ (i.e. $i\seq[+]r=\discard[+]_{X_r}$) because, by definition of functionally completeness, we have that $i\seq[+]i^\dagger=\id_{X_r}$ and $ i^\dagger \seq[+]\discard[+]_{X_r}=r$, and hence $i \seq[+]r=\discard[+]_{X_r}$. Now suppose that $\freccia{Z}{f}{X}$ is a map such that $\Cat{C}[f,I](r)=\top_{Z}$ (i.e. $f\seq[+]r =\discard[+]_{Z}$).  Then, in particular, we have that $\top_{Z}\leq \Cat{C}[f,I](r)$, and we can conclude that  $\exists_f(\top_{X_r})\leq r$ because the doctrine $\doctrine{\mathsf{Map}(\Cat{C})}{\Cat{C}[-,I]}$ is elementary and existential, so it has left adjoints along all the arrows (see\Cref{rem:generalised adjoint hyperdoctrine}). Now, by definition of left ajoints for this doctrine (see \Cref{example_cartbicat_provide_eed}) this means that $f^\dagger\seq[+]\discard[+]_{Z}\leq r$, so we can apply \Cref{lem_carboni_tabulations} and conclude that there exists a unique $h$ such that $f=h\seq[+]i$. This concludes the proof that if $(\Cat{C}, \copier[+], \cocopier[+])$ is functionally complete then $\doctrine{\mathsf{Map}(\Cat{C})}{\Cat{C}[-,I]}$ has comprehensions. 
    
    Now we need to prove that comprehensions are full. So let $\freccia{X}{r,r'}{I}$ be two arrows and let $i$ and $i'$ their tabulation. If $i$ factors through $i'$, namely $i=g \seq[+]i'$ for some map $\freccia{X_r}{g}{X_{r'}}$, then we have that $i^\dagger=(i')^\dagger \seq[+]g^\dagger$. Hence, we can conclude that $r\leq r'$ because,  by hypothesis,  $ i^\dagger\seq[+]\discard[+]_{X_r}=r$, and using the fact that $i^\dagger=(i')^\dagger \seq[+]g^\dagger$, we can conclude that $r=(i')^\dagger \seq[+]g^\dagger \seq[+] \discard[+]_{X_r}\leq r'$ (because $ (i')^\dagger \seq[+]\discard[+]_{X_{r'}}=r'$ and $g^\dagger\seq[+]\discard[+]_{X_{r}} \leq \discard[+]_{X_{r'}}$). 
    
    Now we show that full comprehensions implies functionally completeness. So, let us consider an arrow $\freccia{X}{r}{I}$. We claim that the comprehension $\freccia{X_r}{\comp{r}}{X}$ is a tabulation of $r$. First, by \Cref{lem_a_is_exists_comp_top}, we have that $\exists_{\comp{r}}(\top_{X_r})=r$, namely that  $\comp{r}^\dagger \seq[+]\discard[+]_{X_{r}}=r$. Finally, we have that  $\comp{r} \seq[+]\comp{r}^\dagger =\id_{X_r}$ because comprehensions are monomorphisms in $\mathsf{Map}(\Cat{C})$. This concludes the proof that $\comp{r}$ is the tabulation of $r$.
    \end{proof}

The last part of this section is devoted to prove the final corollary, namely that
\[\PeirceBic \equiv\overline{\mathbb{BHD}}_c\cong \mathbb{BC} \]

To properly reach this goal, we summarize here in the language of doctrines the main results presented in this work, and some useful characterizations presented in \cite{TECH}. Our result will follow by combining these results.

First, we have the following result presented in  \cite[Thm. 35]{bonchi2021doctrines}:
\begin{theorem}\label{thm_eqv_CB_EED}
    Let $\doctrine{\Cat{C}}{P}$ be an elementary and existential doctrine. Then the following two conditions are equivalent:
    \begin{itemize}
      \item $P$ has comprehensive diagonals and satisfies (RUC);
      \item $\Cat{C}\cong \map(\REL(P))$ and $P\cong \Cat{C}[-,I]$.
    \end{itemize}
In particular, $\CartBic\equiv \overline{\EED}$.
\end{theorem}

Combining the previous result with \Cref{thm_func_comp_cb_iff_comprehensions} we obtain the following corollary, where  we denote by $\overline{\EED}_c$ the full subcategory of $\overline{\EED}$ given by those doctrines of $\overline{\EED}$ having full comprehensions, and by $\CartBic_f$ the full subcategory of $\CartBic$ given by functionally complete cartesian bicategories.
\begin{corollary}\label{cor_func_comp_eqv_doct_ruc_cd_full_comp}
    Let $\doctrine{\Cat{C}}{P}$ be an elementary and existential doctrine. Then the following two conditions are equivalent:
    \begin{itemize}
      \item $P$ has comprehensive diagonals, full comprehension and satisfies (RUC);
      \item $\Cat{C}\cong \map(\REL(P))$, $\Cat{C}$ is functionally complete and $P\cong \Cat{C}[-,I]$.
    \end{itemize}
In particular, $\CartBic_f\equiv \overline{\EED}_c$.
\end{corollary}
Notice that \Cref{thm_eqv_CB_EED} can be seen as a generalization of another result regarding doctrines and regular categories presented in \cite[Prop. 5.3]{TECH}, establishing the precise connection between regular categories and doctrines:
\begin{theorem}
    Let $\doctrine{\Cat{C}}{P}$ be an elementary and existential doctrine. Then the following two conditions are equivalent:
    \begin{itemize}
      \item $P$ has comprehensive diagonals, full comprehensions and satisfies (RUC);
      \item $\Cat{C}$ is regular and $P\cong \Sub_{\Cat{C}}$.
    \end{itemize}
In particular, $\mathbb{REG}\equiv \overline{\EED}_c$.
    
\end{theorem}
Finally, we can combine the previous result with \Cref{cor_func_comp_eqv_doct_ruc_cd_full_comp} obtaining as corollary the well-known equivalence (see \cite{carboni1987cartesian} and \cite{Fong2019RegularAR})  between functionally complete bicategories and regular categories:
\begin{corollary}
    We have the equivalences of categories  $\mathbb{REG}\equiv\overline{\EED}_c \equiv \CartBic_f$.
\end{corollary}

Since, by definition (see \cite[Sec. A1.4, p. 38]{SAE}), a category $\Cat{C}$ is boolean if and only if the subobjects functor on $\Cat{C}$ is a boolean hyperdoctrine, we have the following corollary, that is a particular instance of the previous one. Here we denote the category of boolean categories by $\mathbb{BC}$ and that of  boolean hyperodctrines satisfying (RUC), with full comprehensions and comprehensive diagonals by $\overline{\mathbb{BHD}}_c$:
\begin{corollary}\label{cor_bool_cat_bool_hyper_ruc_diag_full_com}
    We have an equivalence of categories $\mathbb{BC}\equiv\overline{\mathbb{BHD}}_c $.
\end{corollary}
Combining these results we obtain the proof of our finial corollary:
    \begin{proof}[Proof of \cref{cor:final}]
        By \Cref{thm:theequivalence} we have the equivalences $\FOBic\equiv\overline{\mathbb{BHD}} $, so we can combine this result with \Cref{cor_func_comp_eqv_doct_ruc_cd_full_comp} obtaining the equivalence $\FOBic_f\equiv\overline{\mathbb{BHD}}_c $. Therefore, combining this equivalence with \Cref{cor_bool_cat_bool_hyper_ruc_diag_full_com} we obtain 
        \[\FOBic_f\equiv\overline{\mathbb{BHD}}_c \equiv \mathbb{BC}.\]
    \end{proof}